\documentclass[12pt,twoside,reqno]{amsart}

\usepackage[titletoc,toc,page]{appendix}
\usepackage{mathtools}

\usepackage{amssymb,amsmath,amstext,amsthm,amsfonts,xcolor, amsthm,amscd}

\usepackage{dsfont}

\usepackage{textgreek}

\usepackage[ansinew]{inputenc} 
\usepackage{graphicx}
\usepackage[mathscr]{eucal}

\usepackage{hyperref, enumerate}

\newcommand{\R}{\mathbb{R}}
\newcommand{\C}{\mathbb{C}}
\newcommand{\N}{\mathbb{N}}
\newcommand{\Z}{\mathbb{Z}}
\newcommand{\Q}{\mathbb{Q}}
\newcommand{\T}{\mathbb{T}}

\newcommand{\SL}{{\rm SL}}
\newcommand{\GL}{{\rm GL}}
\newcommand{\gl}{{\rm Mat}}

\newcommand{\Mat}{{\rm Mat}}

\newcommand{\diag}{\mbox{diag}}

\newcommand{\DC}{{\rm DC}}

\newcommand{\strip}{\mathscr{A}}

\newcommand{\sabs}[1]{\left| #1 \right|} % smaller absolute value
\newcommand{\abs}[1]{\bigl| #1 \bigr|} % absolute value
\newcommand{\babs}[1]{\Bigl| #1 \Bigr|} % big absolute value
\newcommand{\norm}[1]{\lVert#1\rVert} % norm
\newcommand{\normr}[1]{\lVert#1\rVert_r} % norm
\newcommand{\bnorm}[1]{\Bigl\| #1\Bigr\|} %big norm
\newcommand{\normtwo}[1]{% Peter Grill norm @tex.stackexchange.com
{\left\vert\kern-0.25ex\left\vert\kern-0.25ex\left\vert #1 
    \right\vert\kern-0.25ex\right\vert\kern-0.25ex\right\vert} }

\newcommand{\avg}[1]{\left< #1 \right>} % average
\newcommand{\transl}{{T}} % base dynamics map

\newcommand{\less}{\lesssim}
\newcommand{\more}{\gtrsim}

\newcommand{\ep}{\epsilon} 
 \newcommand{\ka}{\kappa} 
\newcommand{\la}{\lambda}
\newcommand{\ga}{\gamma}

\newcommand{\om}{\omega}

\newcommand{\vpsi}{\vec{\psi}}

\newcommand{\analyticf}[1]{C^{\om}_{r} (\T^{#1}, \R)}

\newcommand{\LE}[1]{L^{({#1})}}  %%%  Lyapunov Exponent Function
\newcommand{\An}[1]{A^{({#1})}}  %%% nth iteration of cocycle A
\newcommand{\Bn}[1]{B^{({#1})}}  %%% nth iteration of cocycle B

\newcommand{\un}[2]{u^{({#1})}_{{#2}}}   %%%

\newcommand{\comp}{^{\complement}}

\newcommand{\Gr}{{\rm Gr}}

\newcommand{\FF}{\mathscr{F}}

\newcommand{\B}{\mathscr{B}}

\newcommand{\submatr}[2]{\left\{{#1}\right\}_{{#2}}}
\newcommand\restr[2]{{% we make the whole thing an ordinary symbol
  \left.\kern-\nulldelimiterspace % automatically resize the bar with \right
  #1 % the function
  \vphantom{\big|} % pretend it's a little taller at normal size
  \right|_{#2} % this is the delimiter
  }}
\theoremstyle{plain}
\newtheorem{theorem}{Theorem}[section]
\newtheorem{proposition}{Proposition}[section]
\newtheorem{corollary}[proposition]{Corollary}
\newtheorem{lemma}[proposition]{Lemma}

\numberwithin{equation}{section}

\theoremstyle{remark}
\newtheorem{remark}{Remark}[section]

\theoremstyle{definition}
\newtheorem{definition}{Definition}[section]

\newcommand{\nzerobar}{\underline{n_0}}

\title[Continuity, positivity, simplicity of Lyapunov exponents]{Continuity, positivity and simplicity \\of the Lyapunov exponents for quasi-periodic cocycles}

\date{}

\begin{document}

\author[P. Duarte]{Pedro Duarte}
\address{Departamento de Matem\'atica and CMAFCIO\\
Faculdade de Ci\^encias\\
Universidade de Lisboa\\
Portugal 
}
\email{pmduarte@fc.ul.pt}

\author[S. Klein]{Silvius Klein}
\address{Departamento de Matem\'atica, Pontif\'icia Universidade Cat\'olica do Rio de Janeiro, Brazil (PUC-Rio)
  }
\email{silviusk@mat.puc-rio.br}

\begin{abstract} An analytic quasi-periodic cocycle is a linear cocycle over a fixed ergodic torus translation of one or several variables, where the fiber action depends analytically on the base point. Consider the space of all such cocycles of any given dimension and endow it with the uniform norm. Assume that the  translation  vector satisfies a generic Diophantine condition. 
We prove large deviation type estimates for the iterates of such cocycles, which, moreover, are stable under small perturbations of the cocycle. 
As a consequence of these uniform estimates, we establish  continuity properties of the Lyapunov exponents regarded as functions on this space of cocycles. 
This result builds upon our previous work on this topic and its proof uses an abstract continuity theorem of the Lyapunov exponents which we derived in a recent monograph. The new feature of this paper is extending the availability of such results to cocycles that are {\em identically singular} (i.e. non-invertible anywhere), in the several variables torus translation setting. This feature is exactly what allows us, through a simple limiting argument, to obtain criteria for the positivity and simplicity of the Lyapunov exponents of such cocycles. Specializing to the family of cocycles corresponding to a block Jacobi operator, we derive consequences on the continuity, positivity and simplicity of its Lyapunov exponents, and on the continuity of its integrated density of states.  

\end{abstract}

\maketitle

%\subjclass{37D25, 37A05, 37A60, 37A50, 31A05, 32A10, 81Q10.}
%\keywords{Lyapunov exponents, quasi-periodic cocycles, large deviations, subharmonic functions, lattice Schr\"odinger operators, Oseledets decomposition, dominated splitting.}

%%%%%%%%%%%%%%%%%%%%%%%%%%%%
%%%%%%%%%%%%%%%%%%%%%%%%%%%%

%%%%%%%%%%%%%%%%%%%%%%%%%%%%%%%%%%%%%%%%%%%%%%%%%%%%%%%%%%%%%%%%
%%%%%%%%%%%%%%%%%%%%%%%%%%%%%%%%%%%%%%%%%%%%%%%%%%%%%%%%%%%%%%%%

\section{Introduction and statements}\label{introduction}
\newcommand{\qpcmat}[1]{C^{\om}_{r} (\T^{#1}, \Mat_m (\R))}
\newcommand{\qpcmatk}[1]{C^{\om}_{r} (\T^{#1}, \Mat_k (\R))}
\newcommand{\qpcmatkm}[1]{C^{\om}_{r} (\T^{#1}, \Mat_ {k \times m}  (\R))}
\newcommand{\qpcmatl}[1]{C^{\om}_{r} (\T^{#1}, \Mat_l (\R))}

\subsection*{Definitions, notations, framework}\label{definitions}

%%  1-(P) Definitions and notations (short way)

In ergodic theory,  a {\em linear cocycle}
is a dynamical system on a vector bundle, which preserves the linear bundle structure and induces a measure preserving dynamical system on the base.  The vector bundle is usually assumed to be
 trivial and the base dynamics to be  an ergodic
measure preserving transformation $T\colon X \to X$ on some probability space $(X,\FF,\mu)$.
Given a measurable function $A\colon X\to \Mat_m(\R)$,
the map $F\colon X\times\R^m \to  X\times\R^m$ defined by
$F(x,v)=(T x, A(x) v)$   is a linear cocycle over  $T$. The i\-te\-ra\-ted maps $F^n$ are given by $F^n (x,v)=(T^n x, A^{(n)}(x) v)$, where for all $x\in X$
and $n\geq 1$,
 $ \An{n}(x) := A(T^{n-1} x) \ldots A(T x) A(x)$.
 
 When the base map $T$ is fixed we  refer to 
the matrix valued function $A\colon X\to \Mat_m(\R)$ as being the 
linear cocycle. 

The repeated {\em Lyapunov exponents} (LE) of the cocycle $A$  are
denoted by $L_1(A)\geq L_2(A)\geq \ldots \geq L_m(A) \ge - \infty$.
By Kingman's ergodic theorem, they are the pointwise $\mu$-almost everywhere and average limits
$$ L_k(A)=\lim _{n\to+\infty} \frac{1}{n} \log s_k(\An{n}(x) ) = \lim _{n\to+\infty}  \int_X \frac{1}{n} \log s_k(\An{n}(x) )  d \mu (x) , $$
 where $s_k(g)$ stands for the $k$-th singular value of a matrix $g\in\Mat_m (\R)$.

In particular, $L_1(A)$, the top Lyapunov exponent   of $A$, is the limit as $n\to+\infty$ of the {\em finite scale top Lyapunov exponents}
 $$ \LE{n}_1 (A) := \int_X \frac{1}{n} \log \norm{\An{n}(x) } d \mu (x) . $$

We say that a LE is {\em simple} when its multiplicity is one, that is, when it is distinct from all the other LE. When all LE of a cocycle $A$ are simple, we say that $A$ has {\em simple Lyapunov spectrum}.

\smallskip

%%  2-(P) Framework of the paper
 A {\em quasi-periodic cocycle} is a linear cocycle over some ergodic torus translation on a finite dimensional torus 
 $\T^d=(\R/\Z)^d$ equipped with the Haar measure (which we denote by $\sabs{ \, \cdot \, }$).
 
In this paper we study {\em analytic} quasi-periodic cocycles, that is, cocycles
in the Banach space $\qpcmat{d}$
 of  all analytic functions $A\colon \T^d\to\Mat_m(\R)$ having a holomorphic, continuous up to the boundary extension to $\strip_r^d = \strip_r \times \ldots \times \strip_r \subset \C^d$, where  we denote by $\strip_r := \{ z \in \C \colon 1-r < \sabs{z} < 1 + r \}$  the annulus of width $2 r$ around the torus $\T$.  
 We endow this space with the uniform norm $\norm{A}_r := \sup_{z\in\strip_r^d} \norm{A (z)}$.

\medskip

The main result of this paper is a {\em uniform large deviation type estimate} for the iterates of any analytic quasi-periodic cocycle with simple top Lyapunov exponent. We refer to such a result as a {\em uniform fiber LDT} estimate. Our method requires a generic arithmetic assumption on the translation vector.

A fiber LDT estimate for a cocycle $A$ over some base dynamics $(X, \mu, T)$ has the form
\begin{equation}\label{intro-ldt-eq}
\mu \,   \big\{  x \in X \colon \abs{ \frac{1}{n} \log \norm{\An{n} (x)}   - \LE{n}_1 (A) }  > \epsilon  \big\}  < \iota (n, \epsilon) ,
\end{equation}
where $\ep > 0$ is small and $\iota (n, \epsilon) \to 0$ fast as $n \to \infty$.

We call  such an estimate {\em uniform} when the rate function $\iota$ is stable under small perturbations of the cocycle $A$. 

\smallskip

%%  Brief description of the results.

Establishing statistical properties like large deviation estimates is a difficult problem for most dynamical systems. 
The first results of this kind for {\em quasi-periodic} base dynamics were obtained by J. Bourgain and M. Goldstein~\cite{B-G-first} and by M. Goldstein and W. Schlag~\cite{GS-Holder} in the context of a one-parameter family of $\SL_2 (\R)$-valued cocycles corresponding to a lattice Schr\"odinger operator.

The cocycles considered here are $\Mat_m (\R)$-valued. We distinguish between {\em identically singular} (i.e. non-invertible anywhere) and non-identically singular cocycles. We studied the latter in Chapter 6 of our monograph~\cite{DK-book}. This paper is concerned with the former, which presents significant technical challenges, especially in the several variables case $d>1$.

The idea is to first prove a {\em non-uniform} fiber LDT for any identically singular cocycle, through reduction to a maximal rank (hence non-identically singular) cocycle of a lesser dimension. The uniform LDT estimate will then be derived by induction on the number of iterates, with the base step provided by the non-uniform statement. 

This result is interesting in itself and for its subsequent applications; moreover, the method developed here to derive it might also prove useful in other contexts.

%\smallskip
\vspace{.5em}

The next result of this paper concerns the continuity of the Lyapunov exponents at {\em any}  cocycle in $\qpcmat{d}$. 
The result is quantitative, in that it also provides an explicit (weak-H\"older) modulus of continuity locally near cocycles with simple Lyapunov exponents.
These statements follow from the  abstract continuity theorem (ACT) obtained in Chapter 3 of our monograph~\cite{DK-book},  which is applicable in this context once the {\em uniform} fiber LDT is proven.

 \vspace{.5em}

The continuity theorem is then used to derive criteria for the positivity and the simplicity of the Lyapunov exponents for quasi-periodic cocycles. This in turn leads to optimal lower bounds on  Lyapunov exponents and on the gaps between consecutive Lyapunov exponents associated with discrete, quasi-periodic band lattice Schr\"odinger operators (also referred to as block Jacobi operators).

We note that (in addition to its intrinsic interest)  treating the case of identically singular cocycles is exactly what allows us to derive the aforementioned consequences on the positivity and simplicity of the LE. More precisely, for cocycles with a certain structure, using a simple limiting argument, we obtain asymptotic formulas for the LE, which, under appropriate assumptions, imply lower bounds on LE or on the gaps between consecutive LE. A simple illustration of this, regarding the  LE associated with discrete Schr\"odinger operators, is given in Proposition~\ref{sorets-spencer-thm}.

%\smallskip
\vspace{.5em}

%% 4-(P) The difference between T and T^d 

Our method applies equally to cocycles on the $1$-variable torus $\T$ and on a several variables torus $\T^d$ with $d>1$. There are currently more and sharper results available for the former model. 
Let us now emphasize the differences between the analysis of the case $d=1$ and that of the case $d>1$ and explain how most  of the one-variable arguments cannot be applied in the several variables setting.

%% complexity of zeros/singularities  
 
% \smallskip
 
In~\cite{Sorets-Spencer} E. Sorets and T. Spencer considered the quasi-periodic Schr\"o\-din\-ger cocycle 
\begin{equation}\label{intro-s-cocycle}
A_{\lambda, E}(x)=\begin{bmatrix} \lambda\, f (x) - E & -1 \\ 1 & \hphantom{-}0 \end{bmatrix} ,
\end{equation}
where $f$ is a (fixed) real analytic and non-constant function on the $1$-variable  torus $\T$.

They proved that the top LE of this cocycle is bounded from below by $\frac{1}{2} \log \sabs{\la}$ for all $E\in\R$ (and also for all translations $\om$) provided that $\sabs{\la} \ge \la_0$, where $\la_0$ depends only on $f$. 
It is important in applications to spectral theory problems for the corresponding lattice Schr\"odinger operator that the lower bound on the LE and the threshold $\la_0$ be uniform in the energy parameter $E$. 

The idea of the proof is to complexify the analytic function $f (x)$ to a neighborhood of $\T$, and then to use the fact that the equation $f (z) - s = 0$ has {\em finitely} many zeros in a compact set, for any given $s \in \R$. This is due precisely to the fact that $f$ is a {\em one variable} holomorphic, non-constant function. 

That neighborhood of the torus contains an annulus $\strip_r$, which we identify with a strip $[0, 1] \times [-r, r]$. Then for most $y \in [-r, r]$, i.e. for most horizontal lines,  the equation $f (x+i y) -   \frac{E}{\la} = 0$ has no solutions in $x$. Writing for such $y$
$$ 
A_{\lambda, E}(x+i y)=\begin{bmatrix} \lambda\, f (x+i y) - E & -1 \\ 1 & \hphantom{-}0 \end{bmatrix} = \begin{bmatrix} \la \,( f (x + i y) - \frac{E}{\la}) & -1 \\ 1 & \hphantom{-}0 \end{bmatrix} ,
$$
it is then clear that the cocycle $A_{\la, E} (\cdot + i y)$ is {\em uniformly hyperbolic}, provided we choose $\la$ large enough. In particular this ensures a lower bound on the top LE of the cocycles $A_{\la, E} (\cdot + i y)$, for most $y \in [-r, r]$.

The challenge is then to transfer such a  lower bound to $A_{\la, E} (x)$, that is, to $y=0$. E. Sorets and T. Spencer accomplish this by proving an extension of Jensen's formula to meromorphic functions; an alternative argument given by J. Bourgain (see~\cite{B}) uses harmonic measure; in~\cite{DK1} we gave a simple argument,  based on Hardy's convexity theorem (see~\cite{Duren}), that applies to general higher dimensional analytic cocycles on $\T$. In all of these arguments, the crucial ingredient is the fact that the function $\un{n}{A} (x) := \frac{1}{n} \log \norm{\An{n}(x) }$ has a {\em subharmonic} extension to $\strip_r$, which is due to $A (z)$ being a holomorphic function on $\strip_r$.

\smallskip

To summarize, positivity of the top LE for analytic quasi-periodic cocycles on  $\T$ for large $\la$ (which is a necessary condition) follows from the following key ingredients: the set $\{ z \colon f (z) - s = 0 \}$ has ``low algebraic complexity'', hence most orbits of the translation will avoid it; and, essentially, the convexity of the map $y \mapsto L_1 (  A_{\la, E} (\cdot + i y) )$. We emphasize that the former is not available in the several variables ($d>1$) torus translation setting. 

We note the fact that the lower bounds on the top LE obtained in~\cite{Sorets-Spencer, B, DK1} in the one-frequency ($d=1$) setting are {\em uniform} in the frequency. Furthermore, Z. Zhang~\cite{Zhang-positivity} obtained a sharp (and uniform) lower bound for the top LE of the Schr\"odinger cocycle \eqref{intro-s-cocycle}, while very recently, R. Han and C. Marx~\cite{Han-Marx} obtained a precise (and uniform) asymptotic formula for the top LE of the same cocycles.  When $d>1$,  the only lower bound available which is uniform in the frequency   
is due to J. Bourgain~\cite{B-d}, and it applies to the several variables analogue of the Schr\"odinger cocycle \eqref{intro-s-cocycle}.

\smallskip

Let us now discuss the continuity of the LE for quasi-periodic cocycles on $\T$. J. Bourgain and S. Jitormiskaya~\cite{B-J}  proved joint continuity in $(E, \om) \in \R \times (\R\setminus\Q)$ of the LE of the Schr\"odinger cocycle~\eqref{intro-s-cocycle} (this particular result was in fact obtained by J. Bourgain in~\cite{B-d} for $\T^d$ with $d>1$ as well). S. Jitomirskaya and C. Marx proved joint continuity in cocycle and frequency for $\Mat_2 (\C)$-valued cocycles. 
More recently, A. \'Avila, S. Jitomirskaya and C. Sadel~\cite{AJS}
extended this result to cocycles of arbitrary dimension. More precisely, they proved 
continuity of all LE in  $(A,\omega)\in C^\omega_r (\T,\Mat_m(\C))\times (\R\setminus \Q)$.
We stress that these continuity results are not quantitative,
unlike the ones addressed in this work.

The proof of the result in \cite{AJS} uses the fact that given an analytic function
$A:\T\to \Mat_m(\C)$, having a holomorphic extension to some strip of width $r>0$, the complex cocycles
$ A_y:\T \to \Mat_m(\C)$, $A_y(x):=A(x+i y)$ 
have dominated splitting for most $y\in (-r,r)$.
The continuity of the Lyapunov exponents for cocycles
having dominated splitting comes with a soft argument 
from the  theory of hyperbolic dynamical systems,
because {\em dominated splitting} is a kind of
projective hyperbolicity.
Finally, the continuity of the LE for $A=A_{y=0}$ follows using 
the convexity of the function $y\mapsto L_1(A_{y})$.

\smallskip

To summarize, continuity of the LE for (general, higher dimensional) analytic quasi-periodic cocycles on  $\T$ ($d=1$) follows from the following key ingredients: the dominated splitting of certain perturbations of the cocycle and a convexity argument.

\smallskip

In contrast with this,
we show in a separate paper 
that for $d>1$ there are  homotopy classes
of analytic functions  $A\colon \T^d\to \GL_m(\C)$
whose corresponding quasi-periodic cocycles do {\em not} have dominated splitting. Therefore, even if the convexity argument had a counterpart when $d>1$, the approach in~\cite{AJS} could not be used to establish the continuity of the LE in the several variables torus translation setting.

\smallskip

When $d=1$, the dominated splitting property of the cocycle $A_y$ for most $y$ relies essentially on the fact that the zeros of a one variable holomorphic function are isolated points.

This property does not hold for several variables analytic functions
$f \colon \T^d \to \C$, 
$d>1$. 
A simple dimension argument shows that generically, the zero set 
$Z(f, y):=\{ x\in\T^d \colon   f(x+i y)=0 \}$,
 has  dimension  $d-2\geq 0$. Hence, if $d>1$, there are
 open sets of analytic functions $f \colon \T^d \to \C$ having holomorphic extensions to a strip of width $r>0$  which have zeros on every torus  $\T^d_y:=\{ x+i y\colon x\in\T^d\}$
 with $y\in (-r,r)^d$. 
 Thus, considering the cocycle
$ 
	A (x):=\begin{bmatrix}  f(x) & -1 \\ 1 & 0 \end{bmatrix}
$, 
even if most values of $f$ are very large,
the cocycles $A_y(x):=A(x+i y)$ are not expected  to be uniformly hyperbolic. This amounts to saying that the cocycles $A_y$ will not have dominated splitting. 

Therefore, the  zeros of the function $f$ present unavoidable technical difficulties when we address the problem of  continuity of the LE of quasi-periodic cocycles  over a several variables  torus translation.

\smallskip

A more robust approach, that works for both $d=1$ and $d>1$, was introduced by M. Goldstein and W. Schlag~\cite{GS-Holder} in the context of Schr\"odinger cocycles $A_{\la, E}$ like~\eqref{intro-s-cocycle}. This approach proceeds by establishing  LDT estimates for the iterates of the cocycle and it uses in an essential way a deterministic result on the norm growth of long products of $\SL_2 (\R)$ matrices, which the authors call the Avalanche Principle (AP). While this method requires that the translation $\om$ be {\em fixed} and satisfy a generic arithmetic condition, the result provides a modulus of continuity\footnote{According to a private conversation of the second author with Qi Zhou, regarding the latter's yet unpublished (joint) work, an arithmetic condition is in fact {\em necessary} for establishing such a modulus of continuity.} (i.e. H\"older when $d=1$, weak-H\"older when $d>1$) for the top Lyapunov exponent regarded as a function of the energy parameter $E$. 

In our recent monograph~\cite{DK-book}, we extend this approach   in both depth and breadth, making it applicable to any space of cocycles, of any dimension and  over any base dynamics, provided appropriate LDT estimates are available in the given setting. We refer the reader to Chapters 1 and 6 of this monograph, as well as to our survey~\cite{DK-survey} for a more thorough review of related results. Moreover, we note that the recent surveys of S. Jitomirskaya and C. Marx~\cite{JM-survey} and D. Damanik~\cite{David-survey} provide the interested reader with an excellent overview of related topics.

\smallskip

Finally, we comment on the simplicity of the Lyapunov exponents. In the 1980s, Y. Guivarc'h and A. Raugi~\cite{Guivarch-Raugi-sim} and I. Ya. Gol'dsheid and G. A. Margulis~\cite{Gol-Margulis-sim} obtained sufficient criteria  for the simplicity of the LE of locally constant cocycles over a Bernoulli shift. More recently, results on this kind for other models were obtained by C. Bonatti and M. Viana~\cite{BV-sim}  and by A. \'Avila and M. Viana~\cite{AV-sim1, AV-sim2}. 

We are not aware of any previously established sufficient criteria for the simplicity of the LE for  {\em quasi-periodic} models. However, in the $1$-variable torus translation case, the general continuity result in~\cite{AJS} can be used to obtain a similar (but slightly less precise) criterion for simplicity to the one we formulate and prove here.
%%%% Technicalities

\smallskip

As mentioned earlier, the key to all the results in this paper is proving a {\em uniform} fiber LDT estimate like~\eqref{intro-ldt-eq}. The uniformity of this estimate  in the cocycle represents the crucial assumption in the proof of the ACT. This is the main reason for the case of {\em identically singular} cocycles  being significantly more challenging, especially when con\-si\-de\-ring translations on the several variables ($d>1$) torus. We give some details on these challenges and on how they will be overcome.

One difficulty is related  to the fact that being identically singular (and hence not having full rank) is {\em not} an open condition in the space of analytic cocycles. Thus a small perturbation of such a cocycle could have a higher rank. 

Given an identically singular cocycle $A$ with $L_1 (A) > - \infty$, we show that it is semi-conjugated to a maximal rank (hence non-identically singular) cocycle of a smaller dimension; we previously established (see chapter 6 in \cite{DK-book}) uniform fiber LDT estimates for non-identically singular cocycles; via the semi-conjugacy relation, this leads to fiber LDT for the cocycle $A$. Incidentally we also derive the fact (which was recently independently proven in~\cite{Sadel-Xu}) that an analytic cocycle 
is nilpotent if and only if its top Lyapunov exponent is $- \infty$.

We note that  since the rank of $A$ may change under perturbations, the parameters of the LDT obtained through semi-conjugacy may blow up. 
Thus this argument will only provide {\em non-uniform} fiber LDT estimates for analytic cocycles. 
We prove a {\em uniform} fiber LDT estimate in the vicinity of an identically singular cocycle $A$ using an inductive procedure based on the avalanche principle. We explain below the mechanics of the proof.

The non-uniform LDT is used to get the procedure started\textemdash given a large enough initial scale (i.e. number of iterates) $n_0$, the fiber LDT estimate holds for $A$ at this scale, hence if we choose any other cocycle $B$ at a small enough distance from $A$, this estimate will transfer over to $B$ by proximity. Of course, this can only be done once, at an initial scale of order $n_0$, as the size of the neighborhood of $A$ depends on $n_0$. Therefore, we obtain an estimate of the form
\begin{equation}\label{intro-ldt-eqt}
\babs{  \{  x \in \T^d \colon \abs{ \frac{1}{n} \log \norm{\Bn{n} (x)}   - \LE{n}_1 (B) } > n^{-a}  \} } < e^{- n^b} ,
\end{equation}
for some constants $a, b > 0$, for $n \asymp n_0$ and for all cocycles $B$ in a small neighborhood of $A$.

The goal is to derive an estimate like~\eqref{intro-ldt-eqt} at a next scale $n_1 \gg n_0$, then at a scale $n_2 \gg n_1$ and so on, in such a way that the parameters $a$ and $b$ in this estimate do not change from one step to the next (to be precise, they will only stabilize from scale $n_1$ on).  

For every scale $n$ and cocycle $B$ consider the function
$$\un{n}{B} (x) := \frac{1}{n} \log \norm{\Bn{n} (x)} .$$

The AP will essentially allow us to represent the function $\un{n_1}{B} (x)$ corresponding to the scale $n_1$, in terms of certain {\em Birkhoff averages} of the function $\un{m}{B} (x)$ corresponding to scales $m \asymp n_0$, for which the estimate~\eqref{intro-ldt-eqt} is already available. There is, of course, an error term in this representation; furthermore, this representation does not hold for all phases $x$, but only outside of a relatively small set of phases related to the exceptional set in the LDT estimate~\eqref{intro-ldt-eqt} at scales of order $n_0$.

By the pointwise ergodic theorem, these Birkhoff averages converge almost everywhere to the means of the corresponding observables. However, since we are performing an inductive process (or a multiscale a\-na\-ly\-sis with finite steps) we need a {\em quantitative} version of this result, one where the rate of convergence and the size of the exceptional sets of phases depend explicitly on the number of iterates.  

Such a {\em quantitative Birkhoff ergodic theorem} (qBET) will indeed hold due to: the arithmetic assumption on the translation $\om$ (this takes care of some small denominators issues); and the fact that since the cocycles $B (x)$ are real analytic (hence they have holomorphic extensions), the functions $\un{n}{B} (x)$ defined above extend to a neighborhood of $\T^d$ as {\em separately subharmonic} functions (i.e. subharmonic in each variable).

It is crucial for our purposes that the parameters that determine the qBET for the observables $\un{n}{B}$ be {\em uniform} in both $B$ (in a fixed neighborhood of $A$) and $n$.  Moreover, we also need a uniform bound on the $L^2$-norms of these observables\textemdash in part because there are sets of phases over which we have no control, hence an estimate on $\LE{n}_1 (B) = \int_{\T^d} \un{n}{B} (x) d x$ can be obtained  from having control on $\un{n}{B} (x)$ for most phases $x$ and a bound on the global $L^2$-norm of $\un{n}{B}$. These are the two most technically challenging aspects of the proof, and the challenge comes from the fact that the cocycles $B$ may be identically singular. 

The qBET for subharmonic (when $d=1$) or separately subharmonic (when $d>1$) functions $u (z)$ were proven (see for instance~\cite{B, GS-Holder, sK2}) under the assumption that $\sabs{u(z)} \le C < \infty$ throughout the domain. The parameters that determine this estimate depend only on the bound $C$ and the size of the domain, hence such a qBET applies uniformly if the observables considered are uniformly bounded. 

In the context of this paper,  the separately subharmonic functions $\un{n}{B} (z) = \frac{1}{n} \log \norm{\Bn{n} (z)}$ are clearly uniformly bounded from above. But since $\Bn{n} (z)$ may be $0$, which in the case $d>1$ may hold on an algebraically non-trivial set, these functions might not be bounded from below. 
However, a subharmonic function cannot be too small (i.e. too close to $- \infty$) for too much of the phase space, unless it were that small throughout the {\em whole} space. This is due to Cartan's estimate on logarithmic potentials, and by means of Fubini it also holds for separately subharmonic functions.   

Thus, if we can establish that the functions $\un{n}{B} (x)$ have a uniform lower bound at least somewhere, then each of these function could fall below a certain (low enough) threshold only on a small set of phases $x$. We horizontally truncate $\un{n}{B} (z)$ from below, so the resulting function is still separately subharmonic, it is bounded from above and below (hence the qBET for bounded observables applies to it) and it agrees with $\un{n}{B}$ over much of the phase space, provided the threshold for the truncation is chosen small enough. This will allow us to transfer the qBET (and other estimates) from the truncation to $\un{n}{B}$. The threshold for the truncation of $\un{n}{B}$ will be of the order $- n^a$, for some $a \in (0,1)$, hence the bound on the corresponding truncation will be of order $n^a$. This does produce errors of that order, but they can be easily absorbed. 

Finally, we note that the key fact that $\un{n}{B} (x)$ have a uniform lower bound at least somewhere, cannot be established a-priori, but it will be obtained inductively and used to feed the next step of the induction.

%%%%%%%%%%%%%%
%%%%%%%%%%%%%

\subsection*{Formulation of the main statements} We consider the Banach space $\qpcmat{d}$ of linear cocycles on $\T^d$ having a holomorphic extension to $\strip_r^d$, continuous up to the boundary. This space is endowed with the uniform norm. 

Given $t \in (0, 1)$, we denote by $\DC_t$ the set of translation vectors $\omega\in\R^d$ satisfying the following {\em Diophantine condition}:\,  
$$
\norm{k \cdot \om} \ge \frac{t}{\sabs{k}^{d+1}}
\quad \text { for all }\; k\in \Z^d\setminus\{0\}\;,
$$
where  for any real number $x$ we write $\norm{x}:=\min_{k\in\Z} \abs{x-k}$.

\smallskip

The main result of this paper, from which everything else follows, is a large deviation estimate on the iterates of a cocycle in this space. 

\begin{theorem}\label{intro-ufiber-ldt-thm}
Given  $A \in \qpcmat{d}$ with $L_1 (A) > L_2 (A)$ and $\om \in \rm{DC}_t$, there are constants $\delta = \delta (A) > 0$, $n_1 = n_1 (A, t) \in \N$, $a_1 = a_1 (d) > 0$, $b_1 = b_1 (d)  > 0$  so that if $\normr{B-A} \le \delta$ and $n \ge n_1$, then 
\begin{equation*} 
\abs{  \{ x \in \T^d \colon \abs{ \frac{1}{n} \log \norm{\Bn{n} (x)} - \LE{n}_1  (B)  } > n^{-a_1} \} } < e^{- n^{b_1}}.
\end{equation*}
\end{theorem}

The crucial feature of the result above for our consequent continuity statements, and what makes its derivation most challenging, is the local {\em uniformity} of the estimates in the cocycle.

\smallskip

We say that a function is {\em weak-H\"older}  continuous if its modulus of continuity is given by 
$ w (h) = C \, e^{- c\,( \log (1/h) )^b}$,
for some positive constants $C$, $c$ and $b$. Note that if $b=1$  then  we have
$w (h)=C\,h^c$, which corresponds to H\"older  continuity.

We can now formulate the statement on the continuity of the Lyapunov exponents. They (and their consequences)  hold under the assumption that  $\om \in \rm{DC}_t$, that is, the translation vector is Diophantine.

\begin{theorem}\label{cont-le}
The map $\qpcmat{d} \ni A \mapsto L_k (A) \in [ - \infty, \infty)$ is continuous for all $1 \le  k \le m$. 

Moreover, if for some $A$ and $k<m$ we have $L_k (A) > L_{k+1} (A)$, then locally near $A$ the map $\qpcmat{d} \ni B \mapsto (L_1 + \ldots + L_k)  (B) \in \R$ is weak-H\"older continuous.
\end{theorem}

\medskip

Similarly to our results for non-identically singular cocycles, the Oseledets filtration and decomposition also depend continuously on the cocycle, and in fact they are weak-H\"older continuous too. Phrasing the most general version of this result is a bit technical, as one would firstly have to carefully describe the topology considered on the space of flag-valued functions (to which the Oseledets filtration of a cocycle belongs). To avoid these technicalities, here we formulate just a particular version of this result. 

Let $A \in \qpcmat{d}$  and let $1\le k < m$. Assume that $L_k(A)>L_{k+1}(A)$. Denote by $ \Gr_{k}(\R^m)$ the Grassmann manifold of $k$-dimensional subspaces of $\R^m$.

We define $E^-_k (A) \colon \T^d \to  \Gr_{m-k}(\R^m)$ 
to be the measurable component of the Oseledets filtration of $A$ corresponding to  the Lyapunov exponents $\leq L_{k+1}(A)$,
and $E^+_k (A) \colon \T^d \to  \Gr_{k}(\R^m)$ 
to be the direct sum of the components of the Oseledets decomposition corresponding to Lyapunov exponents $\geq L_k(A)$.
With these notations we have the following statement.

\begin{theorem}\label{cont-oseledets} 
Let $A \in \qpcmat{d}$, $1\le k < m$ with $L_k(A)>L_{k+1}(A)$. There are $\delta = \delta (A) > 0$ and 
$\alpha=\alpha(A)>0$ such that 
for all  $B_1, B_2\in \qpcmat{d}$ with $\norm{B_i-A}_r \le \delta$, $i=1,2$ we have
$$\Big| \{ x \in \T^d \colon d ( E^\pm_k (B_1) (x), \, E^\pm_k (B_2) (x) ) >  \norm{B_1-B_2}_r^\alpha \ \}  \Big| < w (\norm{B_1-B_2}_r ) ,$$
where $d$ refers to the distance on the Grassmann manifold and $w$ refers to a weak-H\"older modulus of continuity function. 
\end{theorem}

The reader may consult our monograph~\cite{DK-book} (see Chapter 6, Theorem 6.1) for the general result in the case of non-identically singular cocycles; the same will hold here as well.  

\medskip

To show the usefulness of our continuity result of the Lyapunov exponents for {\em identically singular} cocycles, we present an immediate application\textemdash the positivity of the maximal LE for Schr\"odinger cocycles. 

Recall that a discrete, one-dimensional, quasi-periodic  Schr\"odinger operator is an operator $H_\la (x)$ on $l^2 (\Z) \ni \psi = \{\psi_n\}_{n \in \Z}$, defined by
\begin{equation}\label{ s op}
[ H_\la (x) \, \psi ]_n := - (\psi_{n+1} + \psi_{n-1}) + \la \, f (x + n \om) \, \psi_n\,,
\end{equation} 
where  $\la \neq 0$ is a coupling constant, $f \colon \T^d \to \R$ is the potential function, $x \in \T^d$ is a phase parameter that introduces some randomness into the system  and $\om \in \T^d$ is a fixed incommensurable frequency.

We note that due to the ergodicity of the system, the spectral pro\-per\-ties of the family of operators $\{ H_\la (x) \colon x \in \T^d \}$ are independent of $x$ almost surely. Moreover, $H_\la (x)$ is a bounded, self-adjoint operator, whose norm (and hence spectral radius) is $\le 2 + \sabs{\la} \norm{f}_{L^\infty(\T^d)}$.

Consider the Schr\"odinger (i.e. eigenvalue) equation
\begin{equation}\label{s eq}
H_\la (x) \, \psi = E \, \psi\,,
\end{equation}
for some energy (i.e eigenvalue) $E \in \R$ and state (i.e. eigenvector) $\psi = \{\psi_n\}_{n \in \Z} \subset \R$.

Define the associated  {\em Schr\"odinger cocycle} as the one-parameter (in $E$)  fa\-mi\-ly $A_{\la, E}$, where
$$A_{\la, E} (x) := \left[ \begin{array}{ccc} 
\la \, f (x)  - E  & &  -1  \\
1 & &  \phantom{-}0 \\  \end{array} \right] \in \SL (2, \R)\,.$$

Note that the Schr\"odinger equation~\eqref{s eq} is a second order finite difference equation. An easy calculation shows that its formal solutions are given by 
\begin{align*} 
\left[\begin{array}{c}
\psi_{n+1}\\
\psi_n \\ \end{array}\right]  =  
A^{(n+1)}_{\la, E} (x)
 \cdot  \left[\begin{array}{c}
\psi_0\\
\psi_{-1} \\ \end{array}\right]\,,
\end{align*}
where $A^{(n)}_{\la, E} (x)$ are the iterates of $A_{\la, E} (x)$, for all $n \in \N$ .

\begin{proposition}\label{sorets-spencer-thm}
Let $\om \in \DC_t$, let $f \in \analyticf{d}$ be an analytic, non-constant function, 
and consider the corresponding quasi-periodic Schr\"odinger cocycle
$$A_{\la, E} (x) = \left[ \begin{array}{ccc} \la f (x) - E  & &  - 1  \\
1 & &  \hphantom{-}0 \\  \end{array} \right]  .$$

There are $\la_0 = \la_0 (t, f) < \infty$, $c = c (f) > 0$ and $b = b (d) > 0$, such that if $\sabs{\la} \ge \la_0$ then the top Lyapunov exponent of the cocycle $A_{\la, E}$
has the following asymptotic behavior.

\begin{subequations}\label{sorets-spencer-eqs}
\begin{enumerate}[(a)]
\item If $\sabs{E} \le 2 \, \sabs{\la} \, \norm{f}_r$ then   
\begin{equation}\label{sorets-spencer-eq1} 
L_1 ( A_{\la, E} ) = \int_{\T^d} \, \log \abs{ \la f (x) - E } \, d x + \mathcal{O} \left( e^{- c \, ( \log \sabs{\la} )^b} \right) .  
\end{equation}
\item If $\sabs{E} \ge 2 \, \sabs{\la} \, \norm{f}_r$ then   
\begin{equation}\label{sorets-spencer-eq2}
L_1 ( A_{\la, E} ) = \int_{\T^d} \, \log \abs{ \la f (x) - E } \, d x + \mathcal{O} \left( e^{- c \, ( \log \sabs{E} )^b} \right) .
\end{equation}
\end{enumerate}

In particular, there is $C_0 = C_0 (f) < \infty$ such that  if $\sabs{\la} \ge \la_0$ then
\begin{equation}\label{sorets-spencer-eq3}
L_1 ( A_{\la, E} ) > \log \sabs{\la} - C_0 \quad \text{for all } \, E \in \R .
\end{equation}
\end{subequations}
\end{proposition}

\begin{proof} We first consider the case $\sabs{E} \le 2 \, \sabs{\la} \, \norm{f}_r$.

Given $\delta, s \in \R$ define the cocycle
$$S_{\delta, \, s} (x) := \left[ \begin{array}{ccc} 
 f (x)  - s  & &  -\delta  \\
\delta & &  \phantom{-}0 \\  \end{array} \right] \in \Mat_2 (\R) .$$

By factoring out $\la$ we have $A_{\la, E} = \la \, S_{\frac{1}{\la}, \, \frac{E}{\la}} $ so
$$L_1 ( A_{\la, E} ) = \log \sabs {\la} + L_1 ( S_{\frac{1}{\la}, \, \frac{E}{\la}} ) .$$ 

The map $(\delta, s) \mapsto S_{\delta, \, s}$ is Lipschitz. 

Moreover, 
for all $s \in \R$, the cocycle $S_{0, \, s} = \left[ \begin{array}{cc} 
 f (x)  - s  &  0  \\
0 &   0 \\  \end{array} \right]$, which is {\em identically singular},  satisfies  $$L_1 ( S_{0, \, s} ) = \int_{\T^d} \log \abs{f (x) - s} \, d x > - \infty$$ (because $f$ is analytic and non-constant) and $L_2 ( S_{0, \, s} ) = - \infty$.

In particular, $L_1 ( S_{0, \, s} ) > L_2 ( S_{0, \, s} )$, so by the continuity Theorem~\ref{cont-le},  the map $B \mapsto L_1 (B)$ is locally weak-H\"older near $S_{0, \, s}$ for every $s \in \R$.

Then by compactness, there is $\delta_0 > 0$ such that the map
$$[-\delta_0, \delta_0] \times \big[-2\norm{f}_r, \, 2  \norm{f}_r \big]  \ni (\delta, s) \mapsto L_1 (S_{\delta, \, s}) \in \R$$
is weak-H\"older continuous. 

Therefore,  there is a modulus of continuity function of the form $w (h) = C  e^{- c \, ( \log \frac{1}{h} )^b}$ such that if $\abs{\delta} \le \delta_0$ and if $\abs{s} \le 2  \norm{f}_r$ then
$$\abs{ L_1 (S_{\delta, \, s}) - L_1 (S_{0, \, s})  } \le w (\sabs{\delta}) \to 0 \quad \text{as } \delta \to 0 .$$

Let $\la_0 := \frac{1}{\delta_0}$. Then for all $\la, E$ with $\sabs{\la} \ge \la_0$ and $\sabs{E} \le 2 \, \sabs{\la} \, \norm{f}_r$, via the change of variables $\delta = \frac{1}{\la}$, $s = \frac{E}{\la}$,  from the above we have
$$\abs{ L_1 (S_{\frac{1}{\la}, \, \frac{E}{\la}}) - L_1 (S_{0, \, \frac{E}{\la}})  } \le w \big(\frac{1}{\sabs{\la}} \big) \to 0 \quad \text{as } \sabs{\la} \to \infty .$$

This then translates into
$$\babs{  L_1 ( A_{\la, E} ) - \log \sabs{\la} -  \int_{\T^d} \log \abs{ f (x) - \frac{E}{\la} } \, d x  } \le w \big(\frac{1}{\sabs{\la}}\big) ,$$
which implies~\eqref{sorets-spencer-eq1}.

\smallskip 

The second case, when $\sabs{E} \ge 2 \sabs{\la} \norm{f}_r$, is treated similarly\textemdash by factoring out $E$ instead. Indeed, for $\delta, s \in \R$, define the cocycle 
$$S_{\delta, \, s} (x) := \left[ \begin{array}{ccc} 
 s \, f (x)  - 1  & &  -\delta  \\
\delta & &  \phantom{-}0 \\  \end{array} \right] \in \Mat_2 (\R) .$$

By factoring out $E$ we have $A_{\la, E} = E \, S_{\frac{1}{E}, \, \frac{\la}{E}} $ so
$$L_1 ( A_{\la, E} ) = \log \sabs{E} + L_1 ( S_{\frac{1}{E}, \, \frac{\la}{E}} ) .$$ 

The cocycle $S_{0, \, s}$ is identically singular and for all $s \in \R$ we have $$L_1 ( S_{0, \, s} ) = \int_{\T^d} \log \abs{s f (x) - 1} \, d x > - \infty$$  (since $f$ is analytic and non-constant), while $L_2 ( S_{0, \, s} ) = - \infty$.  The continuity Theorem~\ref{cont-le} is again applicable, ensuring weak-H\"older continuity of the top Lyapunov exponent locally near $S_{0, \, s} $ for every $s \in \R$. 
Then by a simple compactness argument, there is $\delta_0 > 0$ such that the map
 $$[-\delta_0, \delta_0] \times \big[- \frac{1}{2 \norm{f}_r},  \, \frac{1}{2  \norm{f}_r}  \big]  \ni (\delta, s) \mapsto L_1 (S_{\delta, \, s}) \in \R$$
has a weak-H\"older modulus of continuity $w (h)$ as before, that is,
$$\abs{ L_1 (S_{\delta, \, s}) - L_1 (S_{0, \, s})  } \le w (\sabs{\delta}) .$$

Let $\la_0 := \frac{1}{\delta_0 \, 2 \norm{f}_r}$. With the change of coordinates $\delta = \frac{1}{E}$, $s = \frac{\la}{E}$, if $\sabs{\la} \ge \la_0$ and $\sabs{E} \ge 2 \sabs{\la} \norm{f}_r$, then the estimate above applies and we have
$$\abs{ L_1 (S_{\frac{1}{E}, \, \frac{\la}{E}}) - L_1 (S_{0, \, \frac{\la}{E}})  } \le w \big(\frac{1}{\sabs{E}} \big) \to 0 \quad \text{as } \sabs{E} \to \infty .$$

This then translates into
$$\babs{  L_1 ( A_{\la, E} ) - \log \sabs{E} -  \int_{\T^d} \log \abs{ \frac{\la}{E} \, f (x) - 1 } \, d x  } \le w \big(\frac{1}{\sabs{E}}\big) ,$$
which implies~\eqref{sorets-spencer-eq2}.

\smallskip

To derive the last estimate, first note that 
since $f (x) - s \not \equiv 0$ for every $s \in \R$, if $I \subset \R$ is compact, then the map 
$$I \ni s \mapsto  \int_{\T^d} \log \abs{ f (x) - s } \, d x \in \R$$
has a finite lower bound. 

This is due to the {\L}ojasiewicz inequality (we formulate it below, in Proposition~\ref{Loj-prop}, see also Remark 6.2 in \cite{DK-book}), which holds uniformly in a neighborhood of an analytic, non-identically zero function. Alternatively, although this is an overkill, since $\int_{\T^d} \log \sabs{ f (x) - s } d x$ represents the Lyapunov exponent of the one-dimensional cocycle $f - s$, by the same continuity theorem of the Lyapunov exponents, the map above depends continuously on $s$, hence it has a finite lower bound on any compact set.

Then either we have $\sabs{E} \le 2 \sabs{\la} \norm{f}_r$, so $\frac{E}{\la}$ is in a compact set, and applying ~\eqref{sorets-spencer-eq1} we have
$$L_1 ( A_{\la, E} ) \ge  \log \sabs{\la} + \int_{\T^d} \, \log \abs{ f (x) - \frac{E}{\la} } \, d x - w \big(\frac{1}{\sabs{\la}} \big) > \log \sabs{\la}  - C_0 \, ;$$

Or $\sabs{E} \ge 2 \sabs{\la} \norm{f}_r$, so for all $x \in \T^d$, $\abs{\frac{\la}{E} \, f (x)} \le \frac{1}{2}$ and applying ~\eqref{sorets-spencer-eq2} we have
\begin{align*}
L_1 ( A_{\la, E} ) & \ge  \log \sabs{E} + \int_{\T^d} \, \log \abs{\frac{\la}{E} \, f (x) - 1 } \, d x - w \big(\frac{1}{\sabs{E}} \big) \\
& \ge \log \sabs{\la} + \log 2 \norm{f}_r + \log \frac{1}{2} - w \big(\frac{1}{\sabs{E}} \big) >
\log \sabs{\la}  - C_0 .
\end{align*}

This completes the proof.
\end{proof}

\begin{remark}
Having uniform lower bounds of the form $\more \log \sabs{\la}$ on the Lyapunov exponent of a Schr\"odinger cocycle with analytic potential is of course not new. 
M. Herman~\cite{Herman} and E. Sorets and T. Spencer~\cite{Sorets-Spencer} are classical results on this topic, with uniform bounds in $\om \in \T$. 

A similar estimate to~\eqref{sorets-spencer-eq1}  (also under a Diophantine assumption on the frequency) was obtained in the case $d=1$ by J. Bourgain (see Proposition 11.31 in \cite{B}), although the error there is less sharp/explicit.

Moreover, asymptotic formulas for the LE of $2$-dimensional Jacobi cocycles over the $1$-variable torus translation were obtained in~\cite{JitMarx-CMP}. And the continuity theorem in~\cite{AJS} for general, higher dimensional cocycles over the same $1$-variable torus may be employed to the same effects.

\smallskip

What is new here is the {\em precision} of the estimates~\eqref{sorets-spencer-eqs}, especially in the case of $f \in \analyticf{d}$ with $d > 1$.  \footnote{When $d=1$, an even more precise estimate, which is moreover uniform in the frequency, was very recently obtained in~\cite{Han-Marx}.}  
\end{remark}

Next we formulate similar but more general consequences of the continuity Theorem~\ref{cont-le}, namely some sufficient criteria for the positivity and simplicity of the LE of a quasi-periodic cocycle. 

\begin{theorem}\label{pos-sim-thm} Given $\om \in \DC_t$ and dimensions $1\le l < m$, for every $\la \neq 0$
consider a cocycle $A_\la \in \qpcmat{d}$ with a block structure
$$A_{\la} = \left[ \begin{array}{ccc} \la \, M & &  N  \\
P & &  Q \\  \end{array} \right]  ,$$
where $M$ is a cocycle of dimension $l$, i.e. $M \in \qpcmatl{d}$. 

There is $\la_0 = \la_0 (t, M, \norm{N}_r, \norm{P}_r, \norm{Q}_r) < \infty$  such that for all $\sabs{\la} \ge \la_0$ the following hold.
\begin{subequations}\label{pos-sim-eq}
\begin{enumerate}[(a)]
\item If $M$ is non identically singular, i.e. if $\det [ M (x) ] \not \equiv 0$, then there is $C_0 = C_0 (M) < \infty$ such that
\begin{equation}\label{pos-sim-eq1}
L_l (A_\la) > \log \sabs{\la} - C_0 .
\end{equation}
\item If all Lyapunov exponents of $M$ are simple, then the $l$ largest Lyapunov exponents of $A_\la$ are also simple. 

In fact, a precise estimate on the gap between consecutive Lyapunov exponents holds in this case. There are positive constants $c = c(M)$ and $b = b (d)$ such that for all $1 \le k \le l$,
\begin{align}
L_k ( A_{\la} ) &= \log \sabs{\la} + L_k (M)  + \mathcal{O} \left( e^{- c \, (\log \sabs{\la})^b} \right), \text{ so } \label{pos-sim-eq2}\\
L_k ( A_{\la} ) - L_{k+1} ( A_{\la} ) &= L_k ( M) - L_{k+1} ( M )  + \mathcal{O} \left( e^{- c \, (\log \sabs{\la})^b} \right) \label{pos-sim-eq3}
\end{align}
for all $k < l$.
\end{enumerate} 
\end{subequations}
\end{theorem}

\begin{remark}
By factoring out $\la$, the cocycle in the above theorem can be written as $A_\la = \la \, S_{\frac{1}{\la}}$, where $S_\delta :=  \left[ \begin{array}{cc} M & \delta  N  \\
\delta P & \delta  Q \\  \end{array} \right]$. 

Since as $\delta \to 0$, $S_\delta \to  \left[ \begin{array}{cc} M & 0 \\
0 & 0 \\  \end{array} \right]$, by the continuity Theorem~\ref{cont-le} we have that for all $1 \le k \le l$, $L_k (S_\delta) \approx L_k (M)$ as $\delta \to 0$.

Then $L_k (A_\la) \approx \log \sabs{\la} + L_k (M)$ as $\sabs{\la} \to \infty$, so if $k < l$ we also have
$$L_k (A_\la) - L_{k+1} (A_\la) \approx L_k (M) - L_{k+1} (M) \quad \text{ as }  \sabs{\la} \to \infty.$$

This says that whatever the multiplicity of a Lyapunov exponent of the block $M$, if $\sabs{\la} \gg 1$ then the multiplicity of the corresponding Lyapunov exponent of $A_\la$ is the same or lesser.

Item (b) in Theorem~\ref{pos-sim-thm}  gives a precise estimate on the gap size between consecutive Lyapunov exponents when they are all simple.
\end{remark}

As suggested already by Proposition~\ref{sorets-spencer-thm}, in applications to ma\-the\-matical physics problems
it is useful to have a lower bound on the Lyapunov exponents which is uniform with respect
to an extra pa\-ra\-me\-ter. That parameter corresponds to the energy  $E$ of a Schr\"odinger-like family of cocycles.

Below we formulate such a result on the uniform positivity of the Lyapunov exponents  for a one-parameter family of higher dimensional cocycles. This uniform bound will then be applicable to the Lyapunov exponents associated with (band lattice) Schr\"odinger or Jacobi operators (see Section~\ref{mathphys}).

\begin{theorem}\label{pos-energy-thm} Given $\om \in \DC_t$ and dimensions $1 \le l < m$, 
consider the family of cocycles $A_{\la, E} \in \qpcmat{d}$ with a block structure
$$A_{\la, E} = \left[ \begin{array}{ccc} M_{\la, E} & &  N  \\
P & &  Q \\  \end{array} \right] ,$$
where the parameters $\la \neq 0$, $E \in \R$ and the block $M_{\la, E}$ has the form
$$M_{\la, E} (x) = U (x) \, (\la \, F (x) + R (x) - E \, I)$$ 
for  some $U, F, R \in \qpcmatl{d}$, and with $I \in \Mat_l (\R)$ denoting the identity matrix.

We assume that $U$ is non-identically singular and that $F$ has no constant eigenvalues, i.e. $\det [ U (x) ] \not \equiv 0$ and $\det [ F (x) - s \, I ] \not \equiv 0$ for all $s \in \R$. 

There are constants $\la_0 = \la_0 (t, U, F, \norm{R}_r, \norm{N}_r, \norm{P}_r, \norm{Q}_r) < \infty$ and $C_0 = C_0 (U, F) < \infty$  such that if $\sabs{\la} \ge \la_0$ then 
$$L_l (A_{\la, E} ) > \log \sabs{\la} - C_0 \quad \text{for all } E \in \R .$$ 
\end{theorem}   

In Section~\ref{mathphys} we specialize to cocycles associated to block Jacobi operators. These types of operators generalize in different ways the Schr\"odinger operator~\eqref{ s op} and they may  in some sense be regarded as approximations of higher dimensional discrete Schr\"odinger operators. 

The results above imply the continuity of the Lyapunov exponents of such an operator. They also provide (for a large enough coupling constant) optimal lower bounds on its non-negative (i.e. the first half) Lyapunov exponents, sufficient criteria for their simplicity and weak-H\"older continuity of its integrated density of states.  

By analogy to discrete Schr\"odinger operators like~\eqref{ s op}, the availability of these properties could prove crucial in the further study of the spectral properties of block Jacobi operators.

\medskip

The rest of the paper is organized as follows.

%% Non uniform fiber LDT
In Section~\ref{non-unif-ldt} we prove  non uniform fiber LDT estimates for 
identically singular cocycles (Theorem~\ref{nonunif-ldt-thm}).
For that we show that such cocycles are semi-conjugated to maximal rank  cocycles 
(Lemma~\ref{semi-conjugation}). This in turn allows us to relate
the pointwise finite scale LE of an identically singular cocycle
with those of its maximal rank reduction
(Propositions~\ref{reduction-prop} and~\ref{reduction-thm}).

%% Estimates on separatly subharmonic functions
In Section~\ref{ssh-estimates} we prove a quantitative Birkhoff ergodic theorem (Theorem~\ref{qBET-ssh-thm})
 for functions with a separately subharmonic extension to $\strip_r^d$, under the weaker assumption of having a lower bound only at {\em some} point on $\T^d$.

%% Proof of the LDT theorem
In Section~\ref{ufiber-ldt} we prove our main result,
 on the existence of uniform fiber LDT estimates for  identically singular cocycles (Theorem~\ref{u-fiber-ldt-thm}).

%% Proofs of main statements
In Section~\ref{proofs} we prove Theorems~\ref{cont-le} and \ref{cont-oseledets} on the continuity of the Lyapunov exponents and of the Oseledets filtration. Then we prove Theorems~\ref{pos-sim-thm} and~\ref{pos-energy-thm} on the positivity and simplicity of LE for certain families of cocycles.

%% Consequences for Jacobi operators
Finally, in Section~\ref{mathphys} we apply the previous results to block Jacobi operators (Theorem~\ref{pos-band-op} and Corollary~\ref{Jacobi IDS}).

%%%%%%%%%%%%%%%%%%%%%%%%%%%%%%%%%%%%%%%%%%%%%%%%%%%%%%%%%%%%%%%%
%%%%%%%%%%%%%%%%%%%%%%%%%%%%%%%%%%%%%%%%%%%%%%%%%%%%%%%%%%%%%%%%

\section{The proof of the non-uniform fiber LDT}\label{non-unif-ldt}
\newcommand{\rank}[1]{{\rm rank}({#1})}
\newcommand{\Ker}[1]{{\rm Ker}({#1})}
\newcommand{\adj}{{\rm adj}}
\newcommand{\Range}[1]{{\rm Range}({#1})}

\newcommand{\ldtmeas}{\mathscr{C}}

\newcommand{\bigo}{\mathcal{O}}

\newcommand{\inv}{^{+}}

%&  \kern-3em

In a previous work (see \cite{DK-book}), we established (uniform) fiber LDT estimates for analytic cocycles $R (x)$ that are {\em non identically singular}, in the sense that $\det [ R (x) ] \not \equiv 0$.  

The goal of this section is to establish a non-uniform fiber LDT estimate for identically singular cocycles $A \in \qpcmat{d}$ with $L_1 (A) > - \infty$. This is obtained by means of semi-conjugating the iterates of $A (x)$ with those of a reduced maximal rank cocycle $R_A (x)$. Having maximal rank, the reduced cocycle $R_A (x)$ is non-identically singular, hence it satisfies a fiber LDT estimate, and the semi-conjugacy relation allows us to transfer over the LDT estimate to $A (x)$. We note that this approach is {\em not stable} under perturbations of the original cocycle $A (x)$, hence it cannot provide a uniform fiber LDT estimate.

%\medskip
%

\subsection*{Some linear algebra considerations.}
Let $1\leq k \leq m$ and consider a matrix $V\in \Mat_{k\times m}(\R)$
with rank $k$, which is equivalent to saying that
$\det(V V^T)\neq 0$.
The matrix $V$ determines a linear map $V:\R^m\to \R^k$.
Let $V\inv:\R^k \to \R^m$ be the corresponding pseudo-inverse.
Geometrically we take $W=\Ker{V}^\perp$, so that
$V\vert_W:W\to \R^k$ is an isomorphism and set
$$V\inv:= (V\vert_W)^{-1}:\R^k\to W\subset\R^m. $$

The next proposition shows that this definition matches the usual
Moore-Penrose pseudo-inverse.

\begin{lemma}
\label{pseudo-inverse:char}
Given  $V\in \Mat_{k\times m}(\R)$
with rank $k$,
$$ V\inv = V^T ( V V^T)^{-1} .$$
\end{lemma}

\begin{proof}
Since $V V^T\in \Mat_k(\R)$ is a positive symmetric matrix, it admits an orthonormal basis of eigenvectors $v_j$, $j=1,\ldots, k$
such that $V V^T v_j = \lambda_j v_j$ with $\lambda_j>0$.
Hence $(V V^T)^{-1} v_j = \lambda_j^{-1} v_j$ and for all $j=1,\ldots, k$
\begin{align*}
V[ V^T (V V^T)^{-1} v_j ] &= \lambda_j^{-1} V V^T v_j = v_j .
\end{align*}
Now, because  $V^T (V V^T)^{-1} v_j \in W$,
it follows that for all $j=1,\ldots, k$
$$ V\inv v_j = V^T (V V^T)^{-1} v_j . $$
Because these vectors form a basis of $\R^k$ the identity follows.
\end{proof}

\begin{lemma}\label{pseudo-inv-estimate}
If  $V\in \Mat_{k\times m}(\R)$
has rank $k$, then
$$ \frac{1}{\norm{V}} \le \norm{V\inv} \leq \sqrt{ \frac{\norm{\adj (V V^T)}}{\det(V V^T)}}  \ ,$$
where $\adj$ denotes the adjugate of a matrix, i.e. the transpose of its cofactor matrix.
\end{lemma}

\begin{proof} The first inequality follows from the fact that $V\inv$ is a right inverse of $V$: $I = V \, V\inv$.

For the second, using the previous lemma, for any unit vector $x$
\begin{align*} 
\norm{V\inv x}^2 &= \norm{V^T ( V V^T)^{-1} x}^2\\
 &= x^T (V^T ( V V^T)^{-1})^T V^T (V V^T)^{-1} x \\
 &= x^T ( V V^T)^{-T} (V V^T) (V V^T)^{-1} x \\
 &= x^T ( V V^T)^{-1}  x = \frac{1}{\det (V V^T)}\, x^T \adj( V V^T)  x \\
 &\leq \frac{\norm{ \adj( V V^T) }}{\det (V V^T)} .
\end{align*}
\end{proof}

Given $g\in\Mat_m(\R)$ denote by $\rank{g}$, 
$\Ker{g}$ and  $\Range{g}$ respectively the rank, kernel and range of $g$. Let $\pi_g\in\Mat_m(\R)$ denote the  orthogonal
projection onto $\Range{g}$.

\begin{lemma}
\label{ABC rank lemma}
If $g_2, g_1, g_0\in\Mat_m(\R)$ are such that
 $\rank{g_2 g_1}=\rank{g_1}=\rank{g_1 g_0}=k$, then
\begin{enumerate}
\item[(a)] $\rank{g_2 g_1 g_0}=k$.
\item[(b)] $g_2 g_1 g_0$ maps $\Range{ g_0^T g_1^T}$ isomorphically onto
$\Range{g_2 g_1}$.

\end{enumerate}
\end{lemma}

\begin{proof}
Because $\rank{g_2 g_1}=\rank{g_1}$ we have
$\Range{g_1^T g_2^T}=\Range{g_1^T}$. Hence
\begin{align} \nonumber
\Range{g_0^T g_1^T} &=g_0^T \Range{g_1^T}=g_0^T \Range{g_1^T g_2^T}\\
&=\Range{g_0^T g_1^T g_2^T} \label{rank id} ,
\end{align}
which implies (a). In fact
\begin{align*}
\rank{g_2 g_1 g_0}&=\rank{g_0^T g_1^T g_2^T} = \dim{\Range{g_0^T g_1^T g_2^T}}\\
&=\dim{\Range{g_0^T g_1^T}}
=\rank{g_1 g_0} = k .
\end{align*}
Conclusion~\eqref{rank id} applied to the triplet of matrices $g_0^T, g_1^T, g_2^T$ gives  
\begin{equation}
\label{rank id2}
\Range{g_2 g_1} = \Range{g_2 g_1 g_0} .
\end{equation} 
By the fundamental theorem on homomorphisms, $g_2 g_1 g_0$ induces an isomorphism between
$\Range{g_0^T g_1^T g_2^T}=\Ker{g_2 g_1 g_0}^\perp $ and $\Range{g_2 g_1 g_0}$. 
Therefore item (b) follows from~\eqref{rank id} and~\eqref{rank id2}.

\end{proof}

%%%%%%

\subsection*{The rank of a linear cocycle.} Let $A \in \qpcmat{d}$ be an analytic quasi-periodic cocycle. 

\begin{definition}
We call  {\em geometric rank} of $A$ the number
$${\rm r} (A):=\max_{x\in\T^d}  \, \rank{A(x)} .$$
 We call {\em rank} of $A$  the limit
 $$ \rank{A}:= \lim_{n\to+\infty} {\rm r}(\An{n}).$$
When $\rank{A}=m$  the cocycle $A$ is said to  have maximal rank.
\end{definition}

\begin{remark} The limit above exists because the sequence of geometric ranks 
${\rm r} (\An{n})$  decreases.
Hence this sequence eventually stabilizes, i.e. it becomes constant.
\end{remark}

\begin{remark}  The reader should be mindful of the notational difference between $\rank{A}$, which represents the rank of the cocycle $A$,
and $\rank{A(x)}$, the rank of the matrix $A(x)$.
\end{remark}

\begin{remark} Because of analyticity,
$\{x\in\T^d\colon \rank{A(x)} = {\rm r} (A) \}$
is   open and has full measure. Hence, this is a generic set
in both topological and measure theoretical sense.
 \end{remark}

\begin{proposition} 
If ${\rm r}(\An{k}) = {\rm r}(\An{k+1})$
then $\rank{A}={\rm r}(\An{k})$.
\end{proposition}
 
\begin{proof}
Given $x\in \T^d$ consider the triplet of matrices  $A(T^{k+1}x)$, $\An{k}(T x)$ and $A(x)$,
whose product is equal to $\An{k+2}(x)$.
If $x$ is generic the  matrices $\An{k}(x)$, $\An{k+1}(x)$, $\An{k}(T x)$ and $\An{k+1}(T x)$ have maximal rank.
Then by assumption
\begin{align*}
& \rank{A(T^{k+1}x) \An{k}(T x)}  = \rank{\An{k+1}(T x)} =  \rank{\An{k}(T x)} ,\\
& \rank{\An{k}(T x) A(x)} = \rank{\An{k+1}(x)} =  \rank{\An{k}(x)} .
\end{align*}
Hence by Lemma~\ref{ABC rank lemma},
$$ \rank{\An{k+2}(x)} =
\rank{A(T^{k+1}x) \An{k}(T x) A(x)} = \rank{\An{k}(x)} . $$
This proves that ${\rm r}(\An{k+2})={\rm r}(\An{k})$.
By induction ${\rm r}(\An{n})={\rm r}(\An{k})$ for all $n\geq k$.
\end{proof}

\begin{corollary} The  geometric rank ${\rm r}(\An{k})$  stabilizes after some order  $k\leq m$.
 In particular, $\An{m}$ has geometric rank equal to $\rank{A}$ and so does $\An{n}$ for any $n \ge m$.
\end{corollary}

\begin{definition}
A cocycle $A\in \qpcmat{d}$ is said to be {\em nilpotent}
when for some $n\geq 1$, $\An{n}\equiv 0$.
\end{definition}

\begin{remark} \label{rmk zero rank}
The only cocycle with zero geometric rank is the constant zero cocycle. 
The cocycles with zero rank are exactly the nilpotent ones.
 \end{remark}

\begin{remark}
If $\rank{A}=k$ then by the analyticity of $A$ the set of points $x\in\T^d$ where $\rank{\An{m}(x)} = k$ 
is open and dense with full measure. 
 \end{remark}

The following elementary proposition  relates the geometric rank with singular values and exterior powers of a cocycle $A$.

\begin{proposition}
Given  $A\in \qpcmat{d}$ and $k\geq 1$ 
the following statements are equivalent:
\begin{enumerate}
\item $A$  has geometric rank $k$,
\item $k=\max\{ \, 1\leq j\leq m\,:\,
s_j(A(x))\neq 0 \; \text{for some }\, x\in\T^d\;\}$,
\item $\wedge_k A$  has geometric rank $1$.
\end{enumerate}
\end{proposition}

\begin{proof}
Given $g\in\Mat_m(\R)$,
$\rank{g}=\max\{ 1\leq j \leq m\colon s_j(g)>0 \}$.
This implies that (1)$\Leftrightarrow$(2).

On the other hand,
\begin{align*}
s_1(\wedge_k g) &= s_1(g)\, s_2(g)\,\cdots\,
s_k(g),\\
s_2(\wedge_k g) &= s_1(g) \, s_2(g) \,\cdots\,
s_{k-1}(g)\,s_{k+1}(g) .
\end{align*}
Hence, $s_1(g)\geq \ldots \geq s_k(g)>s_{k+1}(g)=0$ if and only if
$s_1(\wedge_k g)>s_2(\wedge_k g)=0$.
This shows that (1)$\Leftrightarrow$(3).
\end{proof}

%%%%%
 
\subsection*{Semi-conjugation to a maximal rank cocycle.}
Consider a linear cocycle  $A\in\qpcmat{d}$  with
rank $k\ge1$ and a function $V_A:\T^d\to \Mat_{k\times m}(\R)$ 
consisting of $k$ rows of $\An{m}$  chosen so that
 $$ \Omega_A:=\{ x\in\T^d\colon \,
\rank{V_A(x)} = k \,  \} $$
is an open set with  full measure in $\T^d$.
Note that  
\begin{equation}
\label{OmegaA:def}
\rank{V_A(x)} = k\quad \Leftrightarrow\quad
\det(V_A(x)\, V_A(x)^T)\neq 0 .
\end{equation} 
Let  $V_A^+(x)$ be the pseudo-inverse of $V_A(x)$,
which is well-defined for $x\in\Omega_A$. By Lemma~\ref{pseudo-inverse:char}, for every $x\in\Omega_A$,
\begin{equation}
\label{VA+}
 V_A^+(x) = V_A(x)^T\, \left( V_A(x)\, V_A(x)^{T}\right)^{-1} .
\end{equation}

Define $W_A(x)=\Range{ \An{m}(x)^T}=\Range{V_A(x)^T} =\Range{V_A^+(x)}$  
and  $P_A(x)$  as the orthogonal projection onto $W_A(x)$,
i.e., $P_A(x)= \pi_{\An{m}(x)^T }$.
By definition, for all $x\in\Omega_A$,
%% \item $W_A(x)= \Range{A(x)^T}$,
\begin{align}
 \label{PA=(VA+)VA}
& P_A(x) = V_A^+(x)\, V_A(x) , \\
\label{Rm=WA oplus Ker}
 & \R^m = W_A(x)\oplus^ \perp \Ker{\An{m}(x)} ,\\
 \label{APA=A}
 & \An{n}(x)\,P_A(x) = \An{n}(x)\quad \text{ for all }\;  n\geq m .
\end{align}
For each $x\in\T^d$ let $\delta_{A,i}(x)$ be the $i$-th 
leading principal minor of $V_A (x)\, V_A (x)^ T$, i.e.,  the sub-determinant corresponding to the first $i$ rows and columns of $V_A (x)\, V_A (x)^ T$,
and then set
\begin{align}\label{gA:def}
g_A(x) &:=  \prod_{j=1}^k \delta_{A,j}(x)\;, \\
\label{h:def}
h_A(x) &:=  \delta_{A,k}(x)\,g_A(\transl x)^3\\
\label{Atilde:def}
\widetilde{A}(x) &:= h_A(x) \,P_A(\transl\,x)\, A(x) \;.
\end{align}

Given $n \ge m$, since
$\Range{\An{n} (x)^T}\subseteq \Range{\An{m} (x)^T}$ 
and   the matrices $\An{n} (x)$ and $\An{m} (x)$
have the same rank  for all $x \in  \cap_{n\ge 0} T^{-n} \Omega_A$ it follows that for a.e. phase $x\in\T^d$
\begin{equation}
\label{WA:RAT:KerAperp}
W_A (x) = \Range{\An{n} (x)^T} = \Ker{\An{n}(x)}^\perp .
\end{equation}

Notice  that $V_A$, $g_{A}$  and $h_A$ are analytic functions over the same analytic domain $\strip_r^d$ as $A$.

The functions $g_{A}$  and $h_A$  are not identically zero. Indeed, by Sylvester's criterion for positive definiteness and the fact that $V_A$  has rank $k\geq 1$, we have
 $\delta_{A,i}(x)>0$ for all $x\in\Omega_A$.
 This implies that $g_A(x)>0$ and $h_A(x)>0$ for all
 $x\in \Omega_A\cup T^{-1}\Omega_A$.

 Finally the following invariance relation holds.
\begin{proposition}
\label{prop Atil:WA->WAT}
For all $x\in \Omega_A\cap T^{-1}\Omega_A$,
\begin{equation}\label{Atilde:inv}
\widetilde{A}(x)\,W_A(x)=W_A(\transl x)\;.
\end{equation}
\end{proposition}

\begin{proof}
Given $x\in\Omega_A\cap T^{-1}\Omega_A$,
consider the  matrices
$A(T^{m+1} x)$, $\An{m}(T x)$ and $A(x)$, whose product
$\An{m+2}(x)$ has rank $k=\rank{A}$.
By~\eqref{APA=A} above we have
\begin{equation}
\label{APA}
\An{m+2}(x)= \An{m+1}(T x) P_A(T x) A(x) .
\end{equation}
By item (b) of Lemma~\ref{ABC rank lemma} applied to the previous triplet of matrices,
$\An{m+2}(x)$ maps $W_A(x)$ isomorphically
onto $\Range{\An{m+1}(T x)}$ and, similarly, the map
$\An{m+1}(T x)$ takes $W_A(T x)$ isomorphically
onto $\Range{\An{m+1}(T x)}$. Therefore the factorization~\eqref{APA}  
implies that the linear map $P_A(T x) A(x)$ sends $W_A(x)$ isomorphically onto $W_A(T x)$. Finally, since $h_A$ does not vanish  over $\Omega_A\cap T^{-1}\Omega_A$, $\widetilde A(x)$ induces an isomorphism between $W_A(x)$ and $W_A(T x)$.
\end{proof}

\begin{definition} 
\label{RA:def}
 The reduced cocycle of $A$
is the natural extension to $\T^d$ of the mapping
$R_A:\Omega_A  \to \Mat_k(\R)$, 
\begin{equation}\label{RA:def:eq}
R_A(x) :=V_A(T x)\,\widetilde{A}(x)\, V_A^+(x)\;.
\end{equation}
\end{definition}

The next lemma ensures that $R_A$  extends analytically to $\strip_r^d$.

\begin{lemma} Assuming $\rank{A}=k\geq 1$,  the following functions have analytic extensions to $\strip_r^d$:
\begin{enumerate}
\item $g_A, h_A:\T^d\to\R$,
\item $(g_A)^ 3\,P_A:\T^d\to\Mat_m(\R)$,
\item $\widetilde{A}:\T^d\to\Mat_m(\R)$,
\item $R_A:\T^d\to\Mat_k(\R)$.
\end{enumerate}
Moreover, $R_A$ has maximal rank, i.e.,
$\det(R_A) \not \equiv 0$.
\end{lemma}

\begin{proof}
Clearly $\delta_{A,i}$, $g_A$, $h_A$ and $V_A$  are  analytic over $\strip_r^d$. 

Consider the QR-decomposition for $V_A(x)^T$. Transposing it we get
$V_A(x)=U(x)\,K(x)$, where
$U(x)\in \Mat_k(\R)$ is a lower triangular matrix with $1$s along the diagonal,
and $K(x)\in \Mat_{k\times m}(\R)$ is orthogonal,
in the sense that the rows  $\kappa_1(x),\ldots, \kappa_k(x)$ of $K(x)$  are mutually orthogonal. In fact
 $K(x)  K(x)^T$ is the square of the following positive diagonal matrix $D(x)=\diag(\norm{\kappa_1(x)},\ldots, \norm{\kappa_k(x)})$.
$K(x)$ is obtained applying the Gram-Schmidt orthogonalization process to the rows of $V_A(x)$.
In each step the row $v_j(x)$ of $V_A(x)$ gives rise to the row $\kappa_j(x)$ of 
the orthogonal matrix $K(x)$
$$ \kappa_j(x):= v_j(x) - \sum_{i<j} \frac{\langle v_j(x), \kappa_i(x)\rangle}{\| \kappa_i(x)\|^ 2}\, k_i(x) \;.$$
Inductively it follows that for each $j=1,\ldots, k$,
$$\norm{\kappa_1(x)}^2\ldots \norm{\kappa_{j-1}(x)}^2 \kappa_{j}(x)\quad \text{ and }
\quad \norm{\kappa_1(x)}^2\ldots \norm{\kappa_{j}(x)}^2 $$
are polynomial functions of the components of the vectors  $v_1(x),\ldots, v_j(x)$, and in particular they are analytic functions on $\strip_r^d$.
Since $\det U(x)=1$,
\begin{align*}
\delta_{A,k}(x) &=  \det\left[ V_A(x)\, V_A(x)^T \right] = 
\det \left[ U(x) K(x) K(x)^T U(x)^ T \right]\\
&= \det \left[  K(x)  K(x)^ T \right] =  \det\left[ D(x)^2\right] =\prod_{j=1}^ k \norm{\kappa_j(x)}^{2}\;.
\end{align*}

Denoting by $\submatr{M}{i}$ the sub-matrix of $M\in\Mat_k(\R)$ having order $i$ and 
both indices in $\{1,\ldots, i\}$, since $U(x)$ is a lower triangular matrix, we have
$$ \submatr{U(x)\,M\,U(x)^T}{i} = 
\submatr{U(x)}{i}\,
\submatr{M}{i}\,
\submatr{U(x)}{i}^T .$$
Hence,  as above we obtain
$$\delta_{A,i}(x) = 
\det \submatr{ U(x) K(x) K(x)^T U(x)^ T }{i}
 = \norm{\kappa_1(x)}^{2}\ldots \norm{\kappa_i(x)}^{2}.$$
Thus  $g_A\,D^{-2}$ and $g_A\,K$   have analytic extensions to $\strip_r^d$.
The rows of $D(x)^{-1} K(x)$ form an orthonormal basis of $W_A(x)$,
and the orthogonal projection onto $W_A(x)$ is given by
\begin{align*}
P_A(x) &= (D(x)^{-1} K(x))^T (D(x)^{-1} K(x))\\
& = K(x)^ T D(x)^ {-2} K(x)\;.
\end{align*}
It follows that  $(g_A)^3\,P_A$  has an analytic extension to $\strip_r^d$. Thus, 
also  $(\delta_{A,k})^{-1} \widetilde{A} = (g_A\circ \transl)^3\,(P_A\circ \transl)\, A$ has an analytic extension to $\strip_r^d$.
From~\eqref{RA:def} and~\eqref{VA+} we obtain
\begin{align*}
R_A(x)&= V_A(\transl x)\,\widetilde{A}(x)\,V_A(x)^T\, \left(\, V_A(x)\, V_A(x)^T\, \right)^{-1}\\
&= \frac{1}{\det \left[ V_A(x)\, V_A(x)^T\, \right] }\, V_A(\transl x)\,\widetilde{A}(x)\,V_A(x)^T\, {\rm adj}\left( V_A(x)\,V_A(x)^T\, \right) \\
&=  V_A(\transl x)\,\left((\delta_{A,k})^{-1}(x) \widetilde{A}(x)\right)\,V_A(x)^T\, {\rm adj}\left( V_A(x)\, V_A(x)^T\, \right) .
\end{align*}
This formula proves that $R_A$ has an analytic extension to $\strip_r^d$.

Finally, in Proposition~\ref{prop Atil:WA->WAT}
we have seen that $\widetilde{A} (x)$ induces an isomorphism between
$W_A(x)$ and $W_A(T x)$, for all $x\in\Omega_A\cap T^{-1}\Omega_A$.  But because $R_A (x) = V_A (T x)  \, \widetilde{A} (x) \, V_A\inv (x)$ for all $x \in \Omega_A$, by the definitions of $V_A (T x)$ and $V_A\inv (x)$, we conclude that $R_A (x)$ induces an automorphism on $\R^k$.
Hence, for all these phases  $\det R_A(x)\neq 0$.

\end{proof}

Let $h_{A}^{(n)}(x):=  \displaystyle \prod_{i=0}^{n-1} h_A(T^i  x)$ be the $n$-th iterate of the one-dimensional cocycle $h_A \in \analyticf{d}$. As mentioned earlier, $h_A \not\equiv 0$. The reduced cocycle $R_A \in \qpcmatk{d}$ was shown above to be non-identically singular. The next lemma relates the iterates of the identically singular cocycle $A$ to those of the non-identically singular cocycles $R_A$ and $h_A$.

\begin{lemma}\label{semi-conjugation} 
Given  
$A\in C^ \omega_r(\T^d,\Mat_m(\R))$ with rank $k\geq 1$, 
the following   relations hold for all iterates $n\geq m$:
\begin{align}\label{RV=VAtil}
& R_A^{(n)}(x)\,V_A(x)  = V_A(T^n x)\, \widetilde{A}^{(n)}(x)\\
\label{Bn}
& \widetilde{A}^{(n)}(x) = h_{A}^{(n-m)}(x)\,
\widetilde{A}^{(m)}(T^{n-m} x) \, \An{n-m}(x) \\
\label{An+1}
&  \An{n}(x)  =  	h_{A}^{(n-m)}(x)^{-1} \, \An{m}(T^{n-m}x) \, \widetilde{A}^{(n-m)}(x)\;.
\end{align}
\end{lemma}

\begin{remark}
Since $m$ is fixed, the factors $\widetilde{A}^{(m)}(T^{n-m} x)$
and $\An{m}(T^{n-m} x)$ will essentially be negligible in future estimates. 
\end{remark}

\begin{proof}
Since the functions $V_A$, $R_A$ and $\widetilde{A}$ are analytic
it is enough to check~\eqref{RV=VAtil}  for most phases $x\in\T^d$.

Consider the set of good phases 
$$G_A := \bigcap_{n \ge 1} T^{-n} \Omega_A \cap \bigcap_{n\ge1} T^{-n} \{h_A\neq 0\} .$$

Then  $G_A$ has full  measure, and 
the previous relations
~\eqref{PA=(VA+)VA},
~\eqref{APA=A},
~\eqref{WA:RAT:KerAperp} and~\eqref{Atilde:inv} hold for all iterates $T^n x$ and any phase $x\in G_A$.

From~\eqref{RA:def:eq} and~\eqref{PA=(VA+)VA}  we get
$$ R_A(x)\, V_A(x)= V_A(T x)\, \widetilde{A}(x)\, P_A(x) =
 V_A(T x)\, B(x)  $$
 where  $B=  \widetilde{A} \, P_A$. Iterating this relation we get for all $n\geq 1$,
 $$ R_A^{(n)}(x)\, V_A(x)= V_A(T^n x)\, B^{(n)}(x) . $$
 
Moreover, since  $B=  \widetilde{A} \, P_A$ and $(P_A\circ T) \widetilde{A}=\widetilde{A}$,
$$ B^{(n)}(x)= \widetilde{A}^{(n)}(x)\,P_A(x) . $$

From~\eqref{WA:RAT:KerAperp}   and iterating  Proposition~\ref{prop Atil:WA->WAT}, for all $n\geq m$,
$$\Range{\An{n}(x)^T}=W_A(x)=\Range{\widetilde{A}^{(n)}(x)^T}. $$

This implies that
$\Ker{\An{n}(x)}= \Ker{\widetilde{A}^{(n)}(x) }$.
By~\eqref{APA=A} we have
$\An{n}(x)(I-P_A(x))=0$.
If $n\geq m$ then, since $I-P_A(x)$ is the orthogonal projection onto $\Ker{\An{n}(x)}=\Ker{\widetilde{A}^{(n)}(x)}$, we also must have
$\widetilde{A}^{(n)}(x)\,(I-P_A(x))=0$.  
Hence for all $n\geq m$ 
$$B^{(n)}(x)= \widetilde{A}^{(n)}(x)\,P_A(x) = \widetilde{A}^{(n)}(x) . $$
This proves~\eqref{RV=VAtil}.

Moreover
$$ \widetilde{A}^{(n+1)}(x) = h_A(x)\,\widetilde{A}^{(n)}(T x)\, P_A(T x)\,A(x)= h_A(x)\,\widetilde{A}^{(n)}(T x)\, A(x) .$$
By induction, we get for all $n\geq m$   
$$ \widetilde{A}^{(n)}(x) = h_{A}^{(n-m)}(x)\,\widetilde{A}^{(m)}(T^{n-m} x)\, \An{n-m}(x) . $$
This proves~\eqref{Bn}.

Let us now write $M_{n,m}(x) := \widetilde{A}^{(m)}(T^{n-m} x)$.
From~\eqref{Bn}  we get
\begin{equation}
\label{Mnm}
 M_{n,m}(x)\, \widetilde{A}^{(n-m)}(x) = h_{A}^{(n-m)}(x)\, M_{n,m}(x)\, \An{n-m}(x) . 
\end{equation}
The matrix $ M_{n,m}(x)$ can be expressed as 
$$ M_{n,m}(x)  = V_A^+(T^n x)\, R_A^{(m)}(T^{n-m} x)\, V_A(T^{n-m} x), $$
In fact,~\eqref{RV=VAtil} implies that
$$R_A^{(m)}(T^{n-m} x) \, V_A(T^{n-m} x) = V_A (T^n x) \, \widetilde{A}^{(m)}(T^{n-m} x) .$$

Multiplying on the left by $V_A\inv (T^n x)$, and using  that
 $V_A\inv \, V_A = P_A$ and $(P_A \circ T) \widetilde{A} = \widetilde{A}$
we get

\begin{align*}
V_A^+(T^n x)\, R_A^{(m)}(T^{n-m} x)\, V_A(T^{n-m} x) &= 
 P_A (T^n x) \widetilde{A}^{(m)}(T^{n-m} x) \\
 & = P_A (T^n x) \, \widetilde{A} (T^{n-1} x) \, \widetilde{A}^{(m-1)}(T^{n-m} x) \\
& = \widetilde{A} (T^{n-1} x) \, \widetilde{A}^{(m-1)}(T^{n-m} x) \\
& = \widetilde{A}^{(m)}(T^{n-m} x) = M_{n,m}(x) .
\end{align*}
Let\footnote{Although this is not important here, the matrix $M_{n, m}\inv (x)$ thus defined is in fact the pseudo-inverse of $M_{m, n} (x)$.}
$$ M_{n,m}^+ (x)  := V_A^+(T^{n-m} x)\, R_A^{(m)}(T^{n-m} x)^{-1}\, V_A(T^{n} x). $$

Then
$$ M_{n,m}^+ (x)\, M_{n,m} (x) = V_A^+(T^{n-m} x)\,
V_A(T^{n-m} x) = P_A(T^{n-m} x). $$
Thus, left multiplying ~\eqref{Mnm}
by $M_{n,m}^+ (x)$ we get 
\begin{align*}
P_A(T^{n-m} x)\, \An{n-m}(x)
& =h_{A}^{(n-m)}(x)^{-1}\, P_A(T^{n-m} x)\, \widetilde{A}^{(n-m)}(x)   \\
& =h_{A}^{(n-m)}(x)^{-1}\, \widetilde{A}^{(n-m)}(x) .
\end{align*}  
Finally,  left multiplying this equality 
by $\An{m}(T^{n-m}  x)$ we get
$$ \An{m}(T^{n-m} x)\,P_A(T^{n-m} x)\, \An{n-m}(x)
= h_{A}^{(n-m)}(x)^{-1}\, \An{m}(T^{n-m} x)\,  \widetilde{A}^{(n-m)}(x) .$$
Since $\An{m}(T^{n-m} x)\,(I-P_A(T^{n-m} x))=0$
it follows that  
\begin{align*}
\An{n}(x) & =
\An{m}(T^{n-m} x)\, \An{n-m}(x)\\
& =  \An{m}(T^{n-m} x)\, P_A(T^{n-m} x)\, \, \An{n-m}(x) \\
&= h_{A}^{(n-m)}(x)^{-1}\, \An{m}(T^{n-m} x)\,  \widetilde{A}^{(n-m)}(x)  .
\end{align*}
This proves~\eqref{An+1}.
\end{proof}

%%%%%

\subsection*{$L^2$-boundedness and the fiber LDT estimate} We use the semi-conjugation relations in Lemma~\ref{semi-conjugation} to establish non-uniform $L^2$-bounds and LDT estimates for the iterates of an analytic cocycle. A direct consequence of the $L^2$-boundedness is the fact, interesting in itself, that an analytic cocycle is nilpotent if and only if its top Lyapunov exponent is $-\infty$. This result was recently obtained independently in~\cite{Sadel-Xu}. 

\subsubsection*{A crucial tool: the {\L}ojasiewicz inequality.} We recall the following result that describes the {\em transversality} property of non-identically zero analytic functions (see for instance  Lemmas 6.1 and 6.2 in~\cite{DK-book}. This result will be used repeatedly throughout the paper. 

\begin{proposition}({\L}ojasiewicz inequality)\label{Loj-prop}
Given $f \in \analyticf{d}$ with  $f \not \equiv 0$, there are constants $S = S (f) < \infty$ and $b = b (f) > 0$ such that 
$$\abs{ \{ x \in \T^d \colon \sabs{ f (x) } <t \}  } < S \, t^b \quad \text{ for all } \, t > 0 .$$
Furthermore, there is a constant $C = C (S, b) = C (f) < \infty$ such that
$$\norm{  \log \sabs{f} }_{L^2 (\T^d)} \le C .$$
\end{proposition}

We note that this result is in fact uniform in $f$, in the sense that the constants $S, b$ and $C$ are stable under small perturbations of $f$. 
Also, the $L^2$-norm may be replaced by any other $L^p$-norm, $1\le p<\infty$. 

\medskip

We begin with some preparatory estimates that use in an essential way the {\L}ojasiewicz inequality.

\begin{lemma}\label{Loj-pseudo-inverse:lemma}
Let $V \in \qpcmatkm{d}$ with $\rank{V (x)} = k$ a.e. Then there are constants $S = S (V) < \infty$ and $c = c (V) > 0$ such that for all $ t > 0$ we have
\begin{align}
\babs{ \{ x \in \T^d \colon \norm{V\inv (x)} > \frac{1}{t} \}   } & < S \, t^c  \label{Loj-eq1}\\
\babs{ \{ x \in \T^d \colon \norm{V (x)} < t  \}   } & < S \, t^c \label{Loj-eq2}
\end{align}
\end{lemma}

\begin{proof}
Since $\rank{V (x)} = k$ for a.e. phase $x \in \T^d$, for those phases, by Lemma~\ref{pseudo-inv-estimate} we have
$$\norm{ V\inv (x) }^2 \le \frac{\norm{ \adj [ V(x) \, V^T (x) ] }}{ \det [ V (x) \, V^T (x) ] } \le \frac{C}{\det [ V (x) \, V^T (x) ] } $$
for some constant $C = C (\norm{V}_{L^\infty}, k, m) < \infty$.

Therefore,
\begin{equation}\label{Loj-eq10}
\{ x  \colon \norm{V\inv (x)} > \frac{1}{t} \}  \subset \{ x  \colon \abs{\det [ V (x) \, V^T (x) ] } < C \, t^2  \} \, . 
\end{equation}

The function $f (x) := \det [ V (x) \, V^T (x) ] $ is $\not \equiv 0$ on $\T^d$ (since $V(x)$ has maximal rank $k$ for a.e. $x$) and it has an analytic extension $f (z) = \det [ V (z) \, V^T (z) ] $ to $\strip_r^d$. The {\L}ojasiewicz inequality is then applicable to $f (x)$ and it implies 
$$\babs{ \{ x \in \T^d \colon \abs{f(x)} \le C \, t^2 \} } \le S \, t^c$$
for some constants $S = S (f) < \infty$ and $c = c (f) > 0$.

Estimate \eqref{Loj-eq1} then follows from \eqref{Loj-eq10}.

By the first inequality in Lemma~\ref{pseudo-inv-estimate}, for a.e. $x \in \T^d$,
$$\frac{1}{\norm{V\inv (x)}} \le \norm{V (x) }  \,,$$
hence
$$ \{ x \in \T^d \colon \norm{V (x)} < t  \}   \subset \{ x  \in \T^d \colon \norm{V\inv (x)} > \frac{1}{t} \}   \, . $$ 

Then \eqref{Loj-eq2} follows from \eqref{Loj-eq1}.
\end{proof}

\begin{corollary}\label{Loj-pseudo-inverse:cor}
Under the same assumptions as in the previous lemma, there are constants  $S = S (V) < \infty$, $c = c (V) > 0$ and $T_0 = T_0 (V) < \infty$ such that for all $T \ge T_0$ we have
\begin{align}
\babs{ \{ x \in \T^d \colon \abs{ \log \norm{V\inv (x)} } > T \}   } & < S \, e^{- c T}  \label{Loj-eq3}\\
\babs{ \{ x \in \T^d \colon \abs{ \log \norm{V  (x)} } > T \}   } & < S \, e^{- c T} \label{Loj-eq4}
\end{align}
In particular we have that $\log \norm{V\inv}, \, \log \norm{V} \in L^2 (\T^d)$.
\end{corollary}

\begin{proof}
Since for a.e. $x \in \T^d$ we have
$$\frac{1}{\norm{V (x)}} \le \norm{V\inv (x) }  \,,$$
it follows that
$$\log \norm{V\inv (x) } \ge - \log \norm{V (x)} \ge - \log \norm{V}_{L^\infty} =: - T_0 \, .$$

Then if $T \ge T_0$,
$$ \{ x  \colon \abs{ \log \norm{V\inv (x)} } > T \}  \subset  \{ x  \colon \log \norm{V\inv (x)}  > T \}  = \{ x \colon \norm{V\inv (x)} > e^T \} \, ,$$
and \eqref{Loj-eq3} follows from \eqref{Loj-eq1} after the change of variables $t := e^{-T}$. 

Estimate \eqref{Loj-eq4} follows in a similar way from \eqref{Loj-eq2}. The $L^2$ bounds are  consequences of \eqref{Loj-eq3}, \eqref{Loj-eq4} and the fact that
$$\int_{\T^d} \abs{\phi (x) } \, d x = \int_0^\infty \, \abs{ \{ x \in \T^d \colon \abs{\phi (x) } > T \} } \, d T \, ,$$
which holds for every measurable function $\phi$.
\end{proof}

\medskip

Given any cocycle $A \in \qpcmat{d}$ of any dimension $m\ge1$, we will use the notations $ A^{(n)}(z):= \,A(T^{n-1} z) \, \ldots \, A(T z) \, A(z)$,
$$\un{n}{A} (z) :=   \frac{1}{n} \log \norm{\An{n} (z)} \ \text{and} \ \avg{\un{n}{A}} = \LE{n}_1 (A) := \int_{\T^d}  \frac{1}{n} \log \norm{\An{n} (x)} \, d x .$$

By the analyticity of $A$, clearly  $\un{n}{A} (z)$ is separately subharmonic on $\strip_r^d$, and if $\An{n}  \not \equiv 0$ then $\un{n}{A} \not \equiv - \infty$.

The following statements relate the  functions $\un{n}{A}$ associated to the iterates of a cocycle $A$, to the corresponding functions $\un{n}{R_A}$ and $\un{n}{h_A}$ a\-sso\-ci\-ated to a maximal rank cocycle $R_A$ and to a non-identically singular one-dimensional cocycle $h_A$.

\begin{proposition}\label{reduction-prop}
Let $A \in \qpcmat{d}$ be a quasi-periodic cocycle with   $\rank{A}=k \ge 1$. There are a cocycle $R_A \in \qpcmatk{d}$ with $\det [R_A (x) ] \not \equiv 0$, a one dimensional cocycle $h_A \in \analyticf{d}$ with $h_A (x) \not \equiv 0$, a function $V_A \in \qpcmatkm{d}$ with $\rank{V_A (x)} = k$ a.e. and a constant $\ldtmeas_A < \infty$ such that for all phases $x \in \T^d$ and for all iterates $n \ge m$ we have:
\begin{align*}
  \un{n}{A} (x) & \geq  \frac{n+m}{n} \, \un{n+m}{R_A} (x) - \un{n}{h_A} (x)  - \frac{\ldtmeas_A}{n} - \frac{1}{n} \, \log \norm{V_A\inv (x) }   \\
  \un{n}{A} (x) & \leq    \frac{n-m}{n} \, \left[ \un{n-m}{R_A} (x) - \un{n-m}{h_A} (x) \right] + \frac{\ldtmeas_A}{n} + \frac{1}{n} \, \log \norm{V_A\inv (T^{n} x) }  .
\end{align*}
\end{proposition}

\begin{proof}
Recall that $V_A^+ \, V_A=P_A$ by~\eqref{PA=(VA+)VA}
and $V_A\, V_A^+ = {\rm id}_{\R^k}$ by~\eqref{VA+}.
Hence, the semi-conjugation~\eqref{RV=VAtil} implies that
\begin{equation*}
 R_A^{(n)}(x) = V_A(T^n x)\, \widetilde{A}^{(n)}(x)\,V_A^+(x) .
\end{equation*}
Therefore, taking norms and logarithms and using that $V_A$ is bounded,
\begin{equation}
\label{uR<=uAtil}
\un{n}{R_A}(x) \leq \un{n}{\widetilde{A}}(x) + \frac{\ldtmeas_A}{n} + \frac{1}{n} \, \log \norm{V_A\inv(x) }    . 
\end{equation}

By definition~\eqref{Atilde:def} we have 
$P_A (T x) \, \widetilde{A} (x) = \widetilde{A} (x)$, hence iterating, $P_A (T^n x) \,  \widetilde{A}^{(n)} (x) =  \widetilde{A}^{(n)} (x) $.
Multiplying  the semi-conjugation~\eqref{RV=VAtil}  
   on the left by $V_A^+ (T^n x)$ we get  
$$ \widetilde{A}^{(n)}(x)
= V_A^+(T^n x)\, R_A^{(n)}(x)\, V_A(x) . $$

Therefore, taking  norms and logarithms in this relation
\begin{equation}
\label{uAtil<=uR}
\un{n}{\widetilde{A}}(x) \leq \un{n}{R_A}(x) + \frac{\ldtmeas_A}{n} + \frac{1}{n} \, \log \norm{V_A\inv(T^n x) }    . 
\end{equation}

Taking norms and logarithms in~\eqref{Bn} and using that $\An{m}$ is bounded,
\begin{equation}
\label{uA>=uAtil}
\un{n}{A}(x) \geq  \frac{n+m}{n} \, \un{n+m}{\widetilde{A}}(x) -\un{n}{h_A}(x) - \frac{\ldtmeas_A}{n}  . 
\end{equation}

Similarly, taking norms and logarithms in~\eqref{An+1}
\begin{equation}
\label{uA<=uAtil}
\un{n}{A}(x) \leq   \frac{n-m}{n}\, \left[ \un{n-m}{\widetilde{A}}(x) -\un{n-m}{h_A}(x)\right] + \frac{\ldtmeas_A}{n}  . 
\end{equation}

Combining~\eqref{uR<=uAtil} and~\eqref{uA>=uAtil} 
we derive the first inequality.
Combining~\eqref{uAtil<=uR} with~\eqref{uA<=uAtil} 
we obtain the second inequality.
\end{proof}

\begin{corollary}\label{non-unif-L2bound}
Every quasi-periodic cocycle  $A \in \qpcmat{d}$ with $\rank A = k\geq 1$ is $L^2$-bounded, in the sense that there is $\ldtmeas_{0} = \ldtmeas_{0} (A) < \infty$ such that for all iterates $n \ge 1$ we have
$$\norm{ \un{n}{A} }_{L^2 (\T^d)} \le \ldtmeas_{0} \, .$$
\end{corollary}

\begin{proof}
This statement was already proven for {\em non-identically singular} cocycles (see Proposition 6.3 in \cite{DK-book}. Therefore it applies to the maximal rank cocycle $R_A$ and to the one-dimensional cocycle $h_A$ and we have
$$\norm{ \un{n}{R_A} }_{L^2 (\T^d)} \le \ldtmeas_{0} \quad \text{and}   \quad  \norm{ \un{n}{h_A} }_{L^2 (\T^d)} \le \ldtmeas_{0} \, ,$$
for some $\ldtmeas_{0} = \ldtmeas_{0} (A) < \infty$ and  for all $n \ge m$.

Moreover, by Corollary~\ref{Loj-pseudo-inverse:cor}, $\log \norm{V_A\inv} \in L^2$ as well, so by the double estimate in Proposition~\ref{reduction-prop}, for all $n \ge m$ we have $\un{n}{A} \in L^2$ with a uniform bound in $n$.

It remains to show that $\un{n}{A} \in L^2 (\T^d)$ also for  $1 \le n < m$.

Since $\rank{A} \ge 1$, for all $n\ge1$ we have $\An{n}  (x) \not \equiv 0$. 
Let $M \in \qpcmat{d}$ refer to any of the iterates $\An{n}$ with $1 \le n < m$. It is then enough to show that if $M (x) \not\equiv 0$ then $\log \norm{M (x)} \in  L^2 (\T^d)$. At least one of the entries of $M(x)$, let us denote it by $m (x)$, must be non-identically zero. Then for all $x\in\T^d$ we have
$$\log \sabs{ m (x) }  \le \log \norm{M (x)} \le C ,$$
for some finite constant $C$.

But since $m \in \analyticf{d}$ and $m (x) \not\equiv 0$, by say {\L}ojasiewicz ine\-qua\-li\-ty, $\log \sabs{ m (x) } \in L^2 (\T^d)$, which completes the proof.
\end{proof}

%The next equivalences were proven independently by C. Sadel and D. Xu~\cite{Sadel-Xu}.

\begin{corollary}
\label{coro zero rank}
Given a quasi-periodic cocycle $A \in \qpcmat{d}$ the following  are equivalent:
\begin{enumerate}
\item $\rank{A}=0$,
\item $A$ is nilpotent,
\item $L_1(A)=-\infty$.
\end{enumerate}
\end{corollary}

\begin{proof}
$(1)\Leftrightarrow(2)$ follows by Remark~\ref{rmk zero rank}. The implication
$(2)\Rightarrow (3)$ is obvious.
Finally $(3)\Rightarrow (1)$ follows from Corollary~\ref{non-unif-L2bound}
by contraposition.
\end{proof}

\begin{proposition}\label{reduction-thm}
Let $A \in \qpcmat{d}$ be a quasi-periodic cocycle with $\rank{A} =: k \ge 1$ and fix any $a \in (0, 1)$. There are a reduced (maximal rank) cocycle $R_A \in \qpcmatk{d}$ with $\det [R_A (x) ] \not \equiv 0$, a one dimensional cocycle $h_A \in \analyticf{d}$ with $h_A (x) \not \equiv 0$ and constants $\ldtmeas_{0} = \ldtmeas_{0} (A) < \infty$, $S = S (A) < \infty$  such that for all phases $x \in \T^d$ and all iterates $n \ge m^{6/a}$,
$$ \un{n}{A} (x) = \un{n}{R_A} (x) - \un{n}{h_A} (x) + r_n (x) ,$$
where the remainder function $r_n (x)$ has the following properties: 
\begin{itemize}
\item[(i)] $\abs{ \{ x \in \T^d \colon \abs{r_n (x)   } > \ldtmeas_{0} \, n^{-a/3}  \} }  < S \, e^{- n^{1-a}} $,

\item[(ii)] $\norm{r_n}_{L^2 (\T^d)}  \le \ldtmeas_{0} \, n^{-a/3}$.
\end{itemize}

\end{proposition}

\begin{proof} 
We estimate the remainder function
$$r_n (x) :=  \un{n}{A} (x) - \un{n}{R_A} (x) + \un{n}{h_A} (x)$$
from below and above using  Proposition~\ref{reduction-prop}. After some simple algebraic manipulations we have:
\begin{subequations}\label{eqstar}
\begin{align}
r_n (x) & \ge  \left[ \un{n+m}{R_A} (x) - \un{n}{R_A} (x) \right] + \frac{m}{n} \, \un{n+m}{R_A} (x)  - \frac{1}{n} \, \log \norm{V_A\inv (x) } - \frac{\ldtmeas_A}{n}\\
r_n (x) & \le  \left[ \un{n-m}{R_A} (x) - \un{n}{R_A} (x) \right]  - \left[ \un{n-m}{h_A} (x) - \un{n}{h_A} (x) \right]  \notag\\
&            -  \frac{m}{n} \, \un{n-m}{R_A} (x) +  \frac{m}{n} \, \un{n-m}{h_A} (x) + \frac{1}{n} \, \log \norm{V_A\inv (T^{n} x) } + \frac{\ldtmeas_A}{n} .
\end{align}
\end{subequations}

Using the triangle inequality, the two estimates on $r_n (x)$ that we have to establish follow from similar estimates on the upper and the lower bound functions above. These in turn follow from similar estimates on each of the terms between brackets as well as on the remaining terms.

\smallskip

First we prove separately that given a cocycle $R \in \qpcmatk{d}$ of any dimension $k \ge 1$, if $f_R (x) := \det [ R (x) ] \not \equiv 0$, then for phases $x$ off of a small set, the following hold: the function $\abs{\un{n}{R} (x)}$ has a sub-linear growth in $n$; the functions   $\un{n}{R} (x)$ and $\un{n-1}{R} (x)$ corresponding to consecutive iterates of $R$, differ only slightly. We then apply these facts to the $k$-dimensional cocycle $R_A$ and to the one-dimensional cocycle $h_A$.

\smallskip

Since $R$ is a non-identically singular analytic cocycle, Proposition 6.3 in~\cite{DK-book} is applicable, and it says
that the functions $\un{n}{R} (x)$ are uniformly bounded from above, and their failure to be bounded from below is captured by Birkhoff averages of the function $\log \abs{f_R}$.

More precisely,  there is $C = C (R) < \infty$ such that  for all $x \in \T^d$, 
\begin{equation}\label{bounds-bet}
 - C + \frac{1}{n} \sum_{i=0}^{n-1} \log \abs{f_R (\transl^i x)}   \le  \un{n}{R} (x)   \le C  .
 \end{equation}

Furthermore, given that $f_R$ is analytic and non-identically zero, the function $\log \abs{f_R} \in L^p (\T^d)$ for all $1\le p < \infty$. Then~\eqref{bounds-bet} implies uniform bounds in $n$ for the $L^p$-norms of $\un{n}{R}$, so we may assume that 
$$\norm{\un{n}{R}}_{L^2 (\T^d)}, \ \norm{\un{n}{R}}_{L^4 (\T^d)} \le C .$$

Applying the {\L}ojasiewicz inequality to $f_R $, for some constants $S = S(R) < \infty$, $\ldtmeas_{0} = \ldtmeas_{0} (R) > 0$ and for any given $a \in (0, 1)$ we have
$$\abs{ \{ x \colon \log \sabs{ f_R (x)  }  \le - \ldtmeas_{0} \, n^{1-a/2}   \} } = \abs{ \{ x \colon  \sabs{ f_R (x)  }  \le e^{- \ldtmeas_{0} \, n^{1-a/2}}   \} } \le S \, e^{- n^{1-a/2}}  .$$

Then if $x$ is outside a set of measure $\le n \, S \, e^{- n^{1-a/2}} < S \, e^{- n^{1-a}} $, 
$$\frac{1}{n} \sum_{i=0}^{n-1} \log \abs{f_R (\transl^i x)}  \ge  - \ldtmeas_{0} \, n^{1-a/2}  ,$$
hence using~\eqref{bounds-bet}, for these phases $x$,
\begin{equation}\label{eq12}
 \abs{ \un{n}{R} (x) } \less   \ldtmeas_{0} \, n^{1-a/2}  .
\end{equation} 

Since $$\norm{R^{(n)} (x)} = \norm{R (T^{n-1} x) \, R^{(n-1)} (x)} \le \norm{R}_{L^\infty} \, \norm{R^{(n-1)} (x)} ,$$ it follows that
$$\un{n}{R} (x) - \un{n-1}{R} (x) \le \frac{\log \norm{R}_{L^\infty}}{n} - \frac{1}{n} \, \un{n-1}{R} (x) ,$$
hence  for $x$ off of that sub-exponentially small set we have
$$\un{n}{R} (x) - \un{n-1}{R} (x)  \le \frac{\ldtmeas_{0}}{n} + \frac{\ldtmeas_{0} \, n^{1-a/2}}{n}   \less \ldtmeas_{0} \, n^{-a/2} .$$

We also have
\begin{align*}
\norm{R^{(n-1)} (x)} & = \norm{R(T^{n-1} x)^{-1} \, R^{(n)} (x)}  \le \norm{R(T^{n-1} x)^{-1} } \, \norm{ R^{(n)} (x)} \\
& = \frac{\norm{\adj  \, R(T^{n-1} x)}}{\abs{\det R(T^{n-1} x)}} \, \norm{ R^{(n)} (x)} \le \frac{ \norm{R}_{L^\infty}^{k-1} }{\abs{f_R (x) } } \, \norm{ R^{(n)} (x)} .
\end{align*}

Then for $x$ off of a sub-exponentially small set, and due to previous considerations, it follows that 
\begin{align*}
\un{n-1}{R} (x) - \un{n}{R} (x) & \le \frac{(k-1) \, \log \norm{R}_{L^\infty}}{n-1} - \frac{\log \abs{f_R (x) } }{n-1} + \frac{1}{n-1} \, \un{n}{R} (x) \\
& \less \frac{ \log \norm{R}_{L^\infty}}{n} +  \frac{\ldtmeas_{0} \, n^{1-a/2}}{n} +  \frac{\ldtmeas_{0} \, n^{1-a/2}}{n} \less \ldtmeas_{0} \, n^{- a/2} .
\end{align*}

We conclude that if $x$ is outside a set of measure $ \less S \, e^{- n^{1-a}}$, then
\begin{equation}\label{eq13}
\abs{  \un{n}{R} (x) - \un{n-1}{R} (x) } <  \ldtmeas_{0} \, n^{- a/2} .
\end{equation}

Furthermore, combining~\eqref{eq13} with the fact that $\norm{\un{n}{R}}_{L^4 (\T^d)} \le C$ and using Cauchy-Schwarz, we also obtain
\begin{equation}\label{eq14}
\norm{  \un{n}{R} - \un{n-1}{R} }_{L^2 (\T^d)} <  \ldtmeas_{0} \, n^{- a/2} .
\end{equation}

Using telescoping sums and the fact that $m \le n^{a/6}$, from \eqref{eq13} we get that off of a sub-exponentially small set,
$$\abs{  \un{n \pm m}{R} (x) - \un{n}{R} (x) } \less  \ldtmeas_{0} \, m \,  n^{- a/2} \less  \ldtmeas_{0} \, n^{- a/3} ,$$ 
and similarly,  from~\eqref{eq14} we get
$$\norm{  \un{n \pm m}{R} - \un{n}{R} }_{L^2 (\T^d)} \less  \ldtmeas_{0} \, n^{- a/3} .$$

Furthermore, \eqref{eq12} implies that off of a sub-exponentially small set,
$$\frac{m}{n} \, \abs{\un{n \pm m}{R} (x)} \less    \ldtmeas_{0} \, m \,  n^{- a/2} \less  \ldtmeas_{0} \, n^{- a/3} ,$$ 
and of course,
$$\frac{m}{n} \,\norm{  \un{n \pm m}{R}}_{L^2 (\T^d)} \le \frac{n^{a/6}}{n} \, C < \ldtmeas_{0} \, n^{- a/3} .$$

Going back to the upper and lower bounds~\eqref{eqstar} on the remainder function $r_n (x)$, applying the above with $R=R_A$ and $R=h_A$ takes care of most of the terms.

The only terms left to consider are $\frac{1}{n} \, \log \norm{V_A\inv (x) }$ and  $\frac{1}{n} \, \log \norm{V_A\inv (T^{n} x) }$. To estimate them, simply apply Corollary~\ref{Loj-pseudo-inverse:cor} with $T:= \ldtmeas_0 \, n^{1-a/3}$. This completes the proof.
\end{proof}

\begin{theorem}\label{nonunif-ldt-thm}
Given a quasi-periodic cocycle  $A \in \qpcmat{d}$ with $L_1(A)>-\infty$ and $\om \in \rm{DC}_t$, there are constants $a_0 = a_0 (d) > 0$, $b_0 = b_0 (d) > 0$  and $\nzerobar = \nzerobar (A, d, t)  \in \N$ such that if $n \ge \nzerobar$ then
\begin{equation} \label{nonunif-ldt-eq}
\babs{  \{ x \in \T^d \colon \abs{ \frac{1}{n} \log \norm{\An{n} (x)} - \LE{n}_1  (A)  } >  n^{-a_0} \} } < e^{- n^{b_0}}.
\end{equation}
 \end{theorem}
 
 \begin{proof}
We apply Proposition~\ref{reduction-thm}. There are: a (maximal rank) cocycle $R_A \in \qpcmatk{d}$ with $\det [R_A (x)] \not \equiv 0$, a  function $h_A \in \analyticf{d}$ with $h_A (x) \not \equiv 0$ and small  remainder functions $r_n (x)$ (in the sense of estimates (i) and (ii))  such that for all phases $x \in \T^d$ and all iterates $n$ after a threshold,
 \begin{equation}\label{semi-conj-eq}
 \un{n}{A} (x) = \un{n}{R_A} (x) - \un{n}{h_A} (x) + r_n (x)  .
 \end{equation}

Since $R_A \in \qpcmatk{d}$ and $\det [R_A] \not \equiv 0$, the LDT estimate for {\em non-identically singular cocycles} applies (see Theorem 6.6 in~\cite{DK-book}.

Moreover, since we may regard $h_A$ as a one-dimensional quasi-periodic cocycle, and since $h_A \not \equiv 0$, the same result also applies to $h_A$.

Applying the aforementioned LDT estimate to the non-identically singular cocycles $R_A$ and $h_A$, we conclude that there are constants $ \ldtmeas_{0} = \ldtmeas_{0} (A, r) < \infty$, $a_0 = a_0 (d) > 0$, $b_0 = b_0 (d) > 0$  and $\nzerobar = \nzerobar (A, t) \in \N$ such that if $n \ge \nzerobar$ then
\begin{align*}
\babs{  \{ x \in \T^d \colon \abs{ \un{n}{R_A} (x) - \avg{ \un{n}{R_A} }  } > \ldtmeas_{0} \, n^{-a_0} \} } < e^{- n^{b_0}}\\
\babs{  \{ x \in \T^d \colon \abs{ \un{n}{h_A} (x) - \avg{ \un{n}{h_A} }  } > \ldtmeas_{0} \, n^{-a_0} \} } < e^{- n^{b_0}} .
\end{align*}

Then for any phase $x$ outside of these two exceptional sets, by~\eqref{semi-conj-eq} we have:
\begin{align*}
\babs{ \un{n}{A} (x) - \avg{ \un{n}{A} } } & \le  \babs{ \un{n}{R_A} (x) - \avg{ \un{n}{R_A} } } \\ 
&  +  \babs{ \un{n}{h_A} (x) - \avg{ \un{n}{h_A} }  } + \abs{ r_n (x)  - \avg{ r_n} } \\
& \le \ldtmeas_{0} \, n^{-a_0} + \ldtmeas_{0} \, n^{-a_0} + \abs{r_n (x)} + \abs{ \avg{r_n} }  .
\end{align*}

Using the estimates (i) and (ii) on the remainder $r_n$ in Theorem~\ref{reduction-thm}, after further excluding another sub-exponentially small set of phases $x$, we conclude that
$$\babs{ \un{n}{A} (x) - \avg{ \un{n}{A} } } \less \ldtmeas_{0} \, n^{-a_0} + \ldtmeas_{0} \, n^{-(1-b_0)/3}  < n^{-a_0'} ,$$
for some $a_0' > 0$ and provided that $n \ge \nzerobar (A, d)$.  The theorem is now proven.
\end{proof}

%%%%%%%%%%%%%%%%%%%%%%%%%%%%%%%%%%%%%%%%%%%%%%%%%%%%%%%%%%%%%%%%
%%%%%%%%%%%%%%%%%%%%%%%%%%%%%%%%%%%%%%%%%%%%%%%%%%%%%%%%%%%%%%%%

\section{Estimates on separately subharmonic functions}\label{ssh-estimates}
\newcommand{\newu}{\widetilde{u}}
\newcommand{\umeas}{\mathscr{S}}
\newcommand{\Dif}{\mathcal{D}}

A function $u \colon \Omega \subset \C^d \to [ - \infty, \infty)$ is called {\em separately subharmonic} if it is continuous and subharmonic in each variable. Given an analytic cocycle $A \in \qpcmat{d}$, the functions $\un{n}{A} (z) := \frac{1}{n} \, \log \norm{\An{n} (z)}$ associated to its iterates are subharmonic along any complex line intersected with the domain $\strip_r^d$, so in particular they are separately subharmonic on $\strip_r^d$.

In this section we establish some general estimates on separately subharmonic functions $u (z)$ defined in a neighborhood of the torus. We later apply these estimates to the functions $\un{n}{A} (z)$. In all of our estimates we assume an upper bound on $u (z)$ throughout the whole domain, and a lower bound at some point on the torus. 

More precisely, throughout this section, we are given a separately subharmonic function $u \colon \strip_r^d \to [ - \infty, \infty)$ such that for some constant $C < \infty$ we have:

\begin{enumerate}
\item $u (z) \le C $ for all $z \in \strip_r^d$ \label{a1} ; \\
\item $u (x_0) \ge - C $ for some $x_0 \in \T^d$ \label{a2} .
\end{enumerate}

\smallskip

All the constants in the estimates derived here will depend only on $C$, $r$ and $d$, and not on the given function $u$ per se. 
Moreover, since the width $r$ of the annulus $\strip_r$ and the number $d$ of variables will be fixed throughout, the dependence of these estimates on $r$ and $d$ will eventually stop being emphasized.  

\smallskip

The estimates obtained will refer to the function $u (x)$ (the restriction of $u (z)$ to the torus $\T^d$) and they will be of  the following kinds: an $L^2$-bound, a boosting of a weak a-priori deviation from the mean, and a quantitative version of the Birkhoff ergodic theorem. These types of estimates were previously derived for {\em bounded} separately subharmonic functions (see \cite{B,GS-Holder,sK2}). 
They were also derived (see Section 6.2 in our monograph~\cite{DK-book}) for {\em unbounded} separately subharmonic functions that satisfy some uniform bounds along {\em every line} parallel to a coordinate axis, for instance if a lower bound as in item \eqref{a2} is available for some point on {\em every} such line. This requirement is crucial as for $d=1$ the argument relies on the Riesz representation theorem for subharmonic functions, while for $d>1$ (when the Riesz representation theorem is {\em not} available), these results are obtained by applying the one-variable estimates along lines parallel to the coordinate axes.

However, we (have to) assume the availability of item~\eqref{a2} only at one point on the torus (we may assume it on a larger set of points, yet not one that intersects every line parallel to a coordinate 
axis). Hence the aforementioned results are not immediately applicable. 

The idea is then to horizontally {\em truncate} the function $u(z)$ from below, at a sufficiently low level. The truncation is still separately subharmonic, but also {\em bounded}, so the kinds  of estimates we are interested in do apply.  Moreover, as  the next lemma shows in a quantitative way, a separately subharmonic function cannot be too small for too long, hence the function itself and its (low enough level) truncation differ only on a small set of inputs. Finally, we note that this approach does create asymptotically large constants in all estimates, however, they will be manageable when applied to the functions $\un{n}{A} (z)$, as part of an inductive process.

\smallskip

\begin{lemma}\label{Cartan-type-ssh}

There are constants $\ga = \ga (d) > 0$ and $K_{r, d} < \infty$ such that for all $T \ge K_{r,d} \, C^2$ we have
\begin{equation}\label{Cartan-type-eq}
\abs{ \{ x \in \T^d \colon u (x) < - T \}   } \le e^{ - T^\ga} .
\end{equation}
\end{lemma}

\begin{proof} 
This result (in a slightly different formulation) was already established in \cite{GS-fine}. The formulation in \cite{GS-fine} says:  if $u$ satisfies the upper bound $u (z) \le C$ for all $z \in \strip_r^d$, then there are constants $C_d$ and $C_{r,d}$, such that if for some $T < \infty$ and $\delta \in (0, 1)$ we have
$$\abs{ \{ x \in \T^d \colon u (x) < - T \}   } >  \delta ,$$
then
\begin{equation}\label{ssh-proof-eq1}
\sup  \, \{ u (x) \colon x \in \T^d \} \le C_{r,d} \, C - \frac{T}{C_{r,d} \, \log^d (C_d / \delta)} \, .
\end{equation}

When $d=1$, this is a direct consequence of Cartan's estimate (see Section 11.2 in \cite{Levin}) for logarithmic potentials and the Riesz representation theorem for subharmonic functions. When $d > 1$ it follows from the one-variable result and an inductive argument on $d$, using Fubini.

Fix $0 < \ga < \frac{1}{2 d}$. If  $\abs{ \{ x \in \T^d \colon u (x) < - T \}   } > C_d \, e^{ - T^\ga}$, then applying \eqref{ssh-proof-eq1} with $\delta := C_d \, e^{ - T^\ga}$, we conclude that
$$ \sup  \, \{ u (x) \colon x \in \T^d \} < C_{r,d} \, C - \frac{1}{C_{r,d}} \, T^{1/2} \le - C ,$$
provided $T$ is large enough (i.e. $T \ge  (C_{r, d} (C_{r, d}+1))^2 \, C^2 =: K_{r, d} \, C^2$), contradicting the assumption~\eqref{a2} that $u (x_0) \ge - C$ for some $x_0 \in \T^d$.
\end{proof}

\begin{lemma}\label{L2bound-ssh}
The function $u (x)$ is in $L^2$:
\begin{equation}\label{L2bound-hole}
\norm{ u }_{L^2 (\T^d)} \less C^2 \, ,
\end{equation}
with the underlying constant depending on $r$ and $d$.
\end{lemma}

\begin{proof}
This is an easy consequence of the previous lemma. Indeed,
\begin{align*}
\norm{u}^2_{L^2 (\T^d)} & = \int_{\T^d} \, \abs{u (x) }^2 \, d x = \int_0^\infty \, \abs{ \{ x \in \T^d \colon \abs{u (x)}^2  >  T \}   } \, d T \\
& =  \int_0^{C^\ast} \, \abs{ \{ x  \colon \abs{u (x)}^2  >  T \}   } \, d T + \int_{C^\ast}^\infty \, \abs{ \{ x  \colon \abs{u (x)}^2  >  T \}   } \, d T \, ,
\end{align*}
where we choose $C^\ast := \max \{C^2, (K_{r, d} \, C^2)^2 \}$.

Note that $$\int_0^{C^\ast} \, \abs{ \{ x  \colon \abs{u (x)}^2  >  T \}   } \, dT \le C^\ast \less C^4 .$$

We  estimate the second integral using Lemma~\ref{Cartan-type-ssh}. Note that since $u (x) \le C$ for all $x \in \T^d$, and since $\sqrt{C^\ast} \ge C$, if $\abs{u (x)}^2  >  T$ and $T \ge C^\ast$ then we must have that $u (x) < - \sqrt{T}$. Hence
\begin{align*}
 \int_{C^\ast}^\infty \, \abs{ \{ x  \colon \abs{u (x)}^2  >  T \}   } \, d T  & =  \int_{C^\ast}^\infty \, \abs{ \{ x  \colon u (x)  < -  \sqrt{T} \}   } \, d T \\
 & \le \int_0^\infty e^{- \sqrt{T}^\ga} \, d T \less 1 ,
 \end{align*}
 where the value of the last integral depends on $\ga$, so on $d$ only.
\end{proof}

\begin{lemma}\label{splitting-ssh}
There are some (explicit) constants $p = p (d) < \infty$, $\alpha = \alpha (d) > 0$ and $\beta = \beta (d) > 0$, such that if the weak a-priori estimate 
\begin{equation}\label{splitting-ssh-weak}
\abs{ \{ x \in \T^d \colon  \abs{ u (x) - \avg{u} }  > \ep _0 \} } < \ep_1
\end{equation}
holds for some $\ep_0, \ep_1 > 0$ small enough, with $\ep_1 \le \ep_0^p$ and $\ep_0 \less C^{-2}$, then the following sharper deviation estimate also holds
\begin{equation}\label{splitting-ssh-strong}
\abs{ \{ x \in \T^d \colon \abs{  u (x) - \avg{u} } >  \ep _0^{\alpha} \} }
 < e^{- \ep_0^{- \beta}} .
 \end{equation}
\end{lemma}

\begin{proof} Let
$$\newu (z) := \max \, \{ u (z), \, - \ep_0^{-1} \} \quad \text{for all } z \in \strip_r^d \, .$$

Then $\newu$ is separately subharmonic  and $- \ep_0^{-1} \le \newu (z) \le C < \ep_0^{-1}$, so 
$$\abs{\newu (z) } \le \ep_0^{-1} \quad \text{for all } z \in \strip_r^d  \, .$$

By Lemma~\ref{Cartan-type-ssh}, there is $\ga = \ga (d) > 0$ such that if we denote
$$\mathcal{D} := \{ x \in \T^d \colon u (x) < - \ep_0^{-1} \} \, ,$$
and since $\ep_0^{-1} \more C^{-2}$, then $\abs{ \mathcal{D} } < e^{- \ep_0^{-\ga}}$ and $u (x) = \newu (x) $  for all   $x \in \mathcal{D}\comp$.   

Note that by Lemma~\ref{L2bound-ssh}, $\norm{ u }_{L^2 (\T^d)} \less C^2 < \ep_0^{-1}$ and clearly $\norm{ \newu }_{L^2 (\T^d)} \le \norm{ \newu }_{L^\infty (\T^d)} \le   \ep_0^{-1}$. Then by Cauchy-Schwarz, 
\begin{align*}
\abs{ \avg{u} - \avg{\newu}  } & \le \int_{\T^d} \abs{ u (x) - \newu (x)  } \, d x = \int_{\mathcal{D}} \abs{ u (x) - \newu (x)  } \, d x \\
& \le ( \norm{ u }_{L^2 (\T^d)} + \norm{ \newu }_{L^2 (\T^d)} ) \cdot \abs{ \mathcal{D} }^{1/2} \less \ep_0^{-1} \cdot e^{- \frac{1}{2} \ep_0^{-\ga}} < \ep_0 ,
\end{align*} 
provided $\ep_0$ is small enough (depending on $\ga$, hence on $d$). 

Let $\B :=  \{ x \in \T^d \colon  \abs{ u (x) - \avg{u} }  > \ep _0 \}$. Then if  $x \notin \B \cup \mathcal{D}$ we have
\begin{align*}
\abs{ \newu (x) - \avg{ \newu}  } & \le  \abs{ \newu (x) - u (x)  }  + \abs{ u (x) - \avg{ u}  }  + \abs{ \avg{u} - \avg{ \newu}  } \\
& \le 0 + \ep_0 + \ep_0 = 2 \, \ep_0 .
\end{align*} 

We may of course assume that $\ep_1 > e^{- \ep_0^{-\ga}}$, otherwise there is nothing to prove. Then
\begin{equation}\label{newu-weak}
\abs{ \{ x \in \T^d \colon  \abs{ \newu (x) - \avg{ \newu} }  > 2 \ep _0 \} } \le \abs{\B} + \abs{\mathcal{D}} \le \ep_1 + e^{- \ep_0^{-\ga}} \le 2 \ep_1 .
\end{equation}

We conclude that $\newu (z)$ is a {\em bounded} separately subharmonic function, with a (weak) a-priori deviation from the mean bound like \eqref{splitting-ssh-weak}. As mentioned earlier, using Lemma 4.12 in \cite{B}, \eqref{newu-weak},  can be boosted to a stronger estimate. We refer below to  our more explicit formulation of this boosting (see Lemma 6.10 in \cite{DK-book}). 

When $d=2$ and if $\ep_1 \le \ep_0^8$, for an absolute constant $c > 0$ we get
$$\abs{ \{ x \in \T^2  \colon \abs{ \newu (x) - \avg{ \newu } }  > (2 \ep_0)^{1/4} \} } < e^{- c \bigl[ (2 \ep_0)^{1/4} + \umeas \, (2 \ep_1)^{1/8} \, (2 \ep_0)^{-1/2}  \bigr] ^{-1}} ,$$
where the constant $\umeas$ is a multiple (depending on $r$) of some uniform bounds on $u (z)$. Thus in our case $\umeas \less \ep_0^{-1}$, and assuming $\ep_1 \le \ep_0^{14}$, 
$$ e^{- c \bigl[ (2 \ep_0)^{1/4} + \umeas \, (2 \ep_1)^{1/8} \, (2 \ep_0)^{-1/2}  \bigr] ^{-1}}  < e^{- c \, \ep_0^{-1/4}} <  e^{- \ep_0^{-1/5}}  .$$

We conclude that if $d=2$ then
$$\abs{ \{ x \in \T^2  \colon \abs{ \newu (x) - \avg{ \newu } }  > (2 \ep_0)^{1/4} \} } <  e^{- \ep_0^{-1/5}}  .$$
A similar argument works for any $d$, where $1/4$ is replaced by $1/2d$ etc.

Therefore, the conclusion \eqref{splitting-ssh-strong} holds for the truncation $\newu$: for some powers $p < \infty$ and $\alpha, \beta > 0$, all depending only on $d$,
$$\abs{ \{ x \in \T^d \colon \abs{  \newu (x) - \avg{ \newu} } >  \ep _0^{\alpha} \} }
 < e^{- \ep_0^{- \beta}} .$$
 
 It follows that a similar estimate holds for the original function $u$. 
 
 Let $\B^\sharp :=  \{ x \in \T^d \colon \abs{  \newu (x) - \avg{ \newu} } >  \ep _0^{\alpha} \}$.
 Then 
 $$\abs{\B^\sharp \cup \mathcal{D} } \le  \abs{\B^\sharp} + \abs{\mathcal{D}} \le e^{- \ep_0^{- \beta}} + e^{- \ep_0^{- \ga}} ,$$ 
 and if $x \notin \B^\sharp \cup \mathcal{D}$, then
\begin{align*}
\abs{ u (x) - \avg{ u}  } & \le  \abs{ u (x) - \newu (x)  }  + \abs{ \newu (x) - \avg{ \newu}  }  + \abs{ \avg{\newu} - \avg{ u}  } \\
& \le 0 + \ep_0^{\alpha} + \ep_0 \less \ep_0^{\alpha} .
\end{align*} 
\end{proof}

We are now ready to formulate and to prove a quantitative version of Birkhoff's ergodic theorem for separately subharmonic functions satisfying the bounds in items~\eqref{a1} and \eqref{a2}.

\begin{theorem}\label{qBET-ssh-thm} Let $u \colon \strip_r^d \to [ - \infty, \infty)$ be a separately subharmonic function such that for some constant $C < \infty$ we have
\begin{enumerate}
\item $u (z) \le C $ for all $z \in \strip_r^d$ \label{item1} ; \\
\item $u (x_0) \ge - C $ for some $x_0 \in \T^d$ \label{item2} .
\end{enumerate}
Let $\om \in {\rm DC}_t$ be a translation vector and denote by $S_n \, u (x) :=  \sum_{j=0}^{n-1} u (T^j x)$ the corresponding Birkhoff sums of $u (x)$.

There are constants $a=a(d)>0, b=b(d)>0, k=k(r, d) <\infty$ and $C_r < \infty$  such that for all $n \ge n_0 := k \cdot t^{-2}$ we have
\begin{equation}\label{qBET-ssh}
\abs{ \{ x \in \T^d \colon \abs{ \frac{1}{n} \, S_n \, u (x) - \avg{u}  } > \umeas \, n^{- a}  \} } <  e^{- n^b} , 
\end{equation}
where $\umeas = C_r \, C^2$.
\end{theorem}

\begin{proof}
We deduce this result from the corresponding one for {\em bounded} (a weaker assumption would suffice) separately subharmonic functions, by horizontal truncation. 

We refer to the more precise formulation of this result given by Theorem 6.5 in~\cite{DK-book}: if $v (z)$ is separately subharmonic on $\strip_r^d$ with some uniform bound $N$ (in particular, if $\abs{v (z) } \le N$ for all $z$), then for some constants $a, b > 0$ depending on $d$, $C_r < \infty$ and for all $n \ge t^{-2}$, we have
\begin{equation}\label{ssh-proof1}
\abs{ \{ x \in \T^d \colon \abs{ \frac{1}{n} \, S_n \, v (x) - \avg{v}  } > C_r \, N \, n^{- a}  \} } < e^{- n^b} 
\end{equation}

Let $k=k(r, d)$ be large enough (so that for instance, $k^{a/2} \ge K_{r, d}$, where $K_{r, d}$ is the constant from Lemma~\ref{Cartan-type-ssh}). Put $n_0 := k \cdot t^{-2}$.

Now fix (any) $n \ge n_0$ and define the truncation
$$ \newu (z) := \max \, \{ u (z), \, - C^2 \, n^{a/2} \} .$$

Then $\newu$ is separately subharmonic and {\em bounded} on $\strip_r^d$, that is, $\abs{\newu (z) } \le C^2 \, n^{a/2}$ (as we may of course assume that $C\ge1$). Hence \eqref{ssh-proof1} applies to $\newu$ with $N = C^2 \, n^{a/2}$ and we have
\begin{equation}\label{ssh-proof2}
\abs{ \{ x \in \T^d \colon \abs{ \frac{1}{n} \, S_n \, \newu (x) - \avg{\newu}  } >  \umeas \, n^{- a/2}  \} } < e^{- n^b} .
\end{equation}

On the other hand,  if $\mathcal{D} := \{ x \in \T^d \colon u (x) < - C^2 \, n^{a/2} \}$, and since $C^2 \, n^{a/2} \ge K_{r, d} C^2$, then by Lemma~\ref{Cartan-type-ssh} we have $\abs{\mathcal{D}} \le e^{- C^{2 \ga} \, n^{a \ga/2}} \le  e^{- n^{a \ga/2}}$. Moreover, if $x \notin \mathcal{D}$ then $u (x) = \newu (x)$.

Let  $$\Dif^{\flat} := \bigcup_{j=0}^{n-1} \, T^{-j} \, \Dif .$$
Then clearly $\abs{ \Dif^{\flat}} \le n \, e^{- n^{a \ga/2} } < e^{- n^{a \ga/4}}$, and if $x \notin \Dif^{\flat}$ then 
$$ \frac{1}{n} \, S_n \, u (x)  =  \frac{1}{n} \, S_n \, \newu (x) .$$ 

Moreover, as before, by Lemma~\ref{L2bound-ssh}, $\norm{ u }_{L^2 (\T^d)} \less C^2$ and clearly $\norm{ \newu }_{L^2 (\T^d)} \le \norm{ \newu }_{L^\infty (\T^d)} \le  C^2 \, n^{a/2}$. Then by Cauchy-Schwarz, 
\begin{align*}
\abs{ \avg{u} - \avg{\newu}  } & \le \int_{\T^d} \abs{ u (x) - \newu (x)  } \, d x = \int_{\mathcal{D}} \abs{ u (x) - \newu (x)  } \, d x \\
& \le ( \norm{ u }_{L^2 (\T^d)} + \norm{ \newu }_{L^2 (\T^d)} ) \cdot \abs{ \mathcal{D} }^{1/2} \le C^2 \,  n^{a/2} \, e^{- \frac{1}{2} \, n^{a \ga/2}} < C^2 e^{- n^{a \ga/4}} .
\end{align*}

Let $\B := \{ x \in \T^d \colon \abs{ \frac{1}{n} \, S_n \, \newu (x) - \avg{\newu}  } >  \umeas \, n^{- a/2}  \}$, so by \eqref{ssh-proof2} we have $\abs{\B} < e^{- n^b}$.

Then for any $x \notin \Dif^\flat \cup \B$, i.e. outside a sub-exponentially small set,
\begin{align*}
\abs{ \frac{1}{n} \, S_n \, u (x) - \avg{ u}  } & \le  \abs{ \frac{1}{n} \, S_n \, u (x) - \frac{1}{n} \, S_n \, \newu (x)  }  \\
&+ \abs{ \frac{1}{n} \, S_n \, \newu (x) - \avg{ \newu}  }  + \abs{ \avg{\newu} - \avg{ u}  } \\
& \le 0 +  \umeas \, n^{-a/2} + C^2 \, e^{- n^{a \ga/4}} < \umeas \, n^{- a/3}  ,
\end{align*} 
which completes the proof.
\end{proof}

%%%%%%%%%%%%%%%%%%%%%%%%%%%%%%%%%%%%%%%%%%%%%%%%%%%%%%%%%%%%%%%%
%%%%%%%%%%%%%%%%%%%%%%%%%%%%%%%%%%%%%%%%%%%%%%%%%%%%%%%%%%%%%%%%

\section{The proof of the uniform fiber LDT}\label{ufiber-ldt}
\newcommand{\rgap}{{\rm gr}}
\newcommand{\rift}{\rho}

In this section we present a new inductive procedure for establishing the uniform fiber LDT for quasi-periodic cocycles. This procedure might prove useful in other settings as well.

In Theorem~\ref{nonunif-ldt-thm} we established a non-uniform fiber LDT estimate for any given identically singular cocycle, by reduction to a maximal rank cocycle, for which this result was already available. We indicated in Section~\ref{non-unif-ldt} that this reduction procedure is {\em unstable} under perturbations of the cocycles, and so it does not produce a uniform result, in the sense that the parameters in the LDT estimate may blow up as we perturb the cocycle. The idea is to use this non-uniform LDT just as an input to start the inductive process\textemdash by proximity to the given cocycle, we derive a uniform LDT at a {\em fixed} initial scale. Then we prove a fiber LDT estimate in the vicinity of the given cocycle, at larger and larger scales, using the avalanche principle. This inductive process allows us to keep track of the parameters in the LDT estimates and it ensures their uniformity.

\begin{theorem}\label{u-fiber-ldt-thm}
Given  $A \in \qpcmat{d}$ with $L_1 (A) > L_2 (A)$ and $\om \in \rm{DC}_t$, there are constants $\delta = \delta (A) > 0$, $n_1 = n_1 (A, t) \in \N$, $a_1 = a_1 (d) > 0$, $b_1 = b_1 (d)  > 0$  so that if $\normr{B-A} \le \delta$ and $n \ge n_1$, then 
\begin{equation} \label{unif-ldt-eq}
\abs{  \{ x \in \T^d \colon \abs{ \frac{1}{n} \log \norm{\Bn{n} (x)} - \LE{n}_1  (B)  } > n^{-a_1} \} } < e^{- n^{b_1}}.
\end{equation}
Moreover, there is a constant $C= C (A) < \infty$ such that for all $n \ge 1$
\begin{equation}\label{uniform-l2-bound-thm}
\bnorm{\frac{1}{n} \, \log \, \norm{\Bn{n}}}_{L^2 (\T^d)} < C ,
\end{equation}
showing that $A$ is uniformly $L^2$-bounded.
\end{theorem}

\begin{proof} We break down the argument into several steps.
%\item
\subsection*{Before we begin.} We define a threshold $\nzerobar$ after which various estimates start being applicable.  It is important to see that this threshold depends only on the given and fixed data: on the frequency $\om$ (in fact, just on the parameter $t$) and on the cocycle $A$ (in fact just on $d$, $r$ and on some measurements of $A$ such as $\norm{A}_r$).

For instance, if $n \ge \nzerobar$, then the non-uniform fiber LDT estimate in Theorem~\ref{nonunif-ldt-thm}, the quantitative Birkhoff Ergodic Theorem~\ref{qBET-ssh-thm}, as well as other results we use, are all applicable at scale $n$.

Asymptotically, any power of $n$ dominates a constant function and it is itself
dominated by an exponential function of  $n$. We assume  $\nzerobar$  to be large enough
that  any of these relations involving concrete size and measure bounds appearing throughout the proof
will  hold for $n \ge \nzerobar$.

As before, for any cocycle $B \in \qpcmat{d}$ and for any number of iterates $n \ge 1$, let $\un{n}{B} \colon \strip_r^d \to [ - \infty, \infty)$ denote the separately subharmonic function
$$\un{n}{B} (z) := \frac{1}{n} \, \log \, \norm{ \Bn{n} (z) } .$$

This function has the uniform upper bound $\un{n}{B} (z) \le \log \norm{B}_r $ but of course,  in general it is not bounded from below (which is the main point of this paper).

Recall the notation $\avg{u}$ for the mean of a function $u (x)$ on $\T^d$. In the case of the functions $\un{n}{B}$ associated to iterates of a cocycle $B$, we also denote this mean by $\LE{n}_1 (B)$, and refer to it as the {\em finite scale} (maximal) Lyapunov exponent of $B$, since $\LE{n}_1 (B) \to L_1 (B)$ as $n \to \infty$. 

\smallskip

Let $\ep_0 := \frac{L_1 (A) - L_2 (A)}{50} > 0$ (if $L_2 (A) > - \infty$, otherwise just choose $\ep_0 :=1$). Assume $\nzerobar$ to be large enough that for all $n, n'  \ge \nzerobar$ we have $\abs{ \LE{n}_1 (A) - \LE{ n'}_1 (A) } < \ep_0$.

\smallskip

Let $C_0$ be a finite constant such that $\log \norm{A}_r < C_0$. This constant is chosen generously enough that if $B$ is a cocycle in a small, fixed neighborhood of $A$, then the bound $\log \norm{B}_r < C_0$ also holds. 

By Corollary~\ref{non-unif-L2bound}, the functions $\un{n}{A}$ are uniformly  (in $n$) $L^2$ bounded, hence we may assume that $\norm{\un{n}{A}}_{L^2 (T^d)} < C_0$. 

This of course also implies that for all $n \ge 1$, $\abs{ \LE{n}_1 (A) } < C_0$.

\smallskip

We will choose other, slightly larger constants $C_0 < C_1 \ll  C_2 $ that accommodate some extra polynomially small error terms, e.g.  $C_0 + \nzerobar^{ - p} < C_1$, for some power $p > 0$ that depends on the given data, and similarly for the other constant.

For two real numbers $a$ and $b$,  we write $a \asymp b$ or $a = \bigo (b)$ to mean $c_1 \, a \le b \le c_2 \, a$ for some absolute, positive constants $c_1$ and $c_2$. When it comes to (integer) scales, $m \asymp n$ is meant in a more strict way as $n \le m \le 3 n$.

%\item
\subsection*{The zeroth scale.} We fix $n_0 \ge \nzerobar$ and define a neighborhood of $A$ such that the fiber LDT estimate on $A$ at any scale $m \asymp n_0$  transfers over to every cocycle $B$ in that neighborhood.

Let $\delta_0 := e^{- \bar{C} \, n_0}$ (where $\bar{C}$ will be chosen later) be the size of the neighborhood.  From now on $B$ will be any cocycle with $\norm{B-A}_r < \delta_0$. 

For all scales $m \asymp n_0$ and for all $x \in \T^d$ we have
\begin{align}
\abs{ \norm{ \An{m} (x) } - \norm{ \Bn{m} (x) } }  & \le \norm{ \An{m} (x) - \Bn{m} (x) } \le \delta_0  \, m \, e^{ C_0 \, m} \notag \\
& <  e^{- \bar{C} \, n_0} \, e^{3 (C_0 +1) \, n_0} < e^{- 2 C_1  \, m}  \label{proximity-first-scale},
\end{align}
provided we choose $\bar{C} \more C_0 + C_1$. 

Since $\abs{\LE{m}_1 (A)} < C_0$, at least for some $x_0 \in \T^d$ we must have $ \abs{ \frac{1}{m} \log \norm{\An{m} (x_0)}} < C_0$, so 
$$\norm{\An{m} (x_0) } > e^{- C_0 \, m} .$$

Combined with~\eqref{proximity-first-scale} this implies
$$\norm{ \Bn{m} (x_0) } > e^{- C_0 \, m} - e^{- 2 C_1  \, m}  \more e^{- C_0  \, m} ,$$
hence
$$ \un{m}{B} (x_0) = \frac{1}{m} \log \norm{\Bn{m} (x_0)} \ge -C_0 - \bigo ( \frac{1}{m} )  > - C_1 .$$

Combining this with the upper bound $\un{m}{B} (z) \le C_0$, Lemma~\ref{L2bound-ssh} implies
\begin{equation*}\label{L2bound-n0}
 \norm{\un{m}{B}}_{L^2 (T^d)} \less C_1^2 < C_2 .
 \end{equation*}

Denote by $\B_n (A) :=  \{ x \in \T^d \colon \abs{ \frac{1}{n} \log \norm{\An{n} (x)} - \LE{n}_1  (A)  } >  n^{-a_0} \}$. If $n \ge \nzerobar$, then by Theorem~\ref{nonunif-ldt-thm}, $\abs{\B_n (A)} \le  e^{- n^{b_0}}$.

We are now ready to derive the proximity of the functions $\un{m}{B} (x)$ and $\un{m}{A} (x)$ for many phases $x \in \T^d$, as well as that of their means. 

If  $x \notin \B_n (A)$, then
\begin{align*}
 \frac{1}{n} \log \norm{\An{n} (x)} \ge \LE{n}_1 (A) -  n^{-a_0} > - C_0 -   \nzerobar^{-a_0}  > - C_1 ,
\end{align*}
hence $\norm{\An{n} (x)} > e^{- C_1 \, n}$. 

Then if $m \asymp n_0$ and if $x \notin \B_m (A)$, a similar lower bound holds also for $\Bn{m} (x)$:
\begin{align*}
\norm{\Bn{m} (x)} & \ge \norm{\An{m} (x)} - \norm{ \An{m} (x) - \Bn{m} (x) }  \\
& > e^{- C_1 \, m} - e^{- 2 C_1 \, m} \more e^{- C_1 \, m} ,
\end{align*}

For $m \asymp n_0$ and $x \notin \B_m (A)$ we then have
\begin{align*}
\babs{  \frac{1}{m} \log \norm{\Bn{m} (x)} -  \frac{1}{m} \log \norm{\An{m} (x)} } & \le \frac{1}{m} \, \frac{ \abs{ \norm{ \An{m} (x) } - \norm{ \Bn{m} (x) } }   }{  \min \{ \norm{ \An{m} (x)}, \, \norm{ \Bn{m} (x) }    \} } \\
& < \frac{ e^{- 2 C_1  \, m}   }{  e^{- C_1  \, m}  } = e^{-  C_1  \, m}  .
\end{align*}

Moreover, by the previous estimates and Cauchy-Schwarz,
\begin{align*}
\babs{ \LE{m}_1 (B) - \LE{m}_1 (A) } & \le \int_{\T^d} \, \babs{  \frac{1}{m} \log \norm{\Bn{m} (x)} -  \frac{1}{m} \log \norm{\An{m} (x)} } \,  d x \\ 
\\
&   \kern-8.3em =  \int_{(\B_m (A))\comp} \, \babs{  \un{m}{B} (x) - \un{m}{A} (x) } \,  d x +  \int_{\B_m(A)} \, \babs{   \un{m}{B} (x) - \un{m}{A} (x) } \,  d x \\
& \kern-8.3em  \le e^{-  C_1  \, m}  + ( \norm{\un{m}{B}}_{L^2 (T^d)}  + \norm{\un{m}{A}}_{L^2 (T^d)} ) \cdot \abs{ \B_m (A) }^{1/2} \\
& \kern-8.3em \less e^{-  C_1  \, m} + C_2 \,  e^{- m^{b_0/2}}  < m^{- a_0}  ,
\end{align*}
hence
\begin{equation}\label{cont-n0}
\babs{ \LE{m}_1 (B) - \LE{m}_1 (A) }  <   m^{- a_0}  < \ep_0 .
\end{equation}

We conclude that for all $m \asymp n_0$ and $x \notin \B_m (A)$ (the set where the fiber LDT for $A$ and the estimates above all hold) we have
\begin{align*}
\abs{ \frac{1}{m} \log \norm{\Bn{m} (x)} - \LE{m}_1  (B)  } & \le e^{-  C_1  \, m}  +  m^{- a_0} \\ 
& + \abs{ \frac{1}{m} \log \norm{\An{m} (x)} - \LE{m}_1  (A)  }  \\
& \le 3  \,  m^{- a_0} ,
\end{align*}
which proves the following {\em uniform} fiber LDT at {\em initial} scales $m \asymp n_0$:
\begin{equation}\label{unif-ldt-n0}
\babs{ \{ x \in \T^d \colon  \abs{ \frac{1}{m} \log \norm{\Bn{m} (x)} - \LE{m}_1  (B)   }  >  3  \, m^{- a_0}  \} } < e^{- m^{b_0}} .
\end{equation}

Let us denote by $\B_m (B)$ the exceptional set in \eqref{unif-ldt-n0}, so for $m \asymp n_0$, $\abs{ \B_m (B)} < e^{- m^{b_0}}$.

\subsubsection*{Summary of estimates at scales $\asymp n_0$.} We use \eqref{cont-n0} to derive two more estimates. If $m \asymp n_0$ then
$$\abs{ \LE{m}_1 (B) } < \ep_0 + \abs{ \LE{m}_1 (A) } < \ep_0 + C_0 < C_1 ,$$ 
and if $m, m' \asymp n_0$ then
\begin{align*}
\abs{ \LE{m}_1 (B) - \LE{m'}_1 (B) } & \le \abs{ \LE{m}_1 (B) - \LE{m}_1 (A) } + \abs{ \LE{m'}_1 (B) - \LE{m'}_1 (A) } \\
& + \abs{ \LE{m}_1 (A) - \LE{m'}_1 (A) } \\
&  \le m^{- a_0} + m'^{- a_0} + \ep_0 < 2 \, \ep_0  .
\end{align*}

 We now summarize the estimates at scales $\asymp n_0$ that are needed at the next scale. 
Let $m, m' \asymp n_0$. Then for any cocycle $B$ with $\norm{B-A}_r \le \delta_0$ the following hold:
\begin{subequations}
\label{summary-scale-n0}
\begin{align}
 \label{summary-n0-eq1bis}
  \norm{\un{m}{B}}_{L^2 (\T^d)} & < C_2 \\
  \label{summary-n0-eq1}
  \abs{\LE{m}_1 (B)} & < C_1 \\
  \label{summary-n0-eq3}
   \abs{ \LE{m}_1 (B) - \LE{m'}_1 (B) } & < 2 \,  \ep_0  \\
   \label{summary-n0-eq2}
  \babs{ \LE{m}_1 (B) - \LE{m}_1 (A) }  & < \ep_0 \\
   \label{summary-n0-eq4}
   \abs{ \frac{1}{m} \log \norm{\Bn{m} (x)} - \LE{m}_1  (B)  } & \le 3 \, m^{- a_0} ,
\end{align}
\end{subequations}
where the last estimate holds for all $x \notin \B_m (B)$, with $\abs{ \B_m (B)} < e^{- m^{b_0}}$.

\subsection*{The first scale.} Let the new scale $n_1$ be such that $n_0^{p'} \le n_1 \le  e^{ n_0^{b_0/2}}$, where the power $p' = p' (d)$ is large enough but finite, and it will be made more explicit later.  We fix (any) such integer $n_1$ and prove a uniform LDT  for scales $\asymp n_1$.
The idea is to break down the long block (i.e. product of matrices) $\Bn{n_1} (x)$ into blocks of length $\asymp n_0$ and apply the avalanche principle to these shorter blocks, thus relating certain quantities at scale $n_1$ to similar quantities at scales $\asymp n_0$.

\subsubsection*{Recalling a couple of relevant results.} For the reader's convenience, we formulate below two crucial results proven in our monograph \cite{DK-book}: the avalanche principle (AP) and the uniform upper semicontinuity (u.s.c.) of the maximal Lyapunov exponent.  

In \cite{DK-book}  we proved a general version of the AP, one that applies to a sequence of higher dimensional, non-invertible matrices. Moreover, compared to previous versions of the principle, we also removed the constraint on the number of matrices in the sequence. Below we describe one of the statements in the AP that will be used in this proof. Recall the following terminology: if $g \in \gl_m (\R)$, let $s_1 (g) \ge s_2 (g) \ge \ldots \ge s_m (g) \ge 0$ be its singular values and let $\rgap (g) := \frac{s_1 (g)}{s_2 (g)} \ge 1$ denote the ratio of its first two singular values.

\begin{proposition} \label{AP-practical}
There exists $c>0$ such that
 given $0<\epsilon<1$,  $0<\kappa\leq c\,\epsilon^ 2$ 
and  \,  $g_0, g_1,\ldots, g_{n-1}\in\gl_m (\R)$, \,
 if
\begin{align*}
\rm{(gaps)} \  & \rgap (g_i) >  \frac{1}{\ka} &  \text{for all }  & \ \  0 \le i \le n-1  %\label{gaps-AP} 
\\
\rm{(angles)} \  & \frac{\norm{ g_i \, g_{i-1} }}{\norm{g_i}  \, \norm{ g_{i-1}}}  >  \ep  & \
 \text{for all }   & \  \ 1 \le i \le n-1  %\label{angles-AP}
\end{align*}
then  denoting $ g^{(n)} := g_{n-1} \ldots g_1 \, g_0$, we have
\begin{align*} %\label{new-angles-AP}
\sabs{ \log \norm{ g^{(n)} } + \sum_{i=1}^{n-2} \log \norm{g_i} -  \sum_{i=1}^{n-1} \log \norm{ g_i \cdot g_{i-1}} } \less n \cdot \frac{\ka}{\ep^2} \;.
\end{align*}
 \end{proposition}

\smallskip

We now describe the uniform u.s.c. of the maximal LE (see \cite{JitMavi} for the original statement and, for a formulation that completely covers our setting, see Proposition 3.1 and Remark 3.2 in \cite{DK-book}).

\begin{proposition}\label{n-unif-usc}
Let $A \in \qpcmat{d}$. 

\begin{enumerate}
\item[(i)] Assume that  $L_1 (A) > - \infty$.   
\end{enumerate}

For every $\epsilon > 0$, there are $\delta = \delta(A, \epsilon) > 0$ and $n_0 = n_0 (A, \epsilon) \in \N$  such that if $B \in \qpcmat{d}$ with $\norm{B-A}_r < \delta$ and if $n \ge n_0$ then for all $x \in \T^d$ 
\begin{equation}\label{n-unif-usc-eq}
\frac{1}{n} \, \log \norm{B^{(n)} (x) } \le L_1 (A) + \epsilon .
\end{equation}

\begin{enumerate}
\item[(ii)] Assume that $L_1 (A) =  - \infty$.  
\end{enumerate}
For every $L < \infty$, there are $\delta = \delta(A, L) > 0$ and $n_0 = n_0 (A, L) \in \N$ such that if $B \in \qpcmat{d}$ with $\norm{B-A}_r < \delta$ and if $n \ge n_0$, then for all $x \in \T^d$ 
\begin{equation}\label{n-unif-usc-eq-infty}
\frac{1}{n} \, \log \norm{B^{(n)} (x) } \le -  L .
\end{equation}
\end{proposition}

We note that from Corollary~\ref{non-unif-L2bound} it follows that in our setting (of analytic, quasi-periodic cocycles), if  $L_1 (A) > - \infty$ then $A$ is automatically $L^1$-bounded. Hence all the assumptions of the more general corresponding result  in \cite{DK-book} are satisfied.

\subsubsection*{Dividing into smaller blocks.} Consider the block $\Bn{n_1} (x)$ of length $n_1$ and break it down into $n$ blocks, each having length $n_0$, except possibly for the last block, which will have length $m_0 \asymp n_0$.  

More precisely, let $n$ and $m_0$ be such that $n_1 = (n-1) \cdot n_0 + m_0$ and $n_0 \le m_0 < 2 m_0$.

For $0 \le i \le n-2$ define 
\begin{align*}
g_i  = g_i (x) & := \Bn{n_0} (T^{i n_0} x) \quad \text{and} \\
g_{n-1}  = g_{n-1} (x) & := \Bn{m_0} (T^{(n-1) n_0} x) . 
\end{align*}

Then $g^{(n)} = g_{n-1} \ldots g_1 \, g_0 = \Bn{n_1} (x)$. 

We show that the geometrical assumptions ``gaps'' and ``angles'' of the AP are satisfied for these matrices if we choose the phases $x$ outside a certain small set.

\subsubsection*{The gaps condition.} We establish this condition as a consequence of  the uniform LDT~\eqref{summary-n0-eq4} and the estimate~\eqref{summary-n0-eq2} proven for scales $m \asymp n_0$. 

In the proof we use the uniform u.s.c. of the maximal LE in Proposition~\ref{n-unif-usc} and the hypothesis on the existence of a gap between the first two LE of $A$ (i.e. the fact that $L_1 (A) - L_2 (A) > 0$).

Note that if $g \in \gl_m (\R)$, then $$\rgap (g) := \frac{s_1 (g)}{s_2 (g)} = \frac{s_1^2 (g)}{s_1 (g) \, s_2 (g)} = \frac{\norm{g}^2}{\norm{ \wedge_2 g }} ,$$
where $\wedge_2 g$ represents the second exterior power of $g$.
If $m \asymp n_0$ then
\begin{equation}\label{angles-eq1}
\frac{1}{m} \, \log \rgap (\Bn{m} (x)) = 2 \, \frac{1}{m} \, \log \norm{\Bn{m} (x)} -  \frac{1}{m} \, \log \norm{\wedge_2 \Bn{m} (x)} \,.
\end{equation} 

We estimate each of the two terms above from below.
 
If $x \notin \B_m (B)$, the LDT estimate~\eqref{summary-n0-eq4} implies
\begin{equation}\label{proof-gaps-n0}
\un{m}{B} (x) =  \frac{1}{m} \, \log \norm{\Bn{m} (x)}  > \LE{m}_1 (B) - 3 \, m^{-a_0}  .
\end{equation}

Combining this with the fact that when $m \asymp n_0$, by \eqref{summary-n0-eq2},  $\LE{m}_1 (B)$ and $\LE{m}_1 (A)$ are close, we conclude that if  $x \notin \B_m (B)$ then
\begin{equation}\label{angles-eq2}
 \frac{1}{m} \, \log \norm{\Bn{m} (x)} > \LE{m}_1 (A) - \ep_0  - 3 \, m^{-a_0} \ge L_1 (A) - 2 \ep_0 .
 \end{equation}

To estimate $\frac{1}{m} \, \log \norm{\wedge_2 \Bn{m} (x)} $ from above, we apply the uniform u.s.c. of the maximal LE in Proposition~\ref{n-unif-usc},  to the cocycle $\wedge_2 B$. By making $\delta_0$ smaller, we may assume that $B$ is close enough to $A$ that in turn $\wedge_2 B$ is close enough to $\wedge_2 A$. 

If  $L_1 (\wedge_2 A) > - \infty$, which is equivalent to $L_2 (A) > - \infty$, we have
\begin{equation}\label{angles-eq10}
\frac{1}{n} \, \log \norm{\wedge_2 \Bn{n} (x)} < L_1 (\wedge_2 A) + \ep_0 = L_1 (A) + L_2 (A) + \ep_0
\end{equation}
for all $n \ge \nzerobar$ and for all $x \in \T^d$ (which is the reason for calling this a uniform statement). 
Then for all $x \in \T^d$ we have
\begin{equation}\label{angles-eq3}
- \frac{1}{m} \, \log \norm{\wedge_2 \Bn{m} (x)}  > - L_1 (A) - L_2 (A) - \ep_0 \, . 
\end{equation}

Estimates \eqref{angles-eq1}, \eqref{angles-eq2}, \eqref{angles-eq3} imply that if $m \asymp n_0$ and $x \notin \B_m (B)$, then
\begin{align*}
\frac{1}{m}\log \rgap (\Bn{m} (x)) & > 2 L_1 (A) -  4 \ep_0 - L_1 (A) - L_2 (A) - \ep_0 \\
& = L_1 (A) - L_2 (A) - 5 \ep_0 >  19 \ep_0 , 
\end{align*}
so
\begin{equation}\label{angles-eq}
\rgap (\Bn{m} (x)) \ge e^{19 \, \ep_0 \, m} \ge e^{19 \, \ep_0 \, n_0} =: \frac{1}{\ka_{n_0}}  \quad \text{for } x \notin \B_m (B) .
\end{equation}

The case when $L_1 (\wedge_2 A) = - \infty$ is treated similarly, the difference being that in \eqref{angles-eq10}, instead of $L_1 (\wedge_2 A) + \ep_0$, we may take $- L$, with $L$ arbitrarily large (it will have to be chosen properly). 

Estimate \eqref{angles-eq} shows that the gaps condition in the AP (Proposition~\ref{AP-practical}) holds for $g = g (x) := \Bn{m} (x)$ provided  $m \asymp n_0$ and $x \notin \B_m (B)$. Therefore, in order to hold for all matrices $g_0, g_1, \ldots, g_{n-1}$ defined above, we simply exclude the set of phases
$$\B_{n_0}^{\rm gaps} (B) :=   T^{- (n-1) n_0} \, \B_{m_0} (B) \,  \cup \,  \bigcup_{i=0}^{n-2} \, T^{- i n_0} \, \B_{n_0} (B)  .$$

Note that $\abs{ \B_{n_0}^{\rm gaps} (B)   } < n \, e^{- n_0^{b_0}} \le n_1 \, e^{- n_0^{b_0}}  < e^{- \frac{1}{2} \, n_0^{b_0}}$.

\subsubsection*{The angles condition.} We derive this condition from the estimate~\eqref{summary-n0-eq3} and the LDT estimate~\eqref{summary-n0-eq4}, which are available at scales $m \asymp n_0$.

Let $m = n_0$ and $n_0 \le m' \le 2 n_0$, so that $m, m', m+m' \asymp n_0$. Then the estimate~\eqref{summary-n0-eq3} and the fiber LDT estimate~\eqref{summary-n0-eq4} apply at scales $m, m', m+m'$, so if $x \notin \B_m (B) \cup T^{-m} \B_{m'} (B) \cup \B_{m+m'} (B) =: \B_{(m, m')} (B)$, 
\begin{align*}
& \frac{\norm{\Bn{m+m'} (x)}}{\norm{\Bn{m'} (\transl^{m} x)} \ \norm{\Bn{m} (x)}} \\
& \qquad > \frac{ e^{  (m+m') \, [ \LE{m+m'}_{1} (B)   - 3 \,  (m+m')^{-a_0} ] }   }{   e^{  m' \, [ \LE{m'}_{1} (B)   +  3 \, (m')^{-a_0} ] }  \ e^{  m \, [ \LE{m}_{1} (B)   + 3 \, m^{-a_0} ] } }  \\
& \qquad >  e^{m [ \LE{m+m'}_{1} (B) - \LE{m}_{1} (B)  ] + m' \, [ \LE{m'+m}_{1} (B) - \LE{m'}_{1} (B) ]  - (m+m') \, 9 \, n_0^{-a_0}} \\
& \qquad >  e^{- (m+m') ( 2 \, \ep_0 + 9 \, n_0^{-a_0})} > e^{- 9 \, \ep_0 \, n_0} .
\end{align*}

Therefore, if $ x \notin \B_{(m, m')} (B)$, which is a set of measure $ < 3 \, e^{- n_0^b}$, we have
\begin{equation}\label{gaps-eq}
\frac{ \norm{ \Bn{m+m'} (x) } }{ \norm{ \Bn{m'} (T^m x) } \, \norm{ \Bn{m} (x) } } >  e^{- 9 \, \ep_0 \, n_0} =: \ep_{n_0} .
\end{equation}

Estimate \eqref{gaps-eq} shows that the angles condition in the AP (Proposition~\ref{AP-practical}) holds for $g = g (x) :=  \Bn{m} (x)$ and $g' = g' (x) :=  \Bn{m'} (T^m x)$, provided the scales $m, m'$ and the phases $x$ are as described. Therefore, in order to hold for all matrices $g_0, g_1, \ldots, g_{n-1}$ defined above, we simply exclude the set of phases
$$\B_{n_0}^{\rm angles} (B) :=   \B_{(n_0, m_0)} (B) \, \cup  \, \bigcup_{i=0}^{n-2} \,   T^{- i n_0} \, \B_{(n_0, n_0)} (B)  .$$

Note that $\abs{ \B_{n_0}^{\rm angles} (B)   } < 3 \, n \, e^{- n_0^{b_0}} < e^{- \frac{1}{2} \, n_0^{b_0}}$.

\subsubsection*{Applying the avalanche principle.}
Let $\B_{n_0}^{\rm ap} (B) := \B_{n_0}^{\rm gaps} (B)  \, \cup \, \B_{n_0}^{\rm angles} (B)$, so $ \abs{ \B_{n_0}^{\rm ap} (B)   } < e^{- \frac{1}{2} \, n_0^{b_0}}$.

Moreover, from \eqref{angles-eq}, \eqref{gaps-eq} we have  $\frac{\ka_{n_0}}{\ep_{n_0}^2} = e^{- \ep_0 \, n_0} \ll 1$. 

We conclude that if $x \notin \B_{n_0}^{\rm ap} (B) $ then the AP in Proposition~\ref{AP-practical} applies and we have:
\begin{align*}
\log \, \norm{ \Bn{n_1} (x) } & = - \sum_{i=1}^{n-2} \log \, \norm{ \Bn{n_0} (T^{i n_0} x) } \\
& \kern-4em +   \log \, \norm{ \Bn{n_0 + m_0} (T^{(n-2) n_0} x) } + \sum_{i=1}^{n-2} \log \, \norm{ \Bn{2 n_0} (T^{i n_0} x) } + \bigo \left( n \, \frac{\ka_{n_0}}{\ep_{n_0}^2} \right) .
\end{align*}

Divide both sides by $n_1$ and re-write the expression as
\begin{align*}
\frac{1}{n_1} \, \log \, \norm{ \Bn{n_1} (x) }  = & - \frac{(n-2) \, n_0}{n_1}  \ \frac{1}{n - 2}  \,  \sum_{i=1}^{n-2}   \frac{1}{n_0} \, \log \,  \norm{ \Bn{n_0} (T^{i n_0} x) } \\
& + \frac{2 (n-2) \, n_0}{n_1} \  \frac{1}{n - 2} \,   \sum_{i=1}^{n-2}  \frac{1}{2 n_0} \,  \log \, \norm{ \Bn{2 n_0} (T^{i n_0} x) } \\
& +  \frac{n_0 + m_0}{n_1} \,   \frac{1}{n_0+m_0}  \, \log \, \norm{ \Bn{n_0 + m_0} (T^{(n-2) n_0} x) } \\
& + \bigo ( \frac{n}{n_1} \,  \frac{ \ka_{n_0}}{\ep_{n_0}^2} ) .
\end{align*}

We estimate each of the terms above.

First off, since $n_1 \asymp n \, n_0$, we have
$$\frac{(n-2) \, n_0}{n_1}  = 1 + \bigo \left( \frac{n_0}{n_1}\right), \ \frac{2 (n-2) \, n_0}{n_1}  = 2 + \bigo \left( \frac{n_0}{n_1}\right),  \  \frac{n_0 + m_0}{n_1} = \bigo \left( \frac{n_0}{n_1}\right) .$$

Note that 
$$ \frac{1}{n - 2}  \,  \sum_{i=1}^{n-2}   \frac{1}{n_0} \, \log \,  \norm{ \Bn{n_0} (T^{i n_0} x) } =  \frac{1}{n - 2}  \,  \sum_{i=1}^{n-2} \un{n_0}{B} (x + i \ n_0 \, \om)$$
is a Birkhoff average of $\un{n_0}{B} (x)$ over the translation vector $\om_0 := n_0 \, \om$. 

The goal is to apply the quantitative Birkhoff Ergodic Theorem~\ref{qBET-ssh-thm} to the separately subharmonic function $u=\un{n_0}{B}$ and to the translation vector $\om_0$. We need to verify the assumptions of that theorem and estimate the relevant parameters.

Recall that we always have $\un{n_0}{B} (z) \le C_0$ for all $z \in \strip_r^d$.
Moreover, by \eqref{summary-n0-eq1} we also have $\un{n_0}{B} (x) > - C_1$ for some $x \in \T^d$. Therefore, items~\eqref{item1} and~\eqref{item2} in Theorem~\ref{qBET-ssh-thm} hold.

Since $\om \in \DC_t$, for all $k \in \Z^d \setminus \{0\}$,
$$\norm{ k \cdot \om_0} = \norm{k \cdot n_0 \, \om} = \norm{n_0 \, k \cdot \om} \ge \frac{t}{\sabs{n_0 \, k}^{d+1}} = \frac{t \ n_0^{-(d+1)}}{\sabs{k}^{d+1}} \, ,$$
so $\om_0 \in \DC_{t_0}$, where $t_0 := t \ n_0^{-(d+1)} $.

Then Theorem~\ref{qBET-ssh-thm} is applicable provided $n \ge k \, t_0^{-2} = k \, t^{-2} \, n_{0}^{2(d+1)}$, where $k = k (r, d) < \infty$. But $n \asymp \frac{n_1}{n_0}$ and $n_1 \ge n_0^{p'}$, hence for everything to work out we just have to choose $p' > 2 (d+1) + 1$ and $\nzerobar$ large enough depending on $t, r, d$.

There are positive constants $a = a(d)$, $b = b (d)$ and $C_r < \infty$ such that
 $$\frac{1}{n - 2}  \,  \sum_{i=1}^{n-2} \un{n_0}{B} (x + i \ n_0 \, \om) = \avg{ \un{n_0}{B} } + \bigo \left( C_r \, C_1^2 \, n^{-a} \right) \, ,$$
 provided $x$ is outside a set $\B_1$ with $\abs{\B_1} < e^{- n^b}$. 

We conclude that if $x \notin \B_1$, then
$$ \ \frac{1}{n - 2}  \,  \sum_{i=1}^{n-2}   \frac{1}{n_0} \, \log \,  \norm{ \Bn{n_0} (T^{i n_0} x) }  = \LE{n_0}_1 (B) + \bigo \left( C_r \, C_1^2 \, n^{-a} \right) \, .$$

Similarly, since all of this applies in fact to $\un{m}{B} (x)$ for all $m \asymp n_0$, so in particular it applies to $m = 2 n_0$, there is a set $\B_2$ with $\abs{\B_2} <  e^{- n^b}$ such that if $x \notin \B_2$ then
$$ \frac{1}{n - 2} \,   \sum_{i=1}^{n-2}  \frac{1}{2 n_0} \,  \log \, \norm{ \Bn{2 n_0} (T^{i n_0} x) } =  \LE{2 n_0}_1 (B) + \bigo \left( C_r \, C_1^2 \, n^{-a} \right) \, .$$

Note that $C_r \, C_1^2 \, n^{-a}  \le C_r \, C_2 \, n_0^a \, n_1^{-a}  < n_1^{- a/2} $.

Since $n_0 + m_0 \asymp n_0$, we may apply \eqref{summary-n0-eq4} to $ \frac{1}{n_0+m_0}  \, \log \, \norm{ \Bn{n_0 + m_0} } $. Hence there is a set $\B_3$ with $\abs{\B_3} < e^{- n_0^{b_0}}$ such that if $x \notin \B_3$ then
$$ \babs{ \frac{1}{n_0+m_0}  \, \log \, \norm{ \Bn{n_0 + m_0} (T^{(n-2) n_0} x) }  -  \LE{n_0+m_0}_1 (B) } \le  3 \, n_0^{- a_0}  .$$

Moreover, by~\eqref{summary-n0-eq1}, $\abs{ \LE{n_0+m_0}_1 (B) } < C_1$, hence for  $x \notin \B_3$
$$\left|  \frac{1}{n_0+m_0}  \, \log \, \norm{ \Bn{n_0 + m_0} (T^{(n-2) n_0} x) } \right| <  C_1 + 3 \, n_0^{- a_0} <  C_2 .$$

\smallskip

Let $\B := \B_1 \cup \B_2 \cup \B_3$. Note that since $n \asymp \frac{n_1}{n_0} \gg n_0$ and $n_1 \le e^{n_0^{b_0 / 2}}$, we have
$$\abs{\B} < e^{- n^b} + e^{- n^b} + e^{- n_0^{b_0}} <  n_1^{- p} \, ,$$
where $p = p (d)$ is chosen large enough.

If $x \notin \B$, then putting it all together we have
\begin{align*}
\frac{1}{n_1} \, \log \, \norm{ \Bn{n_1} (x) }  = & - \left[ 1 +  \bigo \left( \frac{n_0}{n_1}\right) \right] \ \left[ \LE{n_0}_1 (B) + \bigo (n_1^{-a/2}) \right] \\
& + \left[ 2 +  \bigo \left( \frac{n_0}{n_1}\right) \right] \ \left[ \LE{2 n_0}_1 (B) + \bigo (n_1^{-a/2}) \right] \\
& + \bigo \left( C_2 \, \frac{n_0}{n_1}\right)  + \bigo \left( e^{-  \ep_0 \, n_0}  \right) \\
& = -  \LE{n_0}_1 (B) + 2 \,  \LE{2 n_0}_1 (B) + \, \bigo (n_1^{-a/2}) \, .
\end{align*}

We have shown that if $x \notin \B$ then
\begin{equation}\label{indstep-eq10}
\un{n_1}{B} (x) =  -  \LE{n_0}_1 (B) + 2 \,  \LE{2 n_0}_1 (B) + \, \bigo (n_1^{-a/2}) \, .
\end{equation}

Note that by \eqref{summary-n0-eq1} and \eqref{summary-n0-eq3}
\begin{align*}
 -  \LE{n_0}_1 (B) + 2 \,  \LE{2 n_0}_1 (B)  & =  \LE{n_0}_1 (B) - 2 \left[  \LE{n_0}_1 (B) -  \LE{2 n_0}_1 (B) \right] \\
 &  \ge - C_1  - 4 \, \ep_0 \, ,
 \end{align*}
so if $x \notin \B$, then by~\eqref{indstep-eq10} 
$$\un{n_1}{B} (x) > - C_1 - 4 \ep_0 - \bigo (n_1^{-a/2})  > - 2 C_1 \, .$$

Moreover, as always, $\un{n_1}{B} (z) \le C_0$ for all $z \in \strip_r^d$.

Then by Lemma~\ref{L2bound-ssh} we have
\begin{equation}\label{uniform-l2bound-n1}
\norm{ \un{n_1}{B} }_{L^2 (\T^d)} \less C_1^2 < C_2 \, .
\end{equation}

We now show that $ \LE{n_1}_1 (B) \approx -  \LE{n_0}_1 (B) + 2 \,  \LE{2 n_0}_1 (B)$.
\begin{align*}
 \abs{  \LE{n_1}_1 (B)  +  \LE{n_0}_1 (B) - 2 \,  \LE{2 n_0}_1 (B) }   & \\
& \kern-10em = \babs{ \int_{\T^d} \, \left[  \un{n_1}{B} (x)   +  \LE{n_0}_1 (B) - 2 \,  \LE{2 n_0}_1 (B)  \right] \, d x   } \\
 & \kern-10em \le \int_{\B\comp} \, \abs{  \un{n_1}{B} (x)   +  \LE{n_0}_1 (B) - 2 \,  \LE{2 n_0}_1 (B) } \, d x \\
& \kern-10em +   \int_{\B} \, \abs{  \un{n_1}{B} (x)   +  \LE{n_0}_1 (B) - 2 \,  \LE{2 n_0}_1 (B) } \, d x\\
& \kern-10em \less  n_1^{- a/2} + C_2 \ \abs{ \B }^{1/2} \less n_1^{- a/2}  + C_2 \, n_1^{- p/2}  \less   n_1^{- a/2}  \, ,
\end{align*}
where to estimate the integral on $\B$ we used~\eqref{uniform-l2bound-n1} and Cauchy-Schwarz. 
 
Hence
\begin{equation}\label{indstep-eq11}
\abs{  \LE{n_1}_1 (B)  +  \LE{n_0}_1 (B) - 2 \,  \LE{2 n_0}_1 (B) }  \less  n_1^{- a/2}  \, .
\end{equation}

Combining \eqref{indstep-eq10} and \eqref{indstep-eq11} we obtain that if $x \notin \B$ then
$$\abs{ \un{n_1}{B} (x) - \LE{n_1}_1 (B) }  \less  n_1^{- a/2} < n_1^{- a/3} \, .$$

Since $\abs{\B} <  n_1^{- p} $, where $p$ is large enough, we conclude that the se\-pa\-ra\-tely subharmonic function $\un{n_1}{B}$ satisfies the weak a-priori estimate
$$\babs{ \{ x \in \T^d \colon  \abs{ \un{n_1}{B} (x) - \avg{ \un{n_1}{B} } }  >  n_1^{- a/3}  \} } < n_1^{- p} \, .$$

Lemma~\ref{splitting-ssh} is then applicable with $u =  \un{n_1}{B}$, $\ep_0 = n_1^{- a/3}$, $\ep_1 =  n_1^{- p}$, since items~\eqref{item1} and \eqref{item2} also hold.

The weak a-priori estimate is then boosted to
$$\babs{ \{ x \in \T^d \colon  \abs{ \un{n_1}{B} (x) - \avg{ \un{n_1}{B} } }  > n_1^{- \alpha \, a/3}  \} } < e^{- n_1^{ \beta \, a/3}} \, .$$

This establishes the uniform LDT estimate~\eqref{unif-ldt-eq} at scale $n_1$ with parameters $a_1 := \alpha \, a/3$ and $b_1 := \beta \, a/3$, where $a, \alpha, \beta$ depend only on the number of variables $d$. 

It is important to recall the provenance of these constants, namely the estimates in Theorem~\ref{qBET-ssh-thm} and Lemma~\ref{splitting-ssh} on {\em general} separately subharmonic functions. Therefore, these constants will {\em not} change as we continue this process inductively.

\subsubsection*{Summary of estimates at scales $\asymp n_1$.} Let us note that everything we have done at scale $n_1$ applies identically for any scale $m \asymp n_1$. 

We now summarize the estimates at scales $\asymp n_1$ that are needed at the next scale. 
Let $m, m' \asymp n_1$. Then for any cocycle $B$ with $\norm{B-A}_r \le \delta_0$ the following hold:
\begin{subequations}
\label{summary-scale-n1}
\begin{align}
\label{summary-n1-eq1bis}
  \norm{\un{m}{B}}_{L^2 (\T^d)} & < C_2 \\
  \label{summary-n1-eq1}
  \abs{\LE{m}_1 (B)} & < C_1 \\
   \label{summary-n1-eq3}
   \abs{ \LE{m}_1 (B) - \LE{m'}_1 (B) } & < n_1^{-a/3}  \\
   \label{summary-n1-eq2}
  \babs{ \LE{m}_1 (B) - \LE{m}_1 (A) }  & < 9 \, \ep_0 + n_1^{- a/3} < 10 \, \ep_0 \\
   \label{summary-n1-eq4}
   \abs{ \frac{1}{m} \log \norm{\Bn{m} (x)} - \LE{m}_1  (B)  } & \le  \, m^{- a_1} ,
\end{align}
\end{subequations}
where the last estimate holds for all $x \notin \B_m (B)$, with $\abs{ \B_m (B)} < e^{- m^{b_1}}$.

\smallskip

Let us explain why these estimates do indeed hold. First off, \eqref{summary-n1-eq1bis} corresponds to \eqref{uniform-l2bound-n1}, while  \eqref{summary-n1-eq4} is the uniform LDT just proven above for the scale $n_1$, and so for all other similar scales.

Moreover, \eqref{indstep-eq11} also holds for all $m \asymp n_1$ and we have
\begin{equation}\label{indstep-eq12}
\abs{  \LE{m}_1 (B)  +  \LE{n_0}_1 (B) - 2 \,  \LE{2 n_0}_1 (B) }  \less n_1^{- a/2}  ,
\end{equation}
which implies
$ \abs{ \LE{m}_1 (B)  -  \LE{m'}_1 (B) } \less n_1^{- a/2} < n_1^{-a/3}$, 
 justifying \eqref{summary-n1-eq3}.

We rewrite  \eqref{indstep-eq12} as
$$\babs{  \LE{m}_1 (B)  -  \LE{n_0}_1 (B) + 2 \, \left[ \LE{n_0}_1 (B) - \LE{2 n_0}_1 (B) \right] }  \less n_1^{- a/2}$$
and apply it  to both $B$ and $A$. Together with  \eqref{summary-n0-eq2} and \eqref{summary-n0-eq3} we get
\begin{align*}
\abs{  \LE{m}_1 (B)  -  \LE{m}_1 (A) }   & \le  \abs{  \LE{n_0}_1 (B)  -  \LE{n_0}_1 (A) }  \\
& \kern-1.5em + \babs{  \LE{m}_1 (B)  -  \LE{n_0}_1 (B) + 2 \, \left[ \LE{n_0}_1 (B) - \LE{2 n_0}_1 (B) \right] }  \\
& \kern-1.5em + \babs{  \LE{m}_1 (A)  -  \LE{n_0}_1 (A) + 2 \, \left[ \LE{n_0}_1 (A) - \LE{2 n_0}_1 (A) \right] }  \\
& \kern-1.5em + 2 \abs{  \LE{n_0}_1 (B)  -  \LE{2 n_0}_1 (B) } + 2 \, \abs{  \LE{n_0}_1 (A)  -  \LE{2 n_0}_1 (A) } \\
& \kern-1.5em \le \ep_0 + \bigo (n_1^{-a/2}) + \bigo (n_1^{-a/2}) + 4 \ep_0 + 4 \ep_0 \\
& \kern-1.5em < 9 \ep_0 + n_1^{-a/3} < 10 \ep_0 .
\end{align*}
This  justifies \eqref{summary-n1-eq2}, which then implies
$$\abs{ \LE{m}_1 (B) } < 10 \, \ep_0 + \abs{ \LE{m}_1 (A) } < 10 \, \ep_0 + C_0 < C_1 ,$$
proving \eqref{summary-n1-eq1}.

\subsection*{The next scales.} We explain how the inductive procedure continues. The argument is identical to the one used to derive the estimates~\eqref{summary-scale-n1} at scales $\asymp n_1$ from the corresponding estimates~\eqref{summary-scale-n0} at scales $\asymp n_0$.

For every scale $n_1$ in the range prescribed earlier, let the next scale $n_2$ be such that 
$n_1^{p'} \le n_2 \le e^{n_1^{b_1/2}}$. Break down the block $\Bn{n_2} (x)$ into blocks of lengths $\asymp n_1$ and set out to apply the avalanche principe to the resulting chain of matrices. 

The ``gap'' condition for blocks of lengths $\asymp n_1$ is ensured by  the uniform LDT~\eqref{summary-n1-eq4} and the estimate~\eqref{summary-n1-eq2} proven for scales $m \asymp n_1$. 
The proof uses the uniform upper semicontinuity of the maximal LE in Proposition~\ref{n-unif-usc} (which holds for all scales) and the hypothesis on the existence of a gap between the first two LE of $A$. It proceeds as at the previous scale,  the only difference being that from~\eqref{summary-n1-eq4}, for $m \asymp n_1$, \eqref{proof-gaps-n0} becomes instead
$$
\un{m}{B} (x) =  \frac{1}{m} \, \log \norm{\Bn{m} (x)}  > \LE{m}_1 (B) - m^{-a_1}  .
$$
Together with estimate~\eqref{summary-n1-eq2} (which compared to ~\eqref{summary-n0-eq2} has the factor $10$), this shows that the analogue of \eqref{angles-eq2} is
$$
 \frac{1}{m} \, \log \norm{\Bn{m} (x)} > \LE{m}_1 (A) - 10 \, \ep_0  - m^{-a_1} \ge L_1 (A) - 11 \, \ep_0 .
$$

Moreover,  the ``angles" condition for blocks of lengths  $\asymp n_1$ is ensured by ~\eqref{summary-n1-eq4} and ~\eqref{summary-n1-eq3} exactly the same way they were derived at the previous scale. In fact, the bound obtained will be stronger, because ~\eqref{summary-n1-eq4} and ~\eqref{summary-n1-eq3} are sharper than their counterparts ~\eqref{summary-n0-eq4} and ~\eqref{summary-n0-eq3} at scale $\asymp n_0$.

After using the AP, the next step is to apply the quantitative Birkhoff Ergodic Theorem~\ref{qBET-ssh-thm} to the separately subharmonic functions $\un{m}{B} (z)$ with $m \asymp n_1$ and translation vector $\om_1 := n_1 \, \om$. The assumptions in this theorem hold as follows. 
The upper bound $\sup_{z \in \strip_r^d} \, \un{m}{B} (z) \le C_0$ is always true, ensuring that item~\eqref{item1} holds. Estimate~\eqref{summary-n1-eq1} implies $\babs{ \avg{\un{m}{B}}} = \babs{\LE{m}_1 (B)} < C_1$, hence for some $x_0 \in \T^d$, $\un{m}{B} (x_0) > - C_1$, showing that item~\eqref{item2} also holds. Finally, because $n_2 \ge n_1^{p'}$, the number $n \asymp \frac{n_2}{n_1}$ of iterates is large enough relative to the parameter defining the Diophantine condition satisfied by $\om_1$.

The analogue of the bookkeeping in~\eqref{summary-scale-n1} for the next scale is derived in the same way, where the bounds in~\eqref{summary-n1-eq1}, ~\eqref{summary-n1-eq1bis} and ~\eqref{summary-n1-eq4} do not change, while the one in~\eqref{summary-n1-eq3} becomes sharper as a result of the scale increase: $\abs{  \LE{m}_1 (B)  -  \LE{m'}_1 (B) } < n_2^{-a/3}$ for all $m, m' \asymp n_2$.

We provide more details regarding the derivation of the analogue of~\eqref{summary-n1-eq2} at scale $n_2$, as this is the place where the estimate worsens slightly from scale to scale, although the additional errors form a summable series that does not exceed $10 \ep_0$. We have
\begin{align*}
\abs{  \LE{m}_1 (B)  -  \LE{m}_1 (A) }   & \le  \abs{  \LE{n_1}_1 (B)  -  \LE{n_1}_1 (A) }  \\
& \kern-1.5em + \babs{  \LE{m}_1 (B)  -  \LE{n_1}_1 (B) + 2 \, \left[ \LE{n_1}_1 (B) - \LE{2 n_1}_1 (B) \right] }  \\
& \kern-1.5em + \babs{  \LE{m}_1 (A)  -  \LE{n_1}_1 (A) + 2 \, \left[ \LE{n_1}_1 (A) - \LE{2 n_1}_1 (A) \right] }  \\
& \kern-1.5em + 2 \abs{  \LE{n_1}_1 (B)  -  \LE{2 n_1}_1 (B) } + 2 \, \abs{  \LE{n_1}_1 (A)  -  \LE{2 n_1}_1 (A) } \\
& \kern-1.5em \le (9 \, \ep_0 +  n_1^{- a/3}) + \bigo (n_2^{- a/2}) + \bigo (n_2^{- a/2})  + 2 \,   n_1^{- a/3} + 2 \,   n_1^{- a/3} \\
& \kern-1.5em =  9 \, \ep_0  + 5 \,  n_1^{- a/3} + \bigo  (n_2^{- a/2})  < 9 \, \ep_0  + 5 \,  n_1^{- a/3} + n_2^{- a/3}  \\
& \kern-1.5em <  9 \, \ep_0  + 5 \,   \sum_{k=1}^{\infty} \,  \, n_k^{- a/3} 
<  9 \, \ep_0  + 10   \,  n_1^{- a/3}  < 10 \, \ep_0 \, .
\end{align*}

The argument continues the same way with scales 
$$ \ldots \gg n_k \gg \ldots \gg n_3 \gg n_2 \gg n_1\, $$
chosen such that $n_k^{p'} \le n_{k+1} \le e^{n_{k}^{b_1/2}}$, hence their ranges overlap. Therefore, the uniform LDT~\eqref{unif-ldt-eq} holds for all $n \ge \underline{n_1}:= n_0^{p'}$. Furthermore, the uniform $L^2$ bound~\eqref{uniform-l2-bound-thm} holds by~\eqref{summary-n1-eq1bis} and its analogues at higher scales. Strictly speaking, we have derived it only for $n \ge \underline{n_1}$. However, at the cost of decreasing slightly the size $\delta_0$ of the neighborhood around $A$, we may assume that~\eqref{proximity-first-scale} holds in fact for all $m\le n_0^{p'}$, hence the argument following this estimate ensures the uniform $L^2$ bound also at scales $m \le \underline{n_1}$. 
\end{proof}

%%%%%%%%%%%%%%%%%%%%%%%%%%%%%%%%%%%%%%%%%%%%%%%%%%%%%%%%%%%%%%%%
%%%%%%%%%%%%%%%%%%%%%%%%%%%%%%%%%%%%%%%%%%%%%%%%%%%%%%%%%%%%%%%%

\section{The proofs of the main statements}\label{proofs}
Large deviations type estimates for iterates of linear cocycles can be used to establish the continuity of the corresponding Lyapunov exponents. This was the subject of our monograph~\cite{DK-book}. The crucial component of the continuity argument was the {\em uniformity} of the estimates in the cocycle.

Given any cocycle $A \in \qpcmat{d}$ with $L_1 (A) > L_2 (A)$, in Theorem~\ref{u-fiber-ldt-thm} we proved uniform fiber LDT estimates in a neighborhood of $A$; moreover, we established that $A$  is uniformly $L^2$-bounded, i.e. the uniform estimate~\eqref{uniform-l2-bound-thm} on the iterates of any nearby cocycle.

The abstract continuity theorem (ACT) in Chapter 3 of this book  contains additional assumptions. However, in our present setting of analytic, quasi-periodic cocycles\textemdash in fact for any space of cocycles over a uniquely ergodic base dynamics on a compact metric space\textemdash they are automatically satisfied (the reader may consult Section 6.4 in~\cite{DK-book} for the complete argument in the non-identical singular case).  Therefore, the ACT is applicable and it establishes Theorem~\ref{cont-le}, where the weak-H\"older mo\-du\-lus of continuity is a consequence of the sub-exponential rate of decay in the LDT~\eqref{unif-ldt-eq}.

The continuity of the Oseledets filtration and decomposition are si\-mi\-larly consequences of an abstract statement we derived in Chapter 4 of \cite{DK-book}. The assumptions in this statement are the same as for the continuity of the Lyapunov exponents, 
hence Theorem~\ref{cont-oseledets} is also established.

\medskip

Let us move on to the applications of the continuity theorem of the Lyapunov exponents for identically singular cocycles to the positivity and simplicity of the Lyapunov exponents.

 We begin with the proof of Theorem~\ref{pos-sim-thm}. 

For every $\delta \in \R$ define the cocycle
$$S_\delta :=  \left[ \begin{array}{cc} M & \delta \,  N  \\
\delta \,  P &  \delta \,  Q \\  \end{array} \right] .$$

Note that $S_0 = \left[ \begin{array}{cc} M &  0  \\
0 &   0 \\  \end{array} \right]$ is {\em identically singular}.

We can write 
$$A_{\la} = \left[ \begin{array}{cc} \la \, M &   N  \\
P &   Q \\  \end{array} \right]  = \la \, \left[ \begin{array}{cc} M & \frac{1}{\la} \,  N  \\
\\
\frac{1}{\la} \,  P &  \frac{1}{\la} \,  Q \\  \end{array} \right] = \la \, S_{\frac{1}{\la}} ,$$
so for every $1 \le k \le m$, $L_k (A_\la) = \log \sabs{\la} + L_k (S_{\frac{1}{\la}} )$.

As $\sabs{\la} \to \infty$, $S_{\frac{1}{\la}} \to S_0 = \left[ \begin{array}{cc} M &  0  \\
0 &   0 \\  \end{array} \right]$, and $L_l (S_0) = L_l (M) > - \infty$ (because $\det [M (x)] \not \equiv 0$). 

Then if  $\sabs{\la}$ is large enough, by the continuity Theorem~\ref{cont-le} we have
$$L_l (A_\la) = \log \sabs{\la} + L_l (S_{\frac{1}{\la}}) > \log \sabs{\la}  + L_l (M) - 1,$$
which establishes~\eqref{pos-sim-eq1} with $C_0 :=  - L_l (M) + 1 < \infty$. 

\smallskip

The simplicity statement in item (b) of the theorem follows the same way.
The  $l$ largest Lyapunov exponents of $S_0$ are exactly the Lyapunov exponents of $M$, which are assumed to be simple. Then the {\em quantitative} part of Theorem~\ref{cont-le} applies to each of these exponents, so there is a weak-H\"older modulus of continuity function $w (h)$ so that for $\sabs{\delta} \ll 1$ and for every $1 \le k \le l$ we have
$$\abs{L_k (S_\delta) - L_k (M)}  = \abs{L_k (S_\delta) - L_k (S_0)} \le w (\sabs{\delta}) \to 0 \quad \text{as } \delta \to 0.$$

This then translates into
$$\abs{L_k ( A_{\la} ) - \log \sabs{\la} - L_k (M) } \le w \big(\frac{1}{\sabs{\la}}\big) ,$$
thus proving~\eqref{pos-sim-eq2}.

\medskip

The proof of Theorem~\ref{pos-energy-thm} proceeds the same way as that of Proposition~\ref{sorets-spencer-thm}.
We first factor out $\la$ to write
$$A_{\la, E} =  \la \, \left[ \begin{array}{ccc} U (x) \, (F (x) + \frac{1}{\la} \, R (x) - \frac{E}{\la} \, I) & &  \frac{1}{\la} \, N  \\
\\
 \frac{1}{\la} \,  P & &   \frac{1}{\la} \,  Q \\  \end{array} \right] .$$

Make the change of coordinates $\delta= \frac{1}{\la}$, $s = \frac{E}{\la}$, and consider the cocycle
$$S_{\delta, \, s} (x) := \left[ \begin{array}{ccc} U (x) \, (F (x) + \delta \, R (x) - s \, I) & &  \delta \, N  \\
\\
 \delta \,  P & &    \delta \,  Q \\  \end{array} \right] ,$$
so we have $A_{\la, E} =  \la \, S_{\frac{1}{\la}, \, \frac{E}{\la}}$.

Then
$$L_l (A_{\la, E}) = \log \sabs{\la} + L_l (S_{\frac{1}{\la}, \, \frac{E}{\la}}) .$$

For every $s \in \R$, 
$S_{0, \, s} (x) =  \left[ \begin{array}{cc} U (x) \, (F (x) - s \, I) & 0  \\
\\
0 &    0 \\  \end{array} \right] $, so
$$L_l (S_{0, \, s}) = L_l \big( U (x) \, (F (x) - s \, I) \big) > - \infty$$
because by our assumptions, $\det [ U (x) \, (F (x) - s \, I) ]  \not \equiv 0$.

By the continuity Theorem~\ref{cont-le}, the map $(\delta, s) \mapsto L_l (S_{\delta, \, s})$ is con\-ti\-nuous. Therefore, locally near every point $(0, s)$, this map has a finite lower bound. By compactness, given any compact interval $I$, there are $\delta_0 > 0$ and $C_0 < \infty$ such that on $[ - \delta_0, \delta_0 ] \times I$, the map  $(\delta, s) \mapsto L_l (S_{\delta, \, s})$ is bounded from below by $- C_0$. 

Translating this back, it follows that for $\sabs{\la}$ large enough, if $\frac{E}{\la}$ is bounded, say $\sabs{E} \le 2 \sabs{\la} \norm{F}_r$, then
$$L_l (A_{\la, E}) > \log \sabs{\la} - C_0 .$$

Next we factor out $E$, to get
$$A_{\la, E} =  E \, \left[ \begin{array}{ccc} U (x) \, ( \frac{\la}{E} \, F (x) +  \frac{1}{E} \, R -  I) & &  \frac{1}{E} \, N  \\
\\
 \frac{1}{E} \,  P & &   \frac{1}{E} \,  Q \\  \end{array} \right] $$
and make the change of variables $\delta= \frac{1}{E}$, $s = \frac{\la}{E}$.

By the same continuity and compactness argument, we conclude that for $ \frac{\la}{E}$ bounded, say $\sabs{\la} \le 2 \sabs{E} \norm{F}_r$, and  for $\sabs{E}$ large enough (which would happen if we chose $\sabs{\la}$ large enough), we have
 $$L_l (A_{\la, E}) > \log \sabs{E} - C_0' > \log \sabs{\la} - C_0 ,$$
which completes the proof of the theorem. \qed

%%%%%%%%%%%%%%%%%%%%%%%%%%%%%%%%%%%%%%%%%%%%%%%%%%%%%%%%%%%%%%%%
%%%%%%%%%%%%%%%%%%%%%%%%%%%%%%%%%%%%%%%%%%%%%%%%%%%%%%%%%%%%%%%%

\section{Consequences for  block Jacobi operators}\label{mathphys}
\newcommand{\newA}{\widetilde{A}_{\la, E}}

In this section we present some immediate applications of our main statements to block Jacobi operators (also called strip or band lattice operators). These types of operators  generalize the one-dimensional lattice Schr\"odinger operator described in Section~\ref{introduction}.  Our applications are concerned with the positivity, continuity and (local) simplicity of the Lyapunov exponents of the corresponding eigenvalue equation, and also with the continuity of the integrated density of states (IDS).

\smallskip

 Let us begin by describing a block Jacobi operator. 

Fix a translation vector $\om \in \T^d$. Let $W, R, F \in \qpcmatl{d}$. Assume that for all phases $x \in \T^d$, $R(x)$ and $F(x)$ are symmetric matrices, that $W$ is not identically singular and denote by $W^T (x)$ the transpose of the matrix $W (x)$. Moreover, for all $n \in \N$, denote 
\begin{equation}\label{wnrndn}
W_n (x) := W (x+n \om), R_n (x) := R (x+n \om), F_n (x) := F (x+n \om)
\end{equation}

A {\em quasi-periodic block Jacobi operator} is an operator  $H = H_{\la} (x)$ acting on $l^2 (\Z, \R^l)$ by
\begin{equation}\label{J-op}
[ H_{\la} (x)  \, \vpsi ]_n := - (W_{n+1} (x) \, \vpsi_{n+1} + W^{T}_{n} (x) \, \vpsi_{n-1} + R_{n} (x) \, \vpsi_{n}) + \la \, F_n (x) \, \vpsi_n ,
\end{equation}
where $\vpsi = \{ \vpsi_n \}_{n\in\Z} \in l^2 (\Z, \R^l)$ is any state, $x \in \T^d$ is a phase that introduces some randomness into the system and $\la \neq 0$ is a coupling constant.

This model contains all quasi-periodic, finite range hopping Schr\"{o}\-din\-ger operators on integer or band integer lattices (which in some sense may be regarded as approximations of higher dimensional lattices). The hopping term is given by the ``weighted'' Laplacian: 
\begin{equation}\label{w-Laplace}
[ \Delta_W (x) \, \vpsi]_n :=   - W_{n+1} (x) \, \vpsi_{n+1} + W^{T}_{n} (x) \, \vpsi_{n-1} + R_{n} (x) \, \vpsi_{n}
\end{equation}
where the hopping amplitude is encoded by the quasi-periodic matrix valued functions $W_n (x)$ and $R_n (x)$.

The potential is given by the quasi-periodic matrix valued function $\la \, F_n (x)$. 

The more relevant situation from a physical point of view  is when the potential function $F (x)$ is a {\em diagonal} matrix, while the entries of the weight $W (x)$ are trigonometric polynomials.

The associated Schr\"{o}dinger equation 
$$
H_{\la} (x) \, \vpsi = E \, \vpsi
$$
for a (generalized) state  $\vpsi = \{ \vpsi_n \}_{n\in\Z} \subset \R^l$ and energy $E \in \R$, gives rise to a cocycle $A_{\la, E} (x)$ of dimension $m= 2 l$. Let $L_k (E) = L_k (A_{\la, E}) $ denote its $k$-th Lyapunov exponent.
An easy calculation shows that
\begin{equation*} \label{J-cocycle}
 A_{\la, E} (x) = 
 \left[\begin{array}{cc}
W^ {-1} (x + \om) \, (\la F (x) + R (x) - E \, I)      &   - W^{-1} (x + \om) \, W^{T} (x)\\ 
\\
I  &  0\end{array}\right] 
\end{equation*}
and that this cocycle can be conjugated to a symplectic cocycle (see Section 8 in~\cite{DK1} for more details).

Note that since $W (x)$ is analytic and $\det [ W (x) ] \not \equiv 0$, $W^{-1} (x)$ exists almost everywhere, so the cocycle $A_{\la, E} (x)$ is defined almost everywhere.

Then the Lyapunov exponents are well defined and they satisfy the relations
$$L_1 (E) \ge \ldots \ge L_l (E) \ge 0 \ge L_{l+1} (E) \ge \ldots \ge L_{2 l} (E) > - \infty$$ 
and $L_{2 l + 1 - k} (E) = - L_k (E)$ for all $1 \le k \le l$.

We are ready to formulate the statement.

\begin{theorem}\label{pos-band-op} 
Assume that $\om \in \DC_t$. Then all Lyapunov exponents of the operator~\eqref{J-op} depend continuously on the data, i.e. on $E, \la \in \R$ as well as on $W, R, F \in \qpcmatl{d}$.

Furthermore, given such matrix-valued functions $W, F, R$, there is a constant
$\la_0 = \la_0 (t, W, F, \norm{R}_r) < \infty$ so that if we fix $\la$ with $\sabs{\la} \ge \la_0$, the following hold.

\smallskip

(i) (Positivity) If $W$ is not identically singular and if $F$ has no constant eigenvalues, then there is $C_0 = C_0 (W, F) < \infty$ such that
$$L_l (E) = L_l (A_{\la, E}) > \log \sabs{\la} - C_0 \quad \text{ for all } E \in \R .$$ 

(ii) (Continuity) If  $W$ is not identically singular and if $F$ has no constant eigenvalues, then the block
$L_1 + \ldots +L_l $ is locally weak-H\"older continuous.

\smallskip

(iii) (Simplicity) Assume that the weight $W$ is one-dimensional, i.e. 
$ W (x) = h (x) \, I$ where $ h  \in \analyticf{d}$ with $h (x) \not \equiv 0$. Assume moreover that $F - s \, I$,  seen as an $l$-dimensional cocycle,
has simple Lyapunov exponents for all $s \in \mathcal{E}$, where $\mathcal{E}$ is a compact interval.  

Then the Lyapunov exponents $L_k (E)$ of the operator $H_\la (x)$ are simple for all energies $E \in \la \, \mathcal{E}$.
More precisely, there is a constant $\ka_0 = \ka_0 (t, h, F, \norm{R}_r, \mathcal{E}) > 0$, such that for all $1 \le k < l$ and $E \in \la \mathcal{E}$, 
$$L_k (E) - L_{k+1} (E) > \ka_0 .$$
\end{theorem}

\begin{proof} The general theorems formulated in Section~\ref{introduction} do not directly apply to the cocycle $A_{\la, E}$ above, because it is not defined (and analytic) everywhere. This can be easily remedied by multiplying it with a one-dimensional cocycle.

Indeed, by Cramer's formula, whenever $\det [ W (x) ] \neq 0$, we have
$\displaystyle W^{-1} (x) = \frac{1}{\det [ W (x) ]}  \, \adj  ( W (x) )$.
 
Let $g (x) := \det [ W (x) ]$, so $g \in \analyticf{d}$ and $g (x) \neq 0$ a.e. Define 
$$\newA (x) := g (x+\om) \, A_{\la, E} (x) = \left[ \begin{array}{cc} M_{\la, E} (x) &   N (x)  \\
\\
P (x) &   Q (x) \\  \end{array} \right] ,$$
where
\begin{align*}
M_{\la, E} (x) & := U (x) \, (\la \, F (x) + R (x) - E \, I) \\
U (x) & := \adj ( W (x+\om) ) \\
N (x) & := - \adj ( W (x+\om) ) \, W^T (x) \\
P (x) & := g (x+\om) \, I \\
Q (x) & \equiv 0 .
\end{align*}

Then clearly $\newA \in \qpcmat{d}$ where $m=2 l$,  and for all $1 \le k \le m$ we have
\begin{equation}\label{jacobi-proof-eq1}
L_k (\newA) = \int_{\T^d} \log \sabs{g (x)} \, d x + L_k (A_{\la, E}) .
\end{equation}

By  the continuity Theorem~\ref{cont-le}, all Lyapunov exponents are con\-ti\-nuous functions on $\qpcmat{d}$. The cocycle $\newA$ depends continuously on $E, \la$ as well as on $W, R, F$ and so the same is true for its Lyapunov exponents. Moreover, the maps $W \mapsto \det [W]$ and $g \mapsto \int_{\T^d} \log \sabs{g}$ are continuous (the latter may be regarded as the continuity of the Lyapunov exponent of the one-dimensional cocycle $g$).   
By way of formula~\eqref{jacobi-proof-eq1}, these observations establish the first continuity statement of the theorem.

\medskip

Now we fix $W, R, F \in \qpcmatl{d}$.

Since $g$ is analytic and $g \not \equiv 0$, $\int_{\T^d} \log \sabs{g (x)} \, d x = C (W) > - \infty$.

Clearly 
$\det [ U (x) ] = \det [ \adj (W (x+\om) ) ] = \det [ W (x+\om) ]^{l-1} \not \equiv 0$.
Since, moreover, $F$ has no constant eigenvalues, Theorem~\ref{pos-energy-thm}  is applicable to $\newA$, hence there is $\la_0$ depending on the fixed data, so that
$$L_l (\newA ) > \log \sabs{\la} - C_0 \quad \text{for all } E \in \R  \ \text{ and for all } \sabs{\la} \ge \sabs{\la_0} .$$ 
This, together with \eqref{jacobi-proof-eq1}, establishes the conclusion of item (i). 

\smallskip

Now fix $\la$ with $\sabs{\la} \ge \sabs{\la_0}$.
By item (i), for all energy parameters $E \in \R$,
$L_l (E) = L_l (A_{\la, E}) > \log \sabs{\la} - C_0 > 0$ (we may increase $\la_0$ if necessary). 
Therefore, 
$$L_l (E) > L_{l+1} (E) \, ( = - L_l (E) ) $$
holds for all $E \in \R$. 

The quantitative statement in Theorem~\ref{cont-le} is then applicable, and we conclude that the map $E \mapsto (L_1 + \ldots + L_l) (E)$ is locally weak-H\"older continuous, which establishes item (ii). 

\smallskip

We now proceed with the proof of item (iii). Since here we assume that $W (x) = h (x) \, I$, so $W^{-1} (x) = \frac{1}{h (x)} \, I$ and $W^T (x) = h (x) \, I$, we can write
\begin{align*}
h (x+\om) \, A_{\la, E} (x) & = \left[ \begin{array}{cc} \la F (x) + R (x) - E \, I &   - h (x) \, I \\
\\
h (x+\om)  \, I &   0 \\  \end{array} \right] \\
\\
& = \la \,  \left[ \begin{array}{cc} F (x) + \frac{1}{\la} \, R (x) - \frac{E}{\la} \, I &   - \frac{1}{\la} \, h (x)   \, I \\
\\
\frac{1}{\la} \, h (x+\om)  \, I &   0 \\  \end{array} \right] .
\end{align*}

Make the change of variables $\delta = \frac{1}{\la}$, $s = \frac{E}{\la}$ and define the cocycle
$$S_{\delta, \, s} (x) :=  \left[ \begin{array}{cc} F (x) + \delta \, R (x) - s  \, I &   -  \delta \, h (x)  \, I  \\
\\
\delta \, h (x+\om)  \, I &   0 \\  \end{array} \right] .$$
 
 Then clearly $S_{\delta, \, s} \in \qpcmat{d}$ and for all $1 \le k \le m$,
 $$L_k (E) = L_k (A_{\la, E}) = - \int_{\T^d} \log \sabs{h (x)} \, d x + \log \sabs{\la} + L_k (S_{\frac{1}{\la}, \, \frac{E}{\la}}) .$$
 
 Since  $h$ is analytic and $h \not \equiv 0$, we have that $\int_{\T^d} \log \sabs{h (x)} \, d x \in \R$. 
 
 Then for all $E \in \R$, $\la \neq 0$ and $1 \le k < l$, 
 \begin{equation}\label{jacobi-proof-eq2}
 L_k (E) - L_{k+1} (E) = L_k (S_{\frac{1}{\la}, \, \frac{E}{\la}})  - L_{k+1} (S_{\frac{1}{\la}, \, \frac{E}{\la}}) .
 \end{equation}

As $\delta \to 0$, we have $S_{\delta, \, s}  \to S_{0, \, s} = \left[ \begin{array}{cc} F  - s  \, I &   0  \\
0 &   0 \\  \end{array} \right] $, so by the con\-ti\-nuity Theorem~\ref{cont-le}, if $1 \le k \le l$ then $$L_k (S_{\delta, \, s})  \to L_k (S_{0, \, s}) = L_k (F-s \, I) .$$ 
 
 We assumed that for all $s \in \mathcal{E}$, the cocycle $F -s \, I$ has simple Lyapunov exponents. Then the $l$ largest Lyapunov exponents of $S_{0, \, s}$ are simple, and so for every fixed $s \in \mathcal{E}$ and every $\delta$ near $0$, $S_{\delta, \, s}$ has simple Lyapunov exponents as well.
 
By the compactness of $\mathcal{E}$ there are $\delta_0 > 0$ and $\ka_0 > 0$ such that if $(\delta, s) \in [- \delta_0, \delta_0] \times \mathcal{E}$ and $ 1\le k < l$, then
$$L_k (S_{\delta, \, s}) - L_{k+1} (S_{\delta, \, s}) > \ka_0 .$$
Together with \eqref{jacobi-proof-eq2} this proves the claim in item (iii).
\end{proof}

\begin{remark}
If for some $s_0 \in \R$, the cocycle $F - s_0 \, I$ has simple Lyapunov exponents, then by the continuity Theorem~\ref{cont-le} this also holds in a whole neighborhood $\mathcal{E}$ of $s_0$. However, this property (and even more so its global counterpart) is difficult to verify, even in the case when $F(x)$ is a diagonal matrix-valued function (unless, of course, the ranges of the diagonal entries are apart from each other). The simplicity statement in the theorem above is therefore not entirely satisfactory.
\end{remark}

Next we apply item (ii)\textemdash the quantitative continuity result on Lyapunov exponents\textemdash  in the above theorem to derive a modulus of continuity for the integrated density of states (IDS) of the block Jacobi operators $H_\la (x)$. 

Let us recall a couple of facts. 

Denote by $P_n$ the coordinate restriction operator to $\{1, 2, \ldots, n\} \subset \Z$, and let $H_\la^{(n)} (x) := P_n \, H_\la (x) \, P_n^\ast$. It is well know that the following limit exists and it is a.e. independent of $x \in \T^d$: 
$$N (E) = N_\la (E) := \lim_{n\to\infty} \, \frac{1}{n} \, \# \big( (- \infty, E] \cap \text{ Spectrum of } H_\la^{(n)} (x) \big)\,.$$
The function $E \mapsto N_\la (E)$ is called the {\em integrated density of states} of the family of ergodic operators $\{H_\la (x) \colon x \in \T^d\}$.

The IDS is known to be $\log$-H\"older continuous in a very general setting (see \cite{Craig-Simon-IDS}). 
We show that in our context it is in fact weak-H\"older continuous. 

\smallskip

It is a classic result in the theory of lattice {\em Schr\"odinger operators} like~\eqref{ s op}  that the Lyapunov exponents and the IDS are related essentially via the Hilbert transform:
$$L (E) = \int_\R \log \abs{E-E'} d N (E')\,.$$ 
This relation is called the Thouless formula.

Recently, J. Chapman and G. Stolz (see \cite{C-Stolz}) obtained a Thouless-type formula (relating the LE and the IDS) which is applicable in the full generality of the operator~\eqref{J-op}. 

Assuming that $C (W) := \int_{\T^d} \log \, \sabs{\det [W (x)] } d x > - \infty$, their result states that\footnote{Strictly speaking, the result in \cite{C-Stolz} is formulated under a stronger assumption, but a private conversation with the authors revealed that it holds, in fact, under this more general condition.} 
\begin{equation}\label{general Thouless} 
(L_1 + \ldots + L_l) (E) = l \,  \int_\R \log \sabs{E-E'} d N (E') - C (W) .
\end{equation}

Our assumption on the weight $W (x)$ clearly implies that $C (W)$ is finite. Then by a standard argument involving some properties of the Hilbert transform (see the end of Section 10 in~\cite{GS-Holder}), the relation~\eqref{general Thouless} above allows us to transfer the local weak-H\"older modulus of continuity of the map $L_1 + \ldots + L_l$ to the IDS $N (E)$. We then conclude with the following.

\begin{corollary}
\label{Jacobi IDS} Consider the block Jacobi operator~\eqref{J-op} and assume that  the weight $W$ is not identically singular and that the potential function $F$ has no constant eigenvalues. Then its integrated density of states is weak-H\"older continuous, provided $\sabs{\la}$ is large enough depending on the data.
\end{corollary}

\begin{remark}
A similar argument also gives a local version of this result. We fix the coupling constant $\la$. In other words, let
$$[ H (x)  \, \vpsi ]_n := - (W_{n+1} (x) \, \vpsi_{n+1} + W^{T}_{n} (x) \, \vpsi_{n-1} + R_{n} (x) \, \vpsi_{n}) + F_n (x) \, \vpsi_n $$
be a block Jacobi operator defined by analytic matrix-valued functions and a Diophantine translation vector. If for some energy $E_0$ we have that $L_l (E_0) > 0$, then by the continuity Theorem~\ref{cont-le} this holds on a small compact neighborhood $\mathcal{E}$ of $E_0$. Moreover, by the same continuity theorem the map $\mathcal{E}  \ni E \to (L_1 + \ldots L_l) (E)$ is weak-H\"older continuous, and so is the IDS $N (E)$. 
\end{remark}

%%%%%%%%%%%%%%%%%%%%%%%%%%%%%%%%%%%%%%%%%%%%%%%%%%%%%%%%%%%%%%%%
%%%%%%%%%%%%%%%%%%%%%%%%%%%%%%%%%%%%%%%%%%%%%%%%%%%%%%%%%%%%%%%%

\subsection*{Acknowledgments}

The first author was supported  by Funda\c{c}\~{a}o para a Ci\^{e}ncia e a Tecnologia, under the project: UID/MAT/04561/2013.

The second author was supported by the Norwegian Research Council project no. 213638, ``Discrete Models in Mathematical Analysis''.

Both authors would like to thank the anonymous referees for their valuable suggestions.

\bigskip

\def\cprime{$'$} \def\cprime{$'$}
\providecommand{\bysame}{\leavevmode\hbox to3em{\hrulefill}\thinspace}
%\providecommand{\MR}{\relax\ifhmode\unskip\space\fi MR }
%\providecommand{\MRhref}[2]{
%  \href{http://www.ams.org/mathscinet-getitem?mr=#1}{#2}
%}
\providecommand{\href}[2]{#2}


\begin{thebibliography}{10}

\bibitem{AJS}
Artur \'Avila, Svetlana Jitomirskaya, and Christian Sadel, \emph{Complex
  one-frequency cocycles}, J. Eur. Math. Soc. (JEMS) \textbf{16} (2014), no.~9,
  1915--1935. %\MR{3273312}

\bibitem{AV-sim2}
Artur \'Avila and Marcelo Viana, \emph{Simplicity of {L}yapunov spectra: a
  sufficient criterion}, Port. Math. (N.S.) \textbf{64} (2007), no.~3,
  311--376. %\MR{2350698}

\bibitem{AV-sim1}
\bysame, \emph{Simplicity of {L}yapunov spectra: proof of the
  {Z}orich-{K}ontsevich conjecture}, Acta Math. \textbf{198} (2007), no.~1,
  1--56. %\MR{2316268}

\bibitem{BV-sim}
C.~Bonatti and M.~Viana, \emph{Lyapunov exponents with multiplicity 1 for
  deterministic products of matrices}, Ergodic Theory Dynam. Systems
  \textbf{24} (2004), no.~5, 1295--1330. %\MR{2104587}

\bibitem{B}
J.~Bourgain, \emph{Green's function estimates for lattice {S}chr\"odinger
  operators and applications}, Annals of Mathematics Studies, vol. 158,
  Princeton University Press, Princeton, NJ, 2005. %\MR{2100420 (2005j:35184)}

\bibitem{B-d}
\bysame, \emph{Positivity and continuity of the {L}yapounov exponent for shifts
  on {$\Bbb T^d$} with arbitrary frequency vector and real analytic potential},
  J. Anal. Math. \textbf{96} (2005), 313--355. %\MR{2177191 (2006i:47064)}

\bibitem{B-G-first}
J.~Bourgain and M.~Goldstein, \emph{On nonperturbative localization with
  quasi-periodic potential}, Ann. of Math. (2) \textbf{152} (2000), no.~3,
  835--879. %\MR{1815703 (2002h:39028)}

\bibitem{B-J}
J.~Bourgain and S.~Jitomirskaya, \emph{Continuity of the {L}yapunov exponent
  for quasiperiodic operators with analytic potential}, J. Statist. Phys.
  \textbf{108} (2002), no.~5-6, 1203--1218, Dedicated to David Ruelle and Yasha
  Sinai on the occasion of their 65th birthdays. %\MR{1933451 (2004c:47073)}

\bibitem{C-Stolz}
Jacob Chapman and G{\"u}nter Stolz, \emph{Localization for random block
  operators related to the {XY} spin chain}, Ann. Henri Poincar\'e \textbf{16}
  (2015), no.~2, 405--435. %\MR{3302603}

\bibitem{Craig-Simon-IDS}
Walter Craig and Barry Simon, \emph{Log {H}\"older continuity of the integrated
  density of states for stochastic {J}acobi matrices}, Comm. Math. Phys.
  \textbf{90} (1983), no.~2, 207--218. %\MR{714434}

\bibitem{David-survey}
D~Damanik, \emph{Schr\"{o}dinger operators with dynamically defined potentials:
  a survey}, preprint (2015), 1--80, to appear in Ergodic Theory and Dynamical
  Systems.

\bibitem{DK1}
Pedro Duarte and Silvius Klein, \emph{Positive {L}yapunov exponents for higher
  dimensional quasiperiodic cocycles}, Comm. Math. Phys. \textbf{332} (2014),
  no.~1, 189--219.

\bibitem{DK-survey}
\bysame, \emph{Large deviation type estimates for iterates of linear cocycles},
  Stoch. Dyn. \textbf{16} (2016), no.~3, 1660010, 54. %\MR{3471363}

\bibitem{DK-book}
\bysame, \emph{Lyapunov exponents of linear cocycles}, Atlantis Studies in
  Dynamical Systems, vol.~3, Atlantis Press, Paris, 2016, Continuity via large
  deviations.

\bibitem{Duren}
Peter~L. Duren, \emph{Theory of {$H^{p}$} spaces}, Pure and Applied
  Mathematics, Vol. 38, Academic Press, New York-London, 1970. %\MR{0268655}

\bibitem{Gol-Margulis-sim}
I.~Ya. Gol{\cprime}dshe{\u\i}d and G.~A. Margulis, \emph{Lyapunov exponents of
  a product of random matrices}, Uspekhi Mat. Nauk \textbf{44} (1989),
  no.~5(269), 13--60. %\MR{1040268}

\bibitem{GS-Holder}
Michael Goldstein and Wilhelm Schlag, \emph{H\"older continuity of the
  integrated density of states for quasi-periodic {S}chr\"odinger equations and
  averages of shifts of subharmonic functions}, Ann. of Math. (2) \textbf{154}
  (2001), no.~1, 155--203. %\MR{1847592 (2002h:82055)}

\bibitem{GS-fine}
\bysame, \emph{Fine properties of the integrated density of states and a
  quantitative separation property of the {D}irichlet eigenvalues}, Geom.
  Funct. Anal. \textbf{18} (2008), no.~3, 755--869. %\MR{2438997 (2010h:47063)}

\bibitem{Guivarch-Raugi-sim}
Y.~Guivarc'h and A.~Raugi, \emph{Products of random matrices: convergence
  theorems}, Random matrices and their applications ({B}runswick, {M}aine,
  1984), Contemp. Math., vol.~50, Amer. Math. Soc., Providence, RI, 1986,
  pp.~31--54. %\MR{841080}

\bibitem{Han-Marx}
Rui Han and C.~A. Marx, \emph{Large coupling asymptotics for the lyapunov
  exponent of quasi-periodic schrödinger operators with analytic potentials},
  preprint, 2016.

\bibitem{Herman}
Michael-R. Herman, \emph{Une m\'ethode pour minorer les exposants de
  {L}yapounov et quelques exemples montrant le caract\`ere local d'un
  th\'eor\`eme d'{A}rnol\cprime d et de {M}oser sur le tore de dimension
  {$2$}}, Comment. Math. Helv. \textbf{58} (1983), no.~3, 453--502. %\MR{727713}

\bibitem{JitMarx-CMP}
S.~Jitomirskaya and C.~A. Marx, \emph{Analytic quasi-perodic cocycles with
  singularities and the {L}yapunov exponent of extended {H}arper's model},
  Comm. Math. Phys. \textbf{316} (2012), no.~1, 237--267. %\MR{2989459}

\bibitem{JM-survey}
\bysame, \emph{Dynamics and spectral theory of quasi-periodic
  {S}chr\"odinger-type operators}, preprint (2015), 1--44, to appear in Ergodic
  Theory and Dynamical Systems.

\bibitem{JitMavi}
Svetlana Jitomirskaya and Rajinder Mavi, \emph{Continuity of the measure of the
  spectrum for quasiperiodic {S}chr\"odinger operators with rough potentials},
  Comm. Math. Phys. \textbf{325} (2014), no.~2, 585--601.

\bibitem{sK2}
Silvius Klein, \emph{Localization for quasiperiodic {S}chr\"{o}dinger operators
  with multivariable {G}evrey potential functions}, J. Spectr. Theory
  \textbf{4} (2014), 1--53.

\bibitem{Levin}
B.~Ya. Levin, \emph{Lectures on entire functions}, Translations of Mathematical
  Monographs, vol. 150, American Mathematical Society, Providence, RI, 1996, In
  collaboration with and with a preface by Yu. Lyubarskii, M. Sodin and V.
  Tkachenko, Translated from the Russian manuscript by Tkachenko.

\bibitem{Sadel-Xu}
Christian Sadel and Disheng Xu, \emph{Singular analytic linear cocycles with
  negative infinite {L}yapunov exponents}, preprint (2016), 1--16.

\bibitem{Sorets-Spencer}
Eugene Sorets and Thomas Spencer, \emph{Positive {L}yapunov exponents for
  {S}chr\"odinger operators with quasi-periodic potentials}, Comm. Math. Phys.
  \textbf{142} (1991), no.~3, 543--566.

\bibitem{Zhang-positivity}
Zhenghe Zhang, \emph{Positive {L}yapunov exponents for quasiperiodic {S}zeg\"o
  cocycles}, Nonlinearity \textbf{25} (2012), no.~6, 1771--1797.

\end{thebibliography}
\end{document}